\documentclass[11pt]{article}

\usepackage[english]{babel}
\usepackage{amssymb}
\usepackage{amsmath}
\usepackage{hyperref}
\usepackage{enumerate}
\usepackage{amsthm}
\usepackage{amsfonts}
\usepackage{listings}
\usepackage[curve]{xypic}

\usepackage{MnSymbol}

\DeclareMathSymbol{\mlq}{\mathord}{operators}{``}
\DeclareMathSymbol{\mrq}{\mathord}{operators}{`'}

\title{Derivator Six-Functor-Formalisms --- Definition and Construction I}
\date{June 7, 2021}
\author{Fritz H\"ormann\\ Mathematisches Institut, Albert-Ludwigs-Universit\"at Freiburg}

\usepackage[numbers,sort&compress]{natbib}

\usepackage{color}

\usepackage{tikz}
\newcommand*\numcirc[1]{\tikz[baseline=(char.base)]{
            \node[shape=circle,draw,inner sep=2pt] (char) {#1};}}

\newtheorem{SATZ}{Theorem}[section]
\newtheorem{HAUPTSATZ}[SATZ]{Main Theorem}
\newtheorem{LEMMA}[SATZ]{Lemma}
\newtheorem{DEF}[SATZ]{Definition}
\newtheorem{PROP}[SATZ]{Proposition}

\newtheorem{BEISPIEL}[SATZ]{Example}

\newtheorem{KOR}[SATZ]{Corollary}

\newtheorem{BEM}[SATZ]{Remark}

\newtheoremstyle{bare}        
  {}            
  {}            
  {\normalfont}                 
  {}                            
  {\bfseries}                   
  {}                            
  {.0em}                           
  {\thmnumber{#2}#1{\thmnote{ \normalfont(#3)}}. } 

\theoremstyle{bare}
\newtheorem{PAR}[SATZ]{}

\newcommand{\comment}[1]{}
\newcommand{\commentempty}[1]{}

\newcommand{\iso}{\stackrel{\sim}{\longrightarrow}}
\newcommand{\Iiso}{\stackrel{\sim}{\Longrightarrow}}

\newcommand{\DD}{ \mathbb{D} }
\newcommand{\EE}{ \mathbb{E} }

\newcommand{\SSS}{ \mathbb{S} }

\DeclareMathOperator{\id}{id}

\DeclareMathOperator{\op}{op}
\DeclareMathOperator{\cart}{cart}
\DeclareMathOperator{\cocart}{cocart}

\DeclareMathOperator{\Hom}{Hom}
\DeclareMathOperator{\Fun}{Fun}
\DeclareMathOperator{\Span}{Cor}

\DeclareMathOperator{\Mor}{Mor}

\DeclareMathOperator{\spec}{spec}
\DeclareMathOperator{\pr}{pr}

\DeclareMathOperator{\cor}{cor}
\DeclareMathOperator{\lax}{lax}
\DeclareMathOperator{\oplax}{oplax}

\DeclareMathOperator{\Dia}{Dia}

\DeclareMathOperator{\Posf}{Posf}

\DeclareMathOperator{\dia}{dia}

\newcommand{\tw}[1]{ {{}^{\downarrow \uparrow} #1 }} 
\newcommand{\twc}[1]{ {{}^{\downarrow \uparrow \downarrow} #1 }}

\begin{document}

\maketitle



\section*{Abstract}

A theory of a {\em derivator version} of six-functor-formalisms is developed, using an extension of the notion of fibered multiderivator due to the author. 
The definition is very compact using the language of (op)fibrations of 2-multicategories. It not only encodes
all compatibilities among the six functors but also their interplay with homotopy Kan extensions. One could say: a nine-functor-formalism. This is essential, for instance, to deal with (co)descent questions of six-functor-formalisms. Finally, it is shown that every fibered multiderivator (for example encoding any kind of derived four-functor formalism $(f_*, f^*, \otimes, \mathcal{HOM})$ occurring in nature) satisfying base-change and projection formula
{\em formally} gives rise to such a derivator six-functor-formalism in which ``$f_!=f_*$'', i.e.\@ a derivator Grothendieck context. 

\tableofcontents

\section{Introduction}

A formalism of the ``six functors'' lies at the core of many different theories in mathematics, as for example the theory of Abelian sheaves on topological spaces, etale, $l$-adic, or coherent sheaves on schemes, D-modules, representations of (pro\nobreakdash-)finite groups, motives, and many more.
Given a base category of ``spaces'' $\mathcal{S}$, for instance, the category of schemes, topological spaces, analytic manifolds, etc. such 
a formalism roughly consists of a collection of (derived) categories $\mathcal{D}_S$ of ``sheaves'', one for each ``base space'' $S$ in $\mathcal{S}$, and
 the following six types of functors between those categories:
\vspace{0.5cm}
\[
\begin{array}{lcrp{1cm}l}
f^* & &f_* &  & \text{\em for each $f$ in $\Mor(\mathcal{S})$}  \\
\\
f_! & &f^! & & \text{\em for each $f$ in $\Mor(\mathcal{S})$} \\
\\
\otimes & & \mathcal{HOM} & & \text{\em in each fiber $\mathcal{D}_S$}
\end{array}
\]
\vspace{0.3cm}

The functors on the left hand side are left adjoints of the functors on the right hand side. The functor $f_!$ is called {\bf push-forward with proper support}, because in the topological setting (Abelian sheaves over topological spaces) this is what it is derived from. Its right adjoint $f^!$ is called the {\bf exceptional pull-back}. These functors come along with a bunch of compatibilities between them. 

\begin{PAR}\label{COMP6FU}
More precisely, part of the datum of the six functors are the following natural isomorphisms in the ``left adjoints'' column:
\vspace{0.3cm}
\begin{center}
\begin{tabular}{r|lll}
& isomorphisms  & isomorphisms  \\
& between left adjoints & between right adjoints \\
\hline
$(*,*)$ & $(fg)^* \iso g^* f^*$ & $(fg)_* \iso f_* g_*$ &\\
$(!,!)$ & $(fg)_! \iso f_! g_!$ & $(fg)^! \iso g^! f^!$ &\\ 
$(!,*)$ 
& $g^* f_! \iso F_! G^*$ & $G_* F^! \iso f^! g_*$ & \\
$(\otimes,*)$ & $f^*(- \otimes -) \iso f^*- \otimes f^* -$ & $f_* \mathcal{HOM}(f^*-, -) \iso \mathcal{HOM}(-, f_*-)$  & \\
$(\otimes,!)$ & $f_!(- \otimes f^* -) \iso  (f_! -) \otimes -$ & $f_* \mathcal{HOM}(-, f^!-) \iso \mathcal{HOM}(f_! -, -)$ & \\ 
& & $f^!\mathcal{HOM}(-, -) \iso \mathcal{HOM}(f^* -, f^!-)$ & \\
$(\otimes, \otimes)$ &  $(- \otimes -) \otimes - \iso - \otimes (- \otimes -)$ &  $\mathcal{HOM}(- \otimes -, -) \iso \mathcal{HOM}(-, \mathcal{HOM}(-, -))$ & 
\end{tabular}
\end{center}
\vspace{0.3cm}
\end{PAR}

Here $f, g, F, G$ are morphisms in $\mathcal{S}$, which in the $(!,*)$-row, are related by a {\em Cartesian} diagram:
\[ \xymatrix{ \cdot \ar[r]^G \ar[d]_F  & \cdot \ar[d]^f \\ \cdot \ar[r]_g & \cdot } \]

In the right hand side column the corresponding adjoint natural transformations have been inserted. In each case the left hand side natural isomorphism uniquely determines the right hand side one and vice versa. (In the $(\otimes,!)$-case there are two versions of the 
commutation between the right adjoints; in this case any of the three isomorphisms determines the other two).
The $(!,*)$-isomorphism (between left adjoints) is called {\bf base change}, the $(\otimes, !)$-isomorphism is called the {\bf projection formula}, and the $(*, \otimes)$-isomorphism is usually part of the definition of a {\bf monoidal functor}. The $(\otimes, \otimes)$-isomorphism is the associativity of the tensor product and part of the definition of a monoidal category. The $(*,*)$-isomorphism, and the $(!,!)$-isomorphism, express that the corresponding functors arrange as a pseudo-functor with values in categories. 

Furthermore part of the datum are isomorphisms
\[ f^* \iso f^! \]
for all isomorphisms $f$\footnote{There are more general formalisms, which we call proper or etale six-functor-formalisms where there is a 
morphism $f^* \rightarrow f^!$ or a morphism $f_! \rightarrow f_*$ for {\em certain} morphisms $f$ (cf.\@ \cite[Section~8]{Hor15b})}.
Of course, there have to be compatibilities among those natural isomorphisms, and it is not easy to give a complete list of them.
In \cite{Hor15b} we explained how to proceed in a more abstract way (like in the ideas of fibered category or multicategory) and get a {\em precise} definition of a {\bf six-functor-formalism} without having to specify any of these compatibilities explicitly. The natural isomorphisms of \ref{COMP6FU} are derived from a composition law in a 2-multicategory $\mathcal{S}^{\cor}$ and all compatibilities are a consequence of the associativity of this composition law. 

Six-functor-formalisms were first introduced by Grothendieck, Verdier and Deligne \cite{Ver77, SGAIV1, SGAIV2, SGAIV3} and there 
has been continuously increasing interest in them in various contexts in the last decade \cite{FHM, Ayo07I, Ayo07II, LO08I, LO08II, CD09, LH09, ZL14, Zh10, Sch15, GR16, Dre18, DG20}.

The six functors are the right framework to study duality theorems like Serre duality, Poincar\'e duality, various (Tate) dualities for the (co)homology of groups, etc.
\begin{BEISPIEL}[Serre duality] Let $k$ be a field. If $\mathcal{S}$ is the category of $k$-varieties, we have a six-functor-formalism in which $\mathcal{D}_S$ is the derived category of (quasi-)coherent sheaves\footnote{Neglecting here for a moment the fact that $f_!$ exists in general only after passing to pro-coherent sheaves.} on $S$.
Let $\pi: S \rightarrow \spec(k)$ be a proper and smooth $k$-scheme of dimension $n$. 
Consider a locally free sheaf $\mathcal{E}$ on $S$ and
consider the following isomorphism (one of the two adjoints of the projection formula):
\[ \pi_* \mathcal{HOM}(\mathcal{E}, \pi^! k) \iso \mathcal{HOM}(\pi_! \mathcal{E}, k) \]
In this case, we have $\pi_! \mathcal{E} = \pi_* \mathcal{E}$ because $\pi$ is proper, and $\pi^! k = \Omega_S^n[n]$. Taking $i$-th homology of complexes we obtain 
\[ H^{i+n}(S, \mathcal{E}^\vee \otimes \Omega_S^n) \cong H^{-i}(S, \mathcal{E})^*.  \]
This is the classical formula of Serre duality. 
\end{BEISPIEL}

\begin{BEISPIEL}[Poincar\'e duality]\label{EXPOINCAREINTRO}
Let $k$ be a field. If $\mathcal{S}$ is a category of nice topological spaces, we have a six-functor-formalism in which $\mathcal{D}_S$ is the derived category of sheaves of $k$-vector spaces on $S$.
Let $X$ be an $n$-dimensional topological manifold. Consider a local system $\mathcal{E}$ of $k$-vector spaces on $X$ and
consider the isomorphism (again one of the two adjoints of the projection formula):
\[ \pi_* \mathcal{HOM}(\mathcal{E}, \pi^! k) \iso \mathcal{HOM}(\pi_! \mathcal{E},  k) \]

We have $\pi^! k = \mathcal{L}_{or}[n]$, where $\mathcal{L}_{or}$ is the orientation sheaf  of $X$ over $k$. 
Taking $i$-th homology of complexes we arrive at 
\[ H^{i+n}(X, \mathcal{E}^\vee \otimes \mathcal{L}_{or}) \cong H^{-i}_c(X, \mathcal{E})^*.  \]
This is the classical formula of Poincar\'e duality. 
\end{BEISPIEL}

\begin{BEISPIEL}[Group (co)homology] The six-functor-formalism of Example~\ref{EXPOINCAREINTRO} extends to stacks. 
Let $G$ be a group and consider the classifying stack $[\cdot/G]$ and the projection $\pi: [\cdot/G] \rightarrow \cdot$. Note: Abelian sheaves on $[\cdot/G]$ = $G$-representations in Abelian groups. The extension of the six-functor-formalism encodes duality theorems like Tate duality. In this case $\pi_*$ yields group cohomology and $\pi_!$ yields group homology.  If $G$ is finite, we also have a natural morphism $\pi_! \rightarrow \pi_*$ whose cone (homotopy cokernel) is Tate cohomology.  
\end{BEISPIEL}

We now recall the precise definition of six-functor-formalisms from \cite{Hor15b}. 
Let $\mathcal{S}$ be a (base) category with fiber products as above. 
The symmetric 2-multicategory $\mathcal{S}^{\cor}$ has the same objects as $\mathcal{S}$, the 1-multimorphisms
$S_1, \dots, S_n \rightarrow T$ are the multicorrespondences
\begin{equation}\label{excor} \vcenter{ \xymatrix{
&&&  A \ar[rd]^f \ar[ld]^{g_n} \ar[llld]_{g_1} \\
S_1 & \cdots & S_n  & ; &  T
} } 
\end{equation}
and the 2-morphisms are isomorphisms between these multicorrespondences\footnote{Later more general definitions of 2-multicategories $\mathcal{S}^{\cor,0}$, and $\mathcal{S}^{\cor,G}$, where 2-morphisms can be more general morphisms between multicorrespondences will become important.}. The composition is given by forming fiber products. 
In \cite{Hor15b}, we explained the following formal definition of a six-functor-formalism.

\vspace{0.2cm}
{\bf Definition. }{\em 
A symmetric six-functor-formalism on $\mathcal{S}$ is a 1-bifibration and 2-bifibration of symmetric 2-multicategories with 1-categorical fibers}
\[ p: \mathcal{D} \rightarrow \mathcal{S}^{\cor}. \]
Such a fibration can also be seen as a pseudo-functor of 2-multicategories
\[ \mathcal{S}^{\cor} \rightarrow \mathcal{CAT} \]
with the property that all multivalued functors in the image have right adjoints w.r.t.\@ all slots. 
Note that $\mathcal{CAT}$, the ``category''\footnote{having, of course, a {\em higher class} of objects.} of categories, has naturally the structure of a 2-``multicategory'' where the 1-multimorphisms are functors of several variables. 

This pseudo-functor maps the correspondence (\ref{excor}) to a functor isomorphic to 
\[ f_! ((g_1^* -) \otimes_A \cdots \otimes_A (g_n^* -)) \]
where $\otimes_A$, $f_!$, and $g^*_i$, are the images of  the following correspondences
\[
 \vcenter{ \xymatrix{
&&&  A \ar@{=}[rd] \ar@{=}[ld] \ar@{=}[llld] \\
A &  & A  & &  A
} } \quad
 \vcenter{ \xymatrix{
& A \ar[rd]^f \ar@{=}[ld]  \\
A   & &  T
} } \quad
 \vcenter{ \xymatrix{
&  A \ar@{=}[rd] \ar[ld]_{g_i}  \\
S_i  & &  A
} }
\]
As was explained in \cite{Hor15b}, the definition of six-functor-formalisms using $\mathcal{S}^{\cor}$ has the advantage that all the 6 types of isomorphisms between those functors
as in \ref{COMP6FU} and all
compatibilities between those isomorphisms (akward to write down a complete list) are already encoded in this simple definition. 

\subsection*{Definition of derivator six-functor-formalisms}

In most cases occurring in nature, the values of the six-functor-formalism, i.e.\@ the fibers of the fibration $\mathcal{D} \rightarrow \mathcal{S}^{\cor}$,
are {\em derived categories}. It is therefore natural to seek to enhance the situation to a {\em fibered multiderivator}. We will not give a detailled account on fibered multiderivators here. In the stable case they enhance fibrations of triangulated monoidal categories in a similar way that a usual derivator enhances triangulated categories. See \cite{Hor15} for a detailed introduction. 

Enhancements of this sort are essential to deal with (co)descent questions. The notion of a fibered multiderivator given in \cite{Hor15}, however, is not sufficient because $\mathcal{S}^{\cor}$ is a 2-multicategory (as opposed to a usual multicategory). Although the 2-multicategory $\mathcal{S}^{\cor}$ gives rise to a usual (not-represented) pre-multiderivator, by identifying 2-isomorphic 1-morphisms in the diagram categories $\mathcal{S}^{\cor}(I)$, a fibered multiderivator over {\em that} pre-multiderivator
 would not encode what we want\footnote{E.g.\@ the push-forward along a correspondence of the form $\{\cdot\} \leftarrow X \rightarrow \{\cdot\}$ should be something like the cohomology with compact support of $X$ with constant coefficients. Identifying 2-isomorphic 1-morphisms in $\mathcal{S}^{\cor}$ would force this to become the invariant part (in a derived sense) under automorphisms of $X$.}. It turns out that
the theory of fibered multiderivators over pre-multiderivators has a straightforward extension to pre-2-multiderivators in which the knowledge of the 2-morphisms of the base is preserved. 

In the first half of this article, we thus develop the theory of fibered multiderivators over pre-2-multiderivators. This allows to consider the symmetric pre-2-multiderivator $\SSS^{\cor}$ represented by the symmetric 2-multicategory $\mathcal{S}^{\cor}$ and to define:

\vspace{0.2cm}
{\bf Definition \ref{DEF6FUDER}. }{\em 
A {\bf (symmetric) derivator six-functor-formalism} is a left and right fibered (symmetric) multiderivator
\[ \DD \rightarrow \SSS^{\cor}. \]}%
It becomes important to have notions of (symmetric) {\em (op)lax}\, fibered multiderivators as well. Those are useful to enhance to derivators the definition of a
proper or etale six-functor-formalism which arises, for instance, whenever for some class of morphisms one has isomorphisms $f^! \cong f^*$ or $f_! \cong f_*$ which are part of the formalism. If this is the case for all morphisms, one speaks respectively of a Wirthm\"uller, or Grothendieck context. 

The second half of the article is devoted to the {\em construction} of derivator six-functor-formalisms. For this we concentrate on the case in which $f_!=f_*$ for all morphisms $f$ in $\mathcal{S}$, i.e.\@ to Grothendieck contexts. The case in which $f_! \not= f_*$ is much more involved and will be discussed in a subsequent article \cite{Hor17}.
In the classical case this construction is almost tautological: 
\begin{enumerate}
\item One starts with a four-functor-formalism ($f_*, f^*, \otimes, \mathcal{HOM}$) encoded by a bifibration of usual (symmetric) multicategories $\mathcal{D} \rightarrow \mathcal{S}^{\op}$ (where $\mathcal{S}^{\op}$ becomes a multicategory via the product), or equivalently by a pseudo-functor of (2-)multicategories $\mathcal{S}^{\op} \rightarrow \mathcal{CAT}$. Then one simply defines  a pseudo-functor
$\mathcal{S}^{\cor} \rightarrow \mathcal{CAT}$ by mapping a multicorrespondence (\ref{excor}) to the functor
\[ f_* ((g_1^* -) \otimes \cdots \otimes (g_n^* -)). \]
It is straightforward (but slightly tedious) to check that this defines a pseudo-functor if and only if base-change and projection formula hold \cite[Proposition~3.13]{Hor15b}. 
\item In the derived world, using theorems on Brown representability, one gets formally that the $f_*$ functors have right adjoints $f^!$ (provided that $f_*$ commutes with infinite coproducts as well), hence the 1-opfibration (and 2-bifibration) with 1-categorical fibers $\mathcal{E} \rightarrow \mathcal{S}^{\cor}$, which corresponds to the pseudo-functor in 1.\@, is also a 1-fibration.
\end{enumerate}
 
It is surprising, however, that constructions 1.\@ and 2.\@ are also possible in the world of fibered multiderivators, although they become more involved. The first is however still completely formal. It turns out that one can relax the condition that $\mathcal{S}^{\op}$ is a multicategory coming from a usual category $\mathcal{S}$ via the categorical product in $\mathcal{S}$. One can start with any opmulticategory $\mathcal{S}$. The definition of $\mathcal{S}^{\cor}$ generalizes readily to this situation. 

\subsection*{Construction of derivator six-functor-formalisms}

Let $\mathcal{S}$ be an opmulticategory with multipullbacks. As before, $\mathcal{S}$ will mostly be a usual category with the 
structure of opmulticategory given by the product, i.e. 
\begin{equation}\label{opmulti}
\Hom_{\mathcal{S}}(Y;  X_1, \dots, X_n) = \Hom_{\mathcal{S}}(Y, X_1) \times \cdots \times \Hom_{\mathcal{S}}(Y, X_n).  
\end{equation}
This structure is canonically symmetric.
However, for the sequel $\mathcal{S}$ may be arbitrary. It also does not need to be representable (i.e.\@ opfibered over $\{ \cdot \}$, i.e.\@ monoidal)
and may be equipped with the structure of symmetric or braided opmulticategory. 
In that case all other multicategories and 2-multicategories occurring, e.g.\@ $\mathcal{S}^{\cor}$, will also be symmetric, resp.\@ braided, and all functors have to be compatible with the corresponding actions. 

\begin{DEF}\label{DEFMULTIBASECHANGE}Let $\mathcal{S}$ be an opmulticategory with multipullbacks. 
Let $\mathcal{D} \rightarrow \mathcal{S}^{\op}$ be a bifibration of usual multicategories. We say that it satisfies
{\bf multi-base-change}, if for every multipullback in $\mathcal{S}$
\[ \xymatrix{
X_1, \dots, X_i, \dots,  X_n  & \ar[l]_-{g}  Z\\
X_1, \dots, X_i', \dots,  X_n \ar[u]^{(\id, \dots, f, \dots, \id)}  & Z' \ar[l]^-G  \ar[u]_F 
} \]
the natural transformation
\[ g_\bullet(-, \dots, \underbrace{f^\bullet-}_{\text{at }i}, \dots, -) \longrightarrow F^\bullet G_\bullet(-, \dots, -)  \]
is an isomorphism. In the case that $\mathcal{S}$ is a usual category equipped with the opmulticategory structure (\ref{opmulti}) this encodes {\em projection formula} and {\em base change}. 
\end{DEF}

In this definition, $f^\bullet$ denotes the pull-back along $f^{\op}$ in $\mathcal{S}^{\op}$, that is, the usual push-forward $f_*$ along $f$ in $\mathcal{S}$. The reason for this notation is that we stick to the convention that $f^\bullet$ is always right adjoint to $f_\bullet$ and, at the same time, we want to avoid the notation $f_*, f^*, f^!, f_!$ because of the possible confusion with the left and right Kan extension functors which will be denoted by $\alpha_!$, and $\alpha_*$, respectively.

\vspace{0.2cm}
{\bf Theorem \ref{MAINTHEOREM1}}. 
{\em Let $\mathcal{S}$ be a (symmetric) opmulticategory with multipullbacks and let $\SSS^{\op}$ be the (symmetric) pre-multiderivator  represented by $\mathcal{S}^{\op}$ . 
Let $\DD \rightarrow \SSS^{\op}$ be a (symmetric) left and right fibered multiderivator such that the following holds:
\begin{enumerate}
\item 
The pullback along 1-ary morphisms (i.e.\@ pushforward along 1-ary morphisms in $\mathcal{S}$) commutes also with homotopy {\em co}limits (of shape in $\Dia$).
\item 
In the underlying bifibration $\DD(\cdot) \rightarrow \SSS(\cdot)$ multi-base-change holds in the sense of Definition~\ref{DEFMULTIBASECHANGE}. 
\end{enumerate}
Then there exists a (symmetric) oplax left fibered multiderivator
\[ \EE \rightarrow \SSS^{\cor,G,\oplax}  \]
satisfying the following properties
\begin{itemize}
\item[a)] 
The corresponding (symmetric) 1-opfibration, and 2-opfibration of 2-multicategories with 1-categorical fibers
\[ \EE(\cdot) \rightarrow \SSS^{\cor,G,\oplax}(\cdot) = \mathcal{S}^{\cor,G}  \]
is just (up to equivalence) obtained from $\DD(\cdot) \rightarrow \mathcal{S}^{\op}$ by the procedure described in \cite[Definition~3.12]{Hor15b}.
\item[b)] For every $S \in \mathcal{S}$ there is a canonical equivalence of fibers: 
\[ \EE_S \cong \DD_S  \]
and if those are stable, {\em all} fibers of $\EE \rightarrow \SSS^{\cor,G}$ are stable derivators with domain $\Posf$ (in particular also right derivators). 
\end{itemize} }

\vspace{0.2cm}

Using standard theorems on Brown representability etc.\@ \cite[Section 3.1]{Hor15}  we can refine this.

\vspace{0.2cm}
{\bf Theorem \ref{MAINTHEOREM2}}. 
{\em Let $\Dia$ be an infinite diagram category \cite[Definition 1.1.1]{Hor15} which contains all finite posets. 
Let $\mathcal{S}$ be a (symmetric) opmulticategory with multipullbacks and let $\SSS$ be the corresponding represented (symmetric) pre-multiderivator with domain $\Dia$. 
Let $\DD \rightarrow \SSS^{\op}$ be an {\em infinite} (symmetric) left and right fibered multiderivator with domain $\Dia$ satisfying conditions 1.\@ and 2.\@ of Theorem~\ref{MAINTHEOREM1}, {\em with stable, perfectly generated fibers} (cf.\@ Definition~\ref{DEFSTABLE} and \cite[Section 3.1]{Hor15}).

Then the restriction of the left fibered multiderivator $\EE$ from Theorem~\ref{MAINTHEOREM1} is a (symmetric) left {\em and right} fibered multiderivator with domain $\Dia$
\[ \EE|_{\SSS^{\cor}} \rightarrow \SSS^{\cor} \]
and has an extension as a (symmetric) {\em lax} right fibered multiderivator with domain $\Dia$
\[ \EE' \rightarrow \SSS^{\cor,G,\lax}. \] }

\vspace{0.2cm}

We call $\EE|_{\SSS^{\cor}}$ 
together with the extensions to $\SSS^{\cor,G,\lax}$ and $\SSS^{\cor,G,\oplax}$, respectively, a {\bf derivator Grothendieck context}, cf.\@ Definition~\ref{DEF6FUDER}.

The construction in \cite[Definition~3.12]{Hor15b} recalled above is quite tautological. Why is Theorem~\ref{MAINTHEOREM1} not similarly tautological? To understand this point, let us look at the following simple example:
A fibered derivator $\DD$ over $\Delta_1$, the (represented pre-derivator of the) usual category with one arrow, encodes an enhancement of an adjunction between derivators (the two fibers of $\DD$). Think about the case, where this is the derived adjunction coming from an adjunction of underived functors $f_*, f^*$. 
This includes, for instance, as fiber of $\DD(\Delta_1)$ over the identity in $\Fun(\Delta_1, \Delta_1)$, the category of {\em coherent} diagrams of the form
$X \rightarrow f_* Y$, or equivalently $f^* X \rightarrow Y$, up to quasi-isomorphisms between such diagrams. Here $f_*, f^*$ are the {\em underived} functors. 
The extension $\EE$ in the theorem allows to consider {\em coherent} diagrams of the form
$f_* X \rightarrow Y$, or equivalently $X \rightarrow f^! Y$, if $f_*$ has a right adjoint $f^!$, the point being that, however, $f^!$ may not exist before passing to the derived categories. Nevertheless, we are now allowed to speak about ``coherent diagrams of the form $X \rightarrow f^! Y$'' although this does not make literally sense. 
In particular, the theorem yields a {\em coherent enhancement} in $\DD(\Delta_1)_{p^* e_i}$ for $i=0$, and $i=1$, respectively, of the unit and counit 
\[  Rf_* f^! \mathcal{E} \rightarrow \mathcal{E} \qquad  \mathcal{E} \rightarrow f^! Rf_*   \mathcal{E}  \]
where $f^!$ now denotes a right adjoint of the derived functor $Rf_*$.

In this particular case, the {\em construction} boils down to the following.
The fiber of $\EE(\Delta_1)$ over the correspondence
\[ \xymatrix{
& e_0 \ar[ld] \ar@{=}[rd] \\
e_1 & & e_0
} \]
will consist of coherent diagrams of the form $f_*X \leftarrow Z \rightarrow Y$ {\em in the original fibered derivator $\DD$} with the property that the induced morphism $R f_*X \leftarrow Z$ is a 
quasi-isomorphism, i.e.\@ the morphism $X \leftarrow Z$ in $\DD(\Delta_1)$ becomes {\em coCartesian}, when considered as a morphism in $\DD(\cdot)$. The purpose of the second part of this article is thus to make this construction work in the full generality of Theorem~\ref{MAINTHEOREM1}. Although the idea is still very simple, the construction of $\EE$ (cf.\@ Definiton~\ref{DEFE}) and the proof that it really is a (left) fibered multiderivator over $\SSS^{\cor}$, becomes quite technical.

\subsection*{Localization triangles}

As an application of the general definition of derivator six-functor-formalisms, in Section~\ref{SECTIONLOC} we explain that the appearance of distinguished triangles like 
\[ \xymatrix{ j_! j^! \mathcal{E} \ar[r] &  \mathcal{E} \ar[r] & \overline{j}_* \overline{j}^* \mathcal{E} \ar[r]^-{[1]} &   } \]
for an ``open immersion'' $j$ and its complementary ``closed immersion'' $\overline{j}$ can be treated elegantly. 
Actually there are four flavours of these sequences, two for
proper derivator six-functor-formalisms, and two for etale derivator six-functor-formalisms. More generally, a sequence of ``open embeddings''
\[ \xymatrix{ X_1 \ar@{^{(}->}[r] & X_2 \ar@{^{(}->}[r] & \cdots \ar@{^{(}->}[r] & X_n  } \]
leads immediately to so called $(n+1)$-angles in the sense of \cite[\S 13]{GS14b} in the fiber over $X_n$ which is a usual stable derivator.

\section{Pre-2-multiderivators}\label{SECTION2PREMULTIDER}

We fix a diagram category $\Dia$ \cite[Definition 1.1.1]{Hor15} once and for all. 
From Section~\ref{CONSTRUCTION} on, we assume that in $\Dia$, in addition to the axioms of \cite{Hor15}, the construction of the diagrams ${}^\Xi I$ of \ref{PARTW} is permitted for any $I \in \Dia$. 
If one wants to specify $\Dia$, one would speak about e.g.\@ pre-2-multiderivators, or fibered multiderivators, {\em with domain $\Dia$}. For better readability we omit this. This is justified because all arguments of this article are completely formal, not depending on the choice of $\Dia$ at all. An exception is Theorem~\ref{MAINTHEOREM2} where
Brown representability type results are applied. 

\begin{DEF}\label{DEF2PREMULTIDER}
A {\bf pre-2-multiderivator} is a functor $\SSS: \Dia^{1-\op} \rightarrow \text{2-$\mathcal{MCAT}$}$ which is strict in 1-morphisms (functors) and pseudo-functorial in 2-morphisms (natural transformations)\footnote{Here 2-$\mathcal{MCAT}$ denotes the 3-category of 2-multicategories. We will not refer to any 3-categorical language in the sequel.}. 
More precisely, it associates with a diagram $I$ a 2-multicategory $\SSS(I)$, with a functor $\alpha: I \rightarrow J$ a strict functor
\[ \SSS(\alpha):  \SSS(J) \rightarrow \SSS(I) \] 
denoted also $\alpha^*$, if $\SSS$ is understood, and with a natural transformation $\mu: \alpha \Rightarrow \alpha'$ a pseudo-natural transformation
\[ \SSS(\eta): \alpha^* \Rightarrow (\alpha')^* \]
such that the following holds:
\begin{enumerate}
\item The association
\[ \Fun(I, J) \rightarrow \Fun(\SSS(J), \SSS(I)) \]
given by $\alpha \mapsto \alpha^*$ on 1-morphisms, resp.\@ $\mu \mapsto \SSS(\mu)$ on 2-morphisms, is
a pseudo-functor (note that this involves the choice of further data). Here $\Fun(\SSS(J), \SSS(I))$ is the 2-category of strict 2-functors, pseudo-natural transformations, and modifications. 
\item (Strict functoriality w.r.t.\@ compositons of 1-morphisms) For functors $\alpha: I \rightarrow J$ and $\beta: J \rightarrow K$, we have 
an {\em equality} of pseudo-functors $\Fun(I, J) \rightarrow \Fun(\SSS(I), \SSS(K))$
\[ \beta^* \circ \SSS(-) = \SSS(\beta \circ -).   \]
\end{enumerate}

A {\bf symmetric, resp.\@ braided pre-2-multiderivator} is given by the structure of strictly symmetric (resp.\@ braided) 2-multicategory on $\SSS(I)$ such that
the strict functors $\alpha^*$ are equivariant w.r.t.\@ the action of the symmetric groups (resp.\@ braid groups). 

Similarly we define a {\bf lax, resp.\@ oplax, pre-2-multiderivator} where the same as before holds but where the 
\[ \SSS(\eta): \alpha^* \Rightarrow (\alpha')^* \]
are lax (resp.\@ oplax) natural transformations and in 1.\@ ``pseudo-natural transformations'' is replaced by ``lax (resp.\@ oplax) natural transformations''.
\end{DEF}

\begin{DEF}\label{DEF2PREMULTIDERSTRICTMOR}
A strict morphism $p: \DD \rightarrow \SSS$ of pre-2-multiderivators (resp.\@ lax/oplax pre-2-multiderivators) is given by a collection of strict 2-functors
\[ p(I): \DD(I) \rightarrow \SSS(I) \]
for each $I \in \Dia$ such that we have $\SSS(\alpha) \circ p(J) = p(I) \circ \DD(\alpha)$ and $\SSS(\mu) \ast p(J) = p(I) \ast \DD(\mu)$ 
 for all functors $\alpha: I \rightarrow J$, $\alpha': I \rightarrow J$ and natural transformations $\mu: \alpha \Rightarrow \alpha'$ as illustrated by the following diagram:
\[ \xymatrix{
\DD(J) \ar[rr]^{p(J)} \ar@/_15pt/[dd]_{\DD(\alpha)}^{\phantom{x}\overset{\DD(\mu)}{\Rightarrow}} \ar@/^15pt/[dd]^{\DD(\alpha')} && \SSS(J) \ar@/_15pt/[dd]_{\SSS(\alpha)}^{\phantom{x}\overset{\SSS(\mu)}{\Rightarrow}} \ar@/^15pt/[dd]^{\SSS(\alpha')} \\
\\
\DD(I) \ar[rr]^{p(I)} && \SSS(I)
} \]
\end{DEF}

\begin{DEF}
Given a (lax/oplax) pre-2-derivator $\SSS$, we define
\[ \SSS^{1-\op}:  I \mapsto \SSS(I^{\op})^{1-\op}  \]
and given a (lax/oplax) pre-2-multiderivator $\SSS$, we define
\[ \SSS^{2-\op}:  I \mapsto \SSS(I)^{2-\op}  \]
reversing the arrow in the (lax/oplax) pseudo-natural transformations. I.e.\@ the second operation interchanges lax and oplax pre-2-multiderivators. 
\end{DEF}

\begin{PAR}\label{PARDER12}
As with usual pre-multiderivators we consider the following axioms: 

\begin{itemize}
\item[(Der1)] For $I, J \in \Dia$, the natural functor $\DD(I \coprod J) \rightarrow \DD(I) \times \DD(J)$ is an equivalence of 2-multicategories. Moreover $\DD(\emptyset)$ is not empty.
\item[(Der2)]
For $I \in \Dia$ the `underlying diagram' functor
\[ \dia: \DD(I) \rightarrow \Fun(I, \DD(\cdot)) \quad \text{resp. } \Fun^{\lax}(I, \DD(\cdot)) \quad \text{resp. }  \Fun^{\oplax}(I, \DD(\cdot))\]
is 2-conservative (this means that it is conservative on 2-morphisms and that a 1-morphism $\alpha$ is an equivalence if $\dia(\alpha)$ is an equivalence).
\end{itemize}
\end{PAR}

\begin{PAR}\label{PARREPR2PREMULTIDER}
Let $\mathcal{D}$ be a 2-multicategory. We define the following {\bf representable} pre-2-multiderivator:
\begin{eqnarray*}
\DD:  \Dia &\rightarrow& \text{2-$\mathcal{MCAT}$}  \\
 I &\mapsto& \Fun(I, \mathcal{D})
\end{eqnarray*}
where $\Fun(I, \mathcal{D})$ is the 2-multicategory of pseudo-functors, pseudo-natural transformations, and modifications.
This is usually considered only if all 2-morphisms in $\mathcal{D}$ are invertible. 

We define a {\bf representable} {\em lax} pre-2-multiderivator as 
\begin{eqnarray*}
\DD^{\mathrm{lax}}:  \Dia &\rightarrow& \text{2-$\mathcal{MCAT}$}  \\
 I &\mapsto& \Fun^{\mathrm{lax}}(I, \mathcal{D})
\end{eqnarray*}
where $\Fun^{\mathrm{lax}}(I, \mathcal{D})$ is the 2-multicategory of pseudo-functors, {\em lax} natural transformations, and modifications,
and similarly a representable oplax pre-2-multiderivator $\DD^{\mathrm{oplax}}$.
\end{PAR}

\section{Correspondences of diagrams in a pre-2-multiderivator}

As explained in the introduction, the first goal of this article is to extend the notion of {\em fibered multiderivator} to bases which are pre-2-multiderivators instead of pre-multiderivators. 
A definition as in \cite{Hor15}, specifying axioms (FDer0) and (FDer3--5) for a morphism $p: \DD \rightarrow \SSS$ of pre-2-multiderivators, is possible (cf.\@ Theorem~\ref{MAINTHEOREMFIBDER2}). However, as was explained in \cite{Hor15b} for usual fibered multiderivators, a much neater definition involving a certain category $\Dia^{\cor}(\SSS)$  \cite[Definition~5.7]{Hor15b} of correspondences of diagrams in $\SSS$ is possible. In this article we take this as our principal approach. We therefore extend the definition of $\Dia^{\cor}(\SSS)$ for pre-multiderivators to pre-2-multiderivators (resp.\@ to lax/oplax pre-2-multiderivators). For this, we first have to extend the definition of $\Span_{\SSS}(I_1, \dots, I_n; J)$ to pre-2-multiderivators $\SSS$. First recall the following 

\begin{DEF}[{\cite[Definition~4.3]{Hor15b}}]\label{DEFSPAN}
Let $I_1, \dots, I_n, J$ be diagrams in $\Dia$. Define $\Span(I_1, \dots, I_n; J)$ to be the following strict 2-category:
\begin{enumerate}
\item
The objects are diagrams of the form
\[ \xymatrix{ &&&A \ar[dlll]_{\alpha_1} \ar[dl]^{\alpha_n}  \ar[dr]^{\beta} \\ 
I_1 & \cdots & I_n &;& J
} \]
with $A \in \Dia$.
\item The 1-morphisms $(A, \alpha_1, \dots, \alpha_n, \beta) \Rightarrow (A', \alpha_1', \dots, \alpha_n', \beta')$ are functors $\gamma: A \rightarrow A'$ and natural transformations $\nu_1, \dots, \nu_n, \mu$:
\[ \vcenter{ \xymatrix{
A \ar[rd]_{\alpha_i} \ar[rr]^{\gamma} & \ar@{}[d]|{\Rightarrow^{\nu_i}} & A' \ar[ld]^{\alpha_i'}  \\
&I_i
} } \qquad \vcenter{ \xymatrix{
A \ar[rd]_{\beta} \ar[rr]^{\gamma} & \ar@{}[d]|{\Leftarrow^\mu}  & A' \ar[ld]^{\beta'}  \\
&J
} } \]
\item The 2-morphisms are natural transformations
$\eta: \gamma \Rightarrow \gamma'$ such that $(\alpha_i' \ast \eta) \circ \nu_i  = \nu_i'$ and $\mu' \circ (\beta' \ast \eta)  = \mu$ hold.
\end{enumerate}
\end{DEF}
We define also the full subcategory $\Span^{F}(I_1, \dots, I_n; J)$ of those objects for which $\alpha_1 \times \cdots \times \alpha_n: A \rightarrow I_1 \times \cdots \times I_n$ is a fibration and $\beta$ is an opfibration.
The $\gamma$'s do not need to be morphisms of fibrations, respectively of opfibrations.

\begin{DEF}[{\cite[Definition~4.4]{Hor15b}}]\label{DEFDIACOR}
We define the {\bf 2-multicategory of correspondences of diagrams} $\Dia^{\cor}$ as the following 2-multicategory (cf.\@ \cite[1.9]{Hor15b} for the notation $\tau_1$):
\begin{enumerate}
\item The objects are diagrams $I \in \Dia$.
\item For every $I_1, \dots, I_n, J$ diagrams in $\Dia$, the category $\Hom_{\Dia^{\cor}}(I_1, \dots, I_n; J)$ of 1-morphisms of $\Dia^{\cor}$ is the truncated category $\tau_1(\Span^F(I_1, \dots, I_n; J))$. 
\end{enumerate}
\end{DEF}

Composition is defined by taking fiber products. The diagram (forgetting the functor to $J_i$)
\[ \xymatrix{
&&&&& A \times_{J_i} B  \ar[lld] \ar[rrd] \\
&&& A \ar[llld] \ar[ld] \ar[rrrd]^{\beta_A} &&&&  B \ar[llld] \ar[ld]^{\alpha_{B,i}}  \ar[rd] \ar[rrrd] \\
I_1 & \cdots & I_n & ; & J_1 & \cdots & J_i & \cdots & J_m & ; & K
} \]
is defined to be the composition of the left hand side correspondence in $\Hom(I_1, \dots, I_n; J_i)$ with the right hand side correspondence in $\Hom(J_1, \dots, J_m; K)$. 
One checks that $A \times_{J_i} B \rightarrow J_1 \times \cdots \times J_{i-1} \times I_1 \times \cdots \times I_n \times J_{i+1} \times \cdots \times J_m$ is again a fibration and that $A \times_{J_i} B \rightarrow K$ is again an opfibration. From \cite[Lemma~4.4]{Hor15b} it follows that the composition is functorial in 2-morphisms and that the relations in $\pi_0$ are respected.
We assume that strictly associative fiber products have been chosen in $\Dia$ (cf.\@ Remark~\ref{BEMAFP} for a similar construction). 

The following is a 2-categorical generalization of the constructions \cite[2.2]{Hor15b} and \cite[Definition~2.3]{Hor15b}:

\begin{DEF}\label{DEFSPANS2X} Let $\SSS$ be a (lax/oplax) pre-2-multiderivator. 
For each collection
$(I_1, S_1), \dots, (I_n, S_n); (J,T)$, where $I_1, \dots, I_n, J$ are diagrams in $\Dia$ and $S_i \in \SSS(I_i), T \in \SSS(J)$ are objects, 
we define a pseudo-functor
\[ \Span_{\SSS}: \Span(I_1 , \dots, I_n; J)^{1-\op} \rightarrow \mathcal{CAT} \]
in the oplax case and
\[ \Span_{\SSS}: \Span(I_1 , \dots, I_n; J)^{1-\op, 2-\op} \rightarrow \mathcal{CAT} \]
in the lax case. 
$\Span_{\SSS}$ maps a multicorrespondence of diagrams in $\Dia$
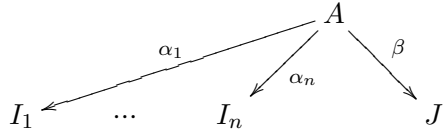
 \[ \xymatrix{ &&&A \ar[dlll]_{\alpha_1} \ar[dl]^{\alpha_n}  \ar[dr]^{\beta} \\ 
I_1 & \cdots & I_n && J
} \]
to the category 
\[ \Hom_{\SSS(A)}(\alpha_1^* S_1, \dots, \alpha_n^* S_n; \beta^*T), \]
maps a 1-morphism $(\gamma, \nu_1, \dots, \nu_n, \mu)$ to the functor
\[  \rho \mapsto \SSS(\mu)(T) \circ (\gamma^* \rho) \circ (\SSS(\nu_1)(S_1), \dots, \SSS(\nu_n)(S_n)) \]
and maps a 2-morphism represented by $\eta: \gamma \Rightarrow \gamma'$ (and such that $(\alpha_i' \ast \eta) \circ \nu_i  = \nu_i'$ and $\mu' \circ (\beta' \ast \eta)  = \mu$) to
the morphism
\begin{equation}\label{eqspans2x1}
  \SSS(\mu)(T) \circ (\gamma^* \rho) \circ (\SSS(\nu_1)(S_1), \dots, \SSS(\nu_n)(S_n)) \leftrightarrow \SSS(\mu')(T) \circ ((\gamma')^* \rho) \circ (\SSS(\nu_1)(S_1), \dots, \SSS(\nu_n)(S_n))  
\end{equation}
given as the composition of the isomorphisms coming from the pseudo-functoriality of $\SSS: \Hom(I, J) \rightarrow \Hom(\SSS(J), \SSS(I))$: 
\[ \SSS(\mu)(T) \iso  \SSS(\mu')(T) \circ  \underbrace{\SSS(\beta' \ast \eta)(T)}_{= \SSS(\eta)((\beta')^*T)}    \]
\[  \underbrace{\SSS(\alpha_i' \ast \eta)(S_i)}_{= \SSS(\eta)((\alpha_i')^*S_i)} \circ \SSS(\nu_i)(S_i) \iso \SSS(\nu_i)(S_i)    \]
with the morphism
\begin{equation}\label{eqspans2x2}
 \SSS(\eta)((\beta')^*T) \circ (\gamma^* \rho)  \leftrightarrow ((\gamma')^* \rho) \circ (\SSS(\eta)((\alpha'_1)^*S_1), \dots, \SSS(\eta)((\alpha'_n)^*S_n))   
\end{equation}
coming from the fact that $\SSS(\eta)$ is a (lax/oplax) pseudo-natural transformation $\gamma^* \Rightarrow  (\gamma')^*$. The morphisms  (\ref{eqspans2x1}) and (\ref{eqspans2x2}) point to the left in the lax case and to the right in the oplax case. 
\end{DEF}

\begin{DEF}\label{DEFSPANS2} Let $\SSS$ be a (lax/oplax) pre-2-multiderivator. 
Let $S_i \in \SSS(I_i)$, for $i=1, \dots, n$ and $T \in \SSS(J)$ be objects. 
Let \[ \Span_{\SSS}((I_1, S_1), \dots, (I_n, S_n); (J,T)) \] be the strict 2-category obtained from the pseudo-functor $\Span_{\SSS}$ defined in \ref{DEFSPANS2X}
by the 2-categorical Grothendieck construction \cite[Definition~2.15]{Hor15b}. 
\end{DEF}

Both definitions depend on the choice of $\Dia$, but we do not specify it explicitly. 

\begin{PAR}
The category $\Span_{\SSS}((I_1, S_1), \dots, (I_n, S_n); (J,T))$ defined in~\ref{DEFSPANS2} is very important to understand fibered multiderivators. Therefore we explicitly spell out the definition in detail:
\begin{enumerate}
\item
Objects are a multicorrespondence of diagrams in $\Dia$
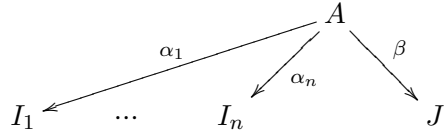
\[ \xymatrix{ &&&A \ar[dlll]_{\alpha_1} \ar[dl]^{\alpha_n}  \ar[dr]^{\beta} \\ 
I_1 & \cdots & I_n && J
} \]
together with a 1-morphism 
\[ \rho \in \Hom(\alpha_1^*S_1, \dots, \alpha_n^*S_n; \beta^*T) \]
in $\SSS(A)$. 

\item The 1-morphisms $(A, \alpha_1, \dots, \alpha_n, \beta, \rho) \rightarrow (A', \alpha_1', \dots, \alpha_n', \beta', \rho')$ are tuples
$(\gamma, \nu_1, \dots, \nu_n, \mu, \Xi)$, where 
 $\gamma: A \rightarrow A'$ is a functor,  $\nu_i$ is a natural transformation in 
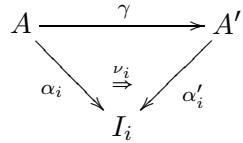
\[ \xymatrix{
A \ar[rd]_{\alpha_i} \ar[rr]^{\gamma} & \ar@{}[d]|{\overset{\nu_i}{\Rightarrow}} & A' \ar[ld]^{\alpha_i'}  \\
&I_i
} \]
and $\mu$ is a natural transformation in 
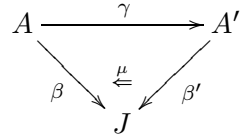
\[ \xymatrix{
A \ar[rd]_{\beta} \ar[rr]^{\gamma} & \ar@{}[d]|{\overset{\mu}{\Leftarrow}}  & A' \ar[ld]^{\beta'}  \\
&J
} \]
and $\Xi$ is a 2-morphism in 
\[ \xymatrix{
\gamma^* (\alpha')^* S \ar[rr]^{\gamma^* \rho'}  && \gamma^* (\beta')^* T \ar[dd]^{\SSS(\mu)} \\
\ar@{}[rr]|{\Uparrow^{\Xi}} &&\\
\alpha^* S \ar[uu]^{\SSS(\nu)} \ar[rr]^\rho && \beta^* T
} \]

\item The 2-morphisms are the natural transformations
$\eta: \gamma \Rightarrow \gamma'$ such that $(\alpha_i' \ast \eta) \circ \nu_i  = \nu_i'$ and $(\beta' \ast \eta) \circ \mu' = \mu$
and such that following prism-shaped diagram 
\[ \xymatrix{
\gamma^* (\alpha')^* S \ar[rrrrr]^{\gamma^* \rho'}  \ar[drdr]^{\qquad\SSS(\alpha' \ast \eta)(S)=\SSS(\eta)((\alpha')^*
S) } &&&&& \gamma^* (\beta')^* T \ar[dddd]^{\SSS(\mu)(T)} \ar[dldl]_{\SSS(\eta)((\beta')^*T)=\SSS(\beta' \ast \eta)(T)\qquad}  \\
&& \ar@{}[r]_{\Updownarrow^{\SSS(\alpha' \ast \eta)(\rho')}} &&  \\
&\overset{\SSS_{\eta, \nu}(S)}{\Rightarrow} &(\gamma')^* (\alpha')^*S \ar[r]^{(\gamma')^* \rho'} & (\gamma')^* (\beta')^*T \ar[drdr]_{\SSS(\mu')(T)} & \overset{\SSS_{\mu',\eta}(T)}{\Rightarrow}  \\
&& \ar@{}[r]|{\Uparrow^{\Xi'}} && \\
\alpha^* S \ar[uuuu]^{\SSS(\nu)(S)} \ar[rrrrr]^\rho \ar[urur]_{\SSS(\nu')(S)} &&&&&  \beta^* T
} \]
is 2-commutative, where the 2-morphism in the front face (not depicted) points upwards and is $\Xi$. We assumed here $n=1$ for simplicity.  Note that we have $\SSS(\alpha' \ast \eta)(S)=\SSS(\eta)((\alpha')^*
S)$ because $\SSS$ is strictly compatible with composition of 1-morphisms (cf.\@ Definition~\ref{DEF2PREMULTIDER}). Note that the 2-morphism denoted $\Updownarrow$ goes up in the lax case and down in the oplax (and plain) case while the two  `horizontal' 2-morphisms are invertible. 
\end{enumerate}
We again define the full subcategory $\Span^F_{\SSS}$ insisting that $\alpha_1 \times \cdots \times \alpha_n: A \rightarrow I_1 \times \cdots \times I_n$ is a Grothendieck fibration and $\beta$ is an opfibration.
\end{PAR}

\begin{LEMMA}\label{LEMMASPAN1}
Let $p: \DD \rightarrow \SSS$ be a strict morphism of (lax/oplax) pre-2-multiderivators (cf.\@ Definition~\ref{DEF2PREMULTIDERSTRICTMOR}). 

Consider the strictly commuting diagram of 2-categories and strict 2-functors
\[ \xymatrix{
\Span^F_{\DD}((I_1, \mathcal{E}_1), \dots, (I_n, \mathcal{E}_n); (J,\mathcal{F})) \ar@{^{(}->}[r] \ar[d] & \Span_{\DD}((I_1, \mathcal{E}_1), \dots, (I_n, \mathcal{E}_n); (J,\mathcal{F})) \ar[d] \\
\Span^F_{\SSS}((I_1, S_1), \dots, (I_n, S_n); (J,T))) \ar@{^{(}->}[r] \ar[d] & \Span_{\SSS}((I_1, S_1), \dots, (I_n, S_n); (J,T)) \ar[d] \\
\Span^F(I_1, \dots, I_n; J) \ar@{^{(}->}[r] &  \Span(I_1, \dots, I_n; J)
} \]
\begin{enumerate}
\item If the functors $\Hom_{\DD(I)}(-,-) \rightarrow \Hom_{\SSS(I)}(-,-)$ induced by $p$ are fibrations, the vertical 2-functors are 1-fibrations with 1-categorical fibers. 
They are 2-fibrations in the lax case and 2-opfibrations in the oplax case. 
\item If the functors $\Hom_{\DD(I)}(-,-) \rightarrow \Hom_{\SSS(I)}(-,-)$ induced by $p$ are fibrations {\em with discrete fibers}, then the upper vertical 2-functors have discrete fibers. 
\item Every object in a 2-category on the right hand side is in the image of the corresponding horizontal 2-functor {\em up to a chain of adjunctions}. 
\end{enumerate}
\end{LEMMA}

\begin{proof}This Lemma is a straightforward generalization of \cite[Lemma~5.4]{Hor15b}.
1.\@ and 2.\@ follow directly from the definition.
3.\@ 
We  first embed the left hand side category, say $\Span_{\SSS}^F((I_1, S_1), \dots, (I_n, S_n); (J,T))$, into the full subcategory of $\Span_{\SSS}((I_1, S_1), \dots, (I_n, S_n); (J,T))$ consisting of objects $(A, \alpha_1, \dots, \alpha_n, \beta, \rho)$, in which $\beta$ is an opfibration but the $\alpha_i$ are arbitrary. 
We will show that every object is connected by an adjunction with an object of this bigger subcategory. By a similar argument one shows that this holds also for the second inclusion. 

Consider an arbitrary correspondence $\xi'$ of diagrams in $\Dia$
 \[ \xymatrix{
& & & A \ar[rd]^{\beta} \ar[ld]^{\alpha_n}  \ar[llld]_{\alpha_1} \\
I_1 & \cdots & I_n & & J
} \]
and the 1-morphisms in $\Span(I_1, \dots, I_n; J)$
 
 \[ 
 \xymatrix{
& & & A  \times_{/J} J \ar[rd]^{\pr_2} \ar[ld]^{\alpha_n\pr_1} \ar[llld]_{\alpha_1\pr_1} \ar[dd]^{\pr_1} \\
I_1 & \cdots & I_n & \ar@{}[r]|{\Rightarrow^{\mu}} & J \\
& & & A  \ar[ru]_{\beta} \ar[lu]_{\alpha_n} \ar[lllu]^{\alpha_1} 
}  \quad  \xymatrix{
& & & A  \ar[rd]^{\beta} \ar[ld]^{\alpha_n} \ar[llld]_{\alpha_1} \ar[dd]^{\Delta}   \\
I_1 & \cdots & I_n &  & J \\
& & & A  \times_{/J} J \ar[ru]_{\pr_2} \ar[lu]_{\alpha_n\pr_1} \ar[lllu]^{\alpha_1\pr_1} 
} \]

One easily checks that $\pr_1 \circ \Delta = \id_A$ and that the obvious 2-morphism  $\Delta \circ \pr_1 \Rightarrow \id_{A\times_{/J} J}$ induced by $\mu$
define an adjunction in the 2-category $\Span(I_1, \dots, I_n; J)$. 
Using \cite[Lemma~5.5]{Hor15b}, we get a corresponding adjunction also in the 2-category $\Span_{\SSS}((I_1, S_1), \dots, (I_n, S_n); (J, T))$.
\end{proof}

\begin{LEMMA}\label{PAREQSPANFIBERED}
Let $p: \DD \rightarrow \SSS$ be a morphism of (lax/oplax) pre-2-multiderivators. Consider the following strictly commuting diagram of functors obtained from the one of Lemma~\ref{LEMMASPAN1} by 1-truncation \cite[1.9]{Hor15b}:  
\[ \xymatrix{
\tau_1(\Span^F_{\DD}((I_1, \mathcal{E}_1), \dots, (I_n, \mathcal{E}_n); (J,\mathcal{F}))) \ar@{^{(}->}[r] \ar[d] & \tau_1(\Span_{\DD}((I_1, \mathcal{E}_1), \dots, (I_n, \mathcal{E}_n); (J,\mathcal{F}))) \ar[d] \\
\tau_1(\Span^F_{\SSS}((I_1, S_1), \dots, (I_n, S_n); (J,T))) \ar@{^{(}->}[r] \ar[d] & \tau_1(\Span_{\SSS}((I_1, S_1), \dots, (I_n, S_n); (J,T))) \ar[d] \\
\tau_1( \Span^F(I_1, \dots, I_n; J)) \ar@{^{(}->}[r] & \tau_1( \Span(I_1, \dots, I_n; J))
} \]
\begin{enumerate}
\item The horizontal functors are equivalences. 
\item If the functors $\Hom_{\DD(I)}(-,-) \rightarrow \Hom_{\SSS(I)}(-,-)$ induced by $p$ are fibrations {\em with discrete fibers}, then the upper vertical morphisms are fibrations with discrete fibers. Furthermore the top-most horizontal functor maps Cartesian morphisms to Cartesian morphisms. 
\end{enumerate}

\end{LEMMA}
\begin{proof}
This Lemma is a straightforward generalization of \cite[Lemma~5.6]{Hor15b}.
That the horizontal morphisms are equivalences follows from the definition of the truncation and Lemma~\ref{LEMMASPAN1}, 3. 
If we have a 1-fibration and 2-isofibration of 2-categories $\mathcal{D} \rightarrow \mathcal{C}$ with
{\em discrete} fibers  then the truncation $\tau_1(\mathcal{D}) \rightarrow \tau_1(\mathcal{C})$ is again fibered (in the 1-categorical sense). Hence the second assertion follows from Lemma~\ref{LEMMASPAN1}, 2.
\end{proof}

\begin{DEF}\label{PARDEFDIACORD2}
Let $\SSS$ be a pre-2-multiderivator. 
We define a 2-multicategory $\Dia^{\cor}(\SSS)$ equipped with a strict functor
\[ \Dia^{\cor}(\SSS) \rightarrow \Dia^{\cor} \] 
as follows
\begin{enumerate}
\item The objects of $\Dia^{\cor}(\SSS)$ are pairs $(I, S)$ consisting of $I \in \Dia$ and $S \in \SSS(I)$.

\item The category $\Hom_{\Dia^{\cor}(\SSS)}((I_1, S_1), \dots, (I_n, S_n); (J, T))$ of 1-morphisms of $\Dia^{\cor}(\SSS)$ is
the truncated category $\tau_1(\Span^F_{\SSS}((I_1, S_1), \dots, (I_n, S_n); (J, T)))$. 
\end{enumerate}

Composition is given by the composition of correspondences of diagrams
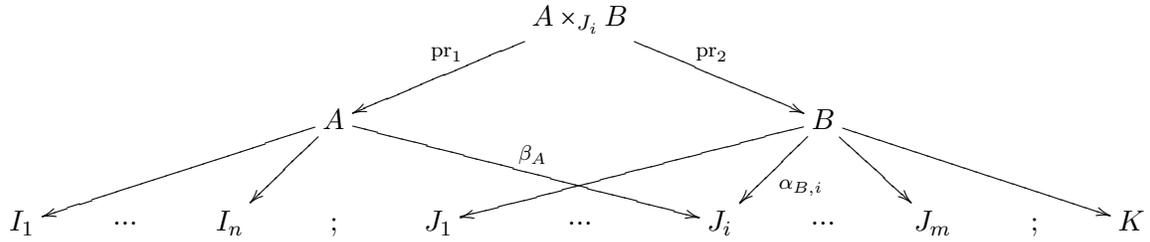
\[ \xymatrix{
&&&&& A \times_{J_i} B  \ar[lld]_{\pr_1} \ar[rrd]^{\pr_2} \\
&&& A \ar[llld] \ar[ld] \ar[rrrd]^{\beta_A} &&&&  B \ar[llld] \ar[ld]^{\alpha_{B,i}}  \ar[rd] \ar[rrrd] \\
I_1 & \cdots & I_n & ; & J_1 & \cdots & J_i & \cdots & J_m & ; & K
} \]
and composing $\rho_A \in \Hom(\alpha^*_{A,1}S_1, \dots, \alpha_{A,n}^* S_n; \beta^*_A T_i)$ with 
$\rho_B \in \Hom(\alpha^*_{B,1}T_1, \dots, \alpha_{B,m}^* T_m; \beta_B^* U)$ to
\[  (\pr_2^* \rho_B) \circ_i  (\pr_1^* \rho_A). \]

If $\SSS$ is symmetric or braided, then there is a natural action of the symmetric, resp.\@ braid groups:
\[ \Hom_{\Dia^{\cor}(\SSS)}((I_1,S_1), \dots, (I_n, S_n); (J, T)) \rightarrow  \Hom_{\Dia^{\cor}(\SSS)}((I_{\sigma(1)},S_{\sigma(1)}), \dots, (I_{\sigma(n)}, S_{\sigma(n)}); (J, T)) \]
involving the corresponding action in $\SSS$. This turns $\Dia^{\cor}(\SSS)$ into a symmetric, resp.\@ braided 2-multicategory. 
\end{DEF}

Note that because of the brute-force truncation this category is in general not 2-fibered anymore over $\Dia^{\cor}$.

For any strict morphism of pre-2-multiderivators $p: \DD \rightarrow \SSS$ we get an induced strict functor 
\[ \Dia^{\cor}(p): \Dia^{\cor}(\DD) \rightarrow \Dia^{\cor}(\SSS). \]

\section{Fibered multiderivators over pre-2-multiderivators}

The definition of a fibered multiderivator over pre-2-multiderivators is a straightforward generalization of the notion of {\em fibered multiderivator} from \cite{Hor15}.
In \cite{Hor15b} it was shown that this can, in a very neat way, alternatively be defined using the language of fibrations of 2-multicategories. It is also true for 
fibered multiderivators over pre-2-multiderivators. In this article we choose the slicker formulation as our definition:

\begin{DEF}\label{DEFFIBDER2}
A strict morphism $\DD \rightarrow \SSS$ of (lax/oplax) pre-2-multiderivators (Definition~\ref{DEF2PREMULTIDERSTRICTMOR}) such that $\DD$ and $\SSS$ each satisfy (Der1) and (Der2) (cf.\@ \ref{PARDER12}) is a
\begin{enumerate}
\item {\bf lax left (resp.\@ oplax right) fibered multiderivator} if the corresponding strict functor of 2-multicategories
\[ \Dia^{\cor}(p): \Dia^{\cor}(\DD) \rightarrow \Dia^{\cor}(\SSS) \]
of Definition~\ref{PARDEFDIACORD2} is a 1-opfibration (resp.\@ 1-fibration) and 2-fibration 
with 1-categorical fibers.
\item {\bf oplax left (resp.\@ lax right) fibered multiderivator} if the corresponding strict functor of 2-multicategories
\[ \Dia^{\cor}(p): \Dia^{\cor}(\DD^{2-\op}) \rightarrow \Dia^{\cor}(\SSS^{2-\op}) \]
of Definition~\ref{PARDEFDIACORD2} is a 1-opfibration (resp.\@ 1-fibration) and 2-fibration 
with 1-categorical fibers.
\end{enumerate}

Similarly, we define {\bf symmetric}, resp.\@ {\bf braided} fibered multiderivators where everything is, in addition, equipped in a compatible way with the action of the symmetric, resp.\@ braid groups. 
\end{DEF}

If in $\SSS$ all 2-morphisms are invertible then
left oplax=left lax and right oplax=right lax. In that case we omit the adjectives ``lax'' and ``oplax''. 

It seems that, in the definition, one could release the assumption on 1-categorical fibers, to get an apparently more general definition.  However, then the 1-truncation involved in the definition of $\Dia^{\cor}(\SSS)$ is probably not the right thing to work with.
In particular one does not get any generalized definition of a 2-derivator (or monoidal 2-derivator) as 2-fibered (multi-)derivator over $\{\cdot\}$.

The following Theorem~\ref{MAINTHEOREMFIBDER2} gives an alternative definition of a left/right fibered multiderivator over 
a pre-2-multiderivator $\SSS$ more in the spirit of the original (1-categorical) definition of \cite{Hor15}.

\begin{SATZ}\label{MAINTHEOREMFIBDER2}
A strict morphism $p: \DD \rightarrow \SSS$ of (lax/oplax) pre-2-multiderivators such that $\DD$ and $\SSS$ both satisfy (Der1) and (Der2) is a
{\em left (resp.\@ right) fibered multiderivator} if and only if the following axioms (FDer0 left/right) and (FDer3--5 left/right) hold true\footnote{where (FDer3--5 left), resp.\@ (FDer3--5 right), only make sense in the presence of (FDer0 left), resp.\@ (FDer0 right)}. 

Here (FDer3--4 left/right) can be replaced by the weaker (FDer3--4 left/right').
\end{SATZ}

\begin{enumerate}
\item[(FDer0 left)]
For each $I$ in $\Dia$ the morphism $p$ specializes to an 1-opfibered 2-multicategory with 1-categorical fibers. 
It is, in addition, 2-fibered in the lax case and 2-opfibered in the oplax case. 
Moreover any functor $\alpha: I \rightarrow J$ in $\Dia$ induces a  diagram
\[ \xymatrix{
\DD(J) \ar[r]^{\alpha^*} \ar[d] & \DD(I) \ar[d]\\
\SSS(J) \ar[r]^{\alpha^*} & \SSS(I) 
}\]
of 1-opfibered and 2-(op)fibered 2-multicategories, i.e.\@ the top horizontal functor maps coCartesian 1-morphisms to coCartesian 1-morphisms and (co)Cartesian 2-morphisms to (co)Cartesian 2-morphisms.

\item[(FDer3 left)]
For each functor $\alpha: I \rightarrow J$ in $\Dia$, and $S \in \SSS(J)$, the functor
$\alpha^*$ between fibers (which are 1-categories by (FDer0 left))
\[ \DD(J)_{S} \rightarrow \DD(I)_{\alpha^*S} \]
has a left adjoint $\alpha_!^{(S)}$.

\item[(FDer4 left)]
For each functor $\alpha: I \rightarrow J$ in $\Dia$, for each object $j \in J$, and for each $S \in \SSS(J)$, for the 2-commutative square
\[ \xymatrix{  I \times_{/J} j \ar[r]^-\iota \ar[d]_{\alpha_j} \ar@{}[dr]|{\Swarrow^\mu} & I \ar[d]^\alpha \\
\{j\} \ar@{^{(}->}[r]^j & J \\
} \]
 the induced natural transformation of functors $\alpha_{j,!}^{(S_j)} (\SSS(\mu)(S))_\bullet \iota^* \rightarrow j^* {\alpha}_!^{(S)} $ is an isomorphism.
\item[(FDer5 left)] For any opfibration $\alpha: I \rightarrow J $ in $\Dia$, and for any 1-morphism $\xi \in \Hom(S_1, \dots, S_n; T)$ in $\SSS(\cdot)$ for some $n\ge 1$, the natural transformation of functors
\[ \alpha_!^{(T)} (\alpha^*\xi)_\bullet (\alpha^*-, \cdots, \alpha^*-,\ \underbrace{-}_{\text{at }i}\ , \alpha^*-, \cdots, \alpha^*-) \cong  \xi_\bullet (-, \cdots, -,\ \underbrace{\alpha_!^{(S_i)}-}_{\text{at }i}\ , -, \cdots, -) \]
is an isomorphisms for all $i=1, \dots,  n$.
\end{enumerate}

Instead of (FDer3/4 left) the following axioms are sufficient:
\begin{enumerate}
\item[(FDer3 left')]
For each {\em opfibration} $\alpha: I \rightarrow J$ in $\Dia$ and each $S \in \SSS(J)$ the functor
$\alpha^*$ between fibers (which are 1-categories by (FDer0 left))
\[ \DD(J)_{S} \rightarrow \DD(I)_{\alpha^*S} \]
has a left-adjoint $\alpha_!^{(S)}$.
\item[(FDer4 left')]
For each {\em opfibration} $\alpha: I \rightarrow J$ in $\Dia$, for each object $j \in J$, and for each $S \in \SSS(J)$
 the induced natural transformation of functors $\pr_{2,!}^{(S_j)} \pr_1^* \rightarrow j^* {\alpha}_!^{(S)}$ is an isomorphism. Here $\pr_1$ and $\pr_2$ are defined by the Cartesian square
 \[ \xymatrix{  I \times_{J} j \ar[r]^-{\pr_1} \ar[d]_{\pr_2}  & I \ar[d]^\alpha \\
\{j\} \ar@{^{(}->}[r]^j & J. \\
} \]

\end{enumerate}

We use the same notation for the axioms as in the case of usual fibered multiderivators because, in case that $\SSS$ is a usual pre-2-multiderivator they specialize to the familiar ones. 
Dually, we have the following axioms: 

\begin{enumerate}
\item[(FDer0 right)]
For each $I$ in $\Dia$ the morphism $p$ specializes to a 1-fibered multicategory with 1-categorical fibers.
It is, in addition, 2-fibered in the lax case and 2-opfibered in the oplax case. 
Furthermore, any {\em opfibration} $\alpha: I \rightarrow J$
in $\Dia$  induces a diagram
\[ \xymatrix{
\DD(J) \ar[r]^{\alpha^*} \ar[d] & \DD(I) \ar[d]\\
\SSS(J) \ar[r]^{\alpha^*} & \SSS(I) 
}\]
of 1-fibered and 2-(op)fibered multicategories, i.e.\@ the top horizontal functor maps Cartesian 1-morphisms w.r.t.\@ the $i$-th slot to Cartesian 1-morphisms w.r.t.\@ the $i$-th slot for any $i$ and maps (co)Cartesian 2-morphisms to (co)Cartesian 2-morphisms.
\item[(FDer3 right)]
For each functor $\alpha: I \rightarrow J$ in $\Dia$, and each $S \in \SSS(J)$ the functor
$\alpha^*$ between fibers (which are 1-categories by (FDer0 right))
\[ \DD(J)_{S} \rightarrow \DD(I)_{\alpha^*S} \]
has a right adjoint $\alpha_*^{(S)}$.
\item[(FDer4 right)]
For each morphism $\alpha: I \rightarrow J$ in $\Dia$,  for each object $j \in J$, and for each $S \in \SSS(J)$, for the 2-commutative square
\[ \xymatrix{  j \times_{/J} I \ar[r]^-\iota \ar[d]_{\alpha_j} \ar@{}[dr]|{\Nearrow^\mu} & I \ar[d]^\alpha \\
\{j\} \ar@{^{(}->}[r]^j & J \\
} \]
 the induced natural transformation of functors $ j^* {\alpha}_*^{(S)} \rightarrow \alpha_{j,*}^{(S_j)} (\SSS(\mu)(S))^\bullet \iota^*$ is an isomorphism.

\item[(FDer5 right)] For {\em each} functor $\alpha: I \rightarrow J$ in $\Dia$, and for each 1-morphism $\xi \in \Hom(S_1, \dots, S_n; T)$ in $\SSS(\cdot)$ for some $n\ge 1$, the natural transformation of functors
\[ \alpha_*^{(S_i)} (\alpha^*\xi)^{\bullet,i} (\alpha^*-, \overset{\widehat{i}}{\cdots}, \alpha^*-\ ;\ -) \cong  \xi^{\bullet,i} (-, \overset{\widehat{i}}{\cdots}, -\ ;\ \alpha_*^{(T)}-) \]
is an isomorphisms for all $i= 1, \dots, n$.
\end{enumerate}

There is similarly a weaker version of (FDer3/4 right) in which $\alpha$ has to be a fibration.

For {\em representable} pre-2-multiderivators (cf.\@ \ref{PARREPR2PREMULTIDER}) we get the following:

\begin{PROP}\label{PROP2MULTIDER}
If $\mathcal{D} \rightarrow \mathcal{S}$ is a 1-bifibration and 2-fibration of 2-multicategories with 1-categorical and bicomplete fibers then 
\begin{enumerate}
\item $\Dia^{\cor}(\DD^{\lax}) \rightarrow \Dia^{\cor}(\SSS^{\lax})$ is a 1-opfibration and 2-fibration,
\item $\Dia^{\cor}(\DD^{\oplax}) \rightarrow \Dia^{\cor}(\SSS^{\oplax})$ is a 1-fibration and 2-fibration. 
\end{enumerate}
If $\mathcal{D} \rightarrow \mathcal{S}$ is a 1-bifibration and 2-opfibration  of 2-multicategories with 1-categorical and bicomplete fibers then 
\begin{enumerate}
\item $\Dia^{\cor}(\DD^{\lax, {2-\op}}) \rightarrow \Dia^{\cor}(\SSS^{\lax, 2-\op})$ is a 1-fibration and 2-fibration.
\item $\Dia^{\cor}(\DD^{\oplax, 2-\op}) \rightarrow \Dia^{\cor}(\SSS^{\oplax, 2-\op})$ is a 1-opfibration and 2-fibration.
\end{enumerate}
\end{PROP}
\begin{proof}
This follows from Proposition~\ref{PROPREPR} doing the same constructions as in \cite[Proposition 5.1.26]{Hor15}.
\end{proof}

\begin{DEF}\label{DEFSTABLE}For (lax/oplax) fibered derivators over an (lax/oplax) pre-2-derivator $p: \DD \rightarrow \SSS$ and an object $S \in \SSS(I)$ we define its {\bf fiber} over $(I, S)$ as the (usual) derivator
\[ \DD_{I,S}: J \mapsto \DD(I \times J)_{\pr_2^*S}.  \]
We say that $p$ has {\bf stable fibers} if $\DD_{I,S}$ is stable for all $S \in \SSS(I)$ and for all $I$. 
\end{DEF}

\begin{PAR}
Let $\SSS$ be a pre-2-multiderivator.
To prove Theorem~\ref{MAINTHEOREMFIBDER2} we need some preparation to improve our understanding of the category $\Dia^{\cor}(\SSS)$.
We recall the results from \cite[Section~6]{Hor15b} indicating necessary changes in the constructions. In \cite{Hor15b} the statements have been proven only for pre-multiderivators (as opposed to pre-2-multiderivators).
\end{PAR}

\begin{PAR}\label{DEF3MORPH}
We will define three types of generating\footnote{``Generating'' in the sense that any 1-morphism in $\Dia^{\cor}(\SSS)$ is 2-isomorphic to a composition of these (cf.\@ Corollary~\ref{KORFIBDER2_2}).} 1-morphisms in $\Dia^{\cor}(\SSS)$. We first define them as objects in the categories
$\Span_{\SSS}(\dots)$ (without the restriction ${}^F$). 

\begin{enumerate}
\item[${[\beta^{(S)}]}$] for a functor $\beta: I \rightarrow J$ in $\Dia$ and an object $S \in \SSS(J)$, 
consists of the correspondence of diagrams
\[ \xymatrix{
& I \ar[rd]^{\beta}\ar@{=}[ld] \\
I & ; & J
} \]
and over it in $\tau_1(\Span_{\SSS}((I, \beta^*S ); (J, S)))$ the canonical correspondence given by the identity $\id_{\beta^*S}$.

\item[${[\alpha^{(S)}]'}$] for a functor $\alpha: I \rightarrow J$ in $\Dia$ and an object $S \in \SSS(J)$,
consists of the correspondence of diagrams 
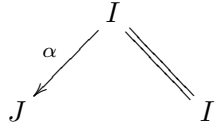
\[ \xymatrix{
& I \ar[ld]_{\alpha}\ar@{=}[rd] \\
J & ; & I
} \]
and over it in $\tau_1(\Span_{\SSS}((J, S); (I, \alpha^*S )))$ the canonical correspondence given by the identity $\id_{\alpha^*S}$.

\item[${[f]}$] for a morphism $f \in \Hom_{\SSS(A)}(S_1, \dots, S_n; T)$, where $A$ is any diagram in $\Dia$, and $S_1, \dots, S_n, T$ are objects in $\SSS(A)$, is defined by the trivial correspondence of diagrams 
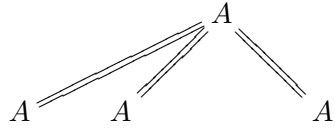
\[ \xymatrix{
& & & A \ar@{=}[llld] \ar@{=}[ld] \ar@{=}[rd] \\
A & \cdots & A &;  & A
} \]
together with $f$.  
\end{enumerate}
\end{PAR}

\begin{PAR}\label{PARDIACORSCAN1MOR}
Note that the correspondences of the last paragraph do not define 1-morphisms in $\Dia^{\cor}(\SSS)$ yet, as we defined it, because they are not always objects in the $\Span^F$ subcategory ($[\alpha^{(S)}]'$ is already, if $\alpha$ is a fibration; $[\beta^{(S)}]$ is, if $\beta$ is an opfibration; and $[f]$ is, if $n=0,1$, respectively). 

From now on, we denote by the same symbols $[\beta^{(S)}], [\alpha^{(S)}]', [f]$ chosen 1-morphisms in $\Dia^{\cor}(\SSS)$ which are isomorphic to those defined above in the $\tau_1$-categories (cf.\@ Lemma~\ref{PAREQSPANFIBERED}). Those are determined only up to 2-isomorphism in $\Dia^{\cor}(\SSS)$. 

For definiteness, we choose $[\beta^{(S)}]$ to be the correspondence 
\[ \xymatrix{
& I \times_{/J} J \ar[rd]^{\pr_2}\ar[ld]_{\pr_1} \\
I & & J
} \]
and over it in $\tau_1(\Span_{\SSS}((I, \beta^*S ); (J, S)))$ the morphism $\pr_1^*\beta^*S \rightarrow \pr_2^*S$ 
given by the natural transformation $\mu_\beta: \beta \circ \pr_1 \Rightarrow \pr_2$.
Similarly, we choose $[\alpha^{(S)}]'$ to be the correspondence 
\[ \xymatrix{
& J \times_{/J} I \ar[rd]^{\pr_2}\ar[ld]_{\pr_1} \\
J & & I
} \]
and over it in $\tau_1(\Span_{\SSS}((J, S), (I, \alpha^*S )))$ the morphism $\pr_1^* S \rightarrow \pr_2^* \alpha^*S$ 
given by the natural transformation $\mu_\alpha: \pr_1 \Rightarrow \alpha \circ  \pr_2$.
\end{PAR}

\begin{PAR}\label{DIACORSADJUNCTION}
For any $\alpha: I \rightarrow J$, and an object $S \in \SSS(J)$, we define a 2-morphism
\[ \epsilon: \id_{(I, \alpha^*S)} \Rightarrow [\alpha^{(S)}]' \circ  [\alpha^{(S)}] \]
given by the commutative diagrams
\[ \vcenter{\xymatrix{
& I \ar@{=}[rd]^{} \ar@{=}[ld] \ar[dd]^{\Delta} \\
I & & I \\
& I \times_{/J} J \times_{/J} I \ar[ru]_{\pr_3}\ar[lu]^{\pr_1} 
} } \qquad \vcenter{ \xymatrix{
\Delta^* \pr_1^*  \alpha^* S \ar@{=}[rrrr]^{} 
 &&&& \Delta^* \pr_3^* \alpha^* S  \ar@{=}[d] \\
\alpha^* S \ar@{=}[rrrr] \ar@{=}[u] &&&& \alpha^* S
} } \]

and we define a 2-morphism
\[ \mu:  [\alpha^{(S)}] \circ  [\alpha^{(S)}]' \Rightarrow \id_{(J,S)} \]
given by the (2-)commutative diagrams
\[ \vcenter{ \xymatrix{
& J \times_{/J} I \times_{/J} J  \ar[dd]|{\alpha \pr_2}  \ar[rd]^{\pr_3} \ar[ld]_{\pr_1}  \\
J \ar@{}[r]|{\overset{\mu_1}{\Rightarrow}} &  \ar@{}[r]|{\overset{\mu_2}{\Rightarrow}}  & J \\
& J \ar@{=}[ru]^{} \ar@{=}[lu]
} } \qquad  \vcenter{ \xymatrix{
\pr_2^* \alpha^*  S   \ar@{=}[rrr] &&&  \pr_2^* \alpha^*  S  \ar[d]^{\SSS(\mu_2)(S)} \\
\pr_1^* S \ar[rrr]_{\SSS(\mu_2)(S) \circ \SSS(\mu_1)(S)} \ar[u]^{\SSS(\mu_1)(S)}  &&& \pr_3^* S
}  } \]
Note that the right hand side diagram (although living in the 2-category $\SSS(J \times_{/J} I \times_{/J} J)$) is strictly commutative. 
\end{PAR}
\begin{PAR}\label{DIACORS2FUNCT}
A natural transformation $\nu: \alpha_1 \Rightarrow \alpha_2$ establishes a 2-morphism
\[ [\nu]': [\SSS(\nu)(S)] \circ [\alpha_1^{(S)}]' \Rightarrow [\alpha_2^{(S)}]'  \]
given by the 2-commutative diagrams: 
\[ \vcenter{ \xymatrix{
& J \times_{/J,\alpha_1} I \ar[rd]^{\pr_2'} \ar[ld]_{\pr_1'} \ar[dd]_{\widetilde{\nu}} \\
J & & I \\
& J \times_{/J,\alpha_2} I  \ar[ru]_{\pr_2}\ar[lu]^{\pr_1} 
} } \qquad  \vcenter{ \xymatrix{
\widetilde{\nu}^* \pr_1^* S \ar@{}[rrrrrd]|\Downarrow \ar[rrrrr]^-{\widetilde{\nu}^* \SSS(\mu_{\alpha_2})(S)} \ar@{=}[d]  & & & &  & \widetilde{\nu}^* \pr_2^* \alpha_2^* S \\
 (\pr_1')^*  S  \ar[rrr]_-{ \SSS(\mu_{\alpha_1})(S)} &&& (\pr_2')^* \alpha_1^* S  \ar[rr]_-{  (\pr_2')^* \SSS(\nu)(S)} &&  (\pr_2')^* \alpha_2^*S  \ar@{=}[u]  
} } \]
where the 2-morphism in the right hand side diagram is the inverse of the pseudo-functoriality of $\SSS: \Hom(J \times_{/J, \alpha_1} I, J) \mapsto \Hom(\SSS(J), \SSS(J \times_{/J, \alpha_1} I))$.
Note that we have the equation of natural transformations $(\pr_2'  \ast \nu) \circ  \mu_{\alpha_1} = \widetilde{\nu} \ast \mu_{\alpha_2}$. 
Here the $\mu_{\alpha_i}$ are as in \ref{PARDIACORSCAN1MOR}.

Similarly, a natural transformation $\nu: \beta_1 \Rightarrow \beta_2$ establishes a 2-morphism
\[ [\nu]:    [\beta_2^{(S)}] \circ [\SSS(\nu)(S)] \Rightarrow  [\beta_1^{(S)}]. \]
\end{PAR}

\begin{PAR}\label{DIACORS2COMMSQUARE}
Consider the diagrams from axiom (FDer3 left/right)
\[ \xymatrix{
I \times_{/J} {j} \ar[r]^-{\iota_1} \ar[d]_{p_1} \ar@{}[rd]|{\Swarrow^{\mu_1}}  & I \ar[d]^\alpha \\
j \ar@{^{(}->}[r] & J
} \quad \xymatrix{
{j} \times_{/J} I  \ar[r]^-{\iota_2} \ar[d]_{p_2} \ar@{}[rd]|{\Nearrow^{\mu_2}}  & I \ar[d]^\alpha \\
j \ar@{^{(}->}[r] & J
} \]
By the construction in \ref{DIACORS2FUNCT}, we get a canonical 2-morphism
\begin{equation}\label{2commsquare_l}
 [\SSS(\mu_1)(S)]   \circ [\iota_1^{(\alpha^*S)}]' \circ [\alpha^{(S)}]'  \Rightarrow  [p_1^{(S_j)}]' \circ [j^{(S)}]',
\end{equation}
and a canonical 2-morphism
\begin{equation}\label{2commsquare_r}
 [\alpha^{(S)}] \circ   [\iota_2^{(\alpha^*S)}] \circ [\SSS(\mu_2)(S)] \Rightarrow  [j^{(S)}] \circ [p_2^{(S_j)}],
\end{equation}
respectively. Here $S_j$ denotes $j^*S$ where $j$, by abuse of notation, also denotes the inclusion of the one-element category $\{j\}$ into $J$. 
\end{PAR}

\begin{PAR}\label{TIMESK}
Let  $\xi$ be any 1-morphism in $\Dia^{\cor}(\SSS)$ given by
\[ \xymatrix{
& && A \ar[rd]^{\delta} \ar[ld]^{\gamma_n} \ar[llld]_{\gamma_1}   \\
I_1 & \cdots &  I_n & ;  & J
} \]
together with a morphism 
\[ f_\xi  \in \Hom_{\SSS(A)}(\gamma_1^*S_1, \dots, \gamma_n^*S_n; \delta^* T). \]
We define a 1-morphism $\xi \times K$ in $\Dia^{\cor}(\SSS)$ by
\[ \xymatrix{
& && A \times K \ar[rd]^{\delta \times \id} \ar[ld]^{\gamma_n \times \id} \ar[llld]_{\gamma_1 \times \id}   \\
I_1 \times K & \cdots &  I_n \times K & & J \times K
} \]
and 
\[ f_{\xi \times K} := \pr_1^* f_\xi  \in \Hom_{\SSS(A)}( \pr_1^* \gamma_1^*S_1, \dots, \pr_1^*  \gamma_n^*S_n; \pr_1^* \delta^* T). \]
Note that the here defined $\xi \times K$ does not necessarily lie in the category $\Span^F_\SSS(\dots)$. Hence we denote by $\xi \times K$ any isomorphic (in the $\tau_1$-truncation) correspondence which does
lie in $\Span^F_\SSS(\dots)$. We also define a correspondence 
 $\xi \times_j K$ in $\Dia^{\cor}(\SSS)$ by
\[ \xymatrix{
& && &&A \times K \ar[rd]^{\delta \times \id} \ar[ld]^{\gamma_n \pr_1} \ar[llld]^{\gamma_j \times K}  \ar[llllld]_{\gamma_1 \pr_1}   \\
I_1  & \cdots&  I_j \times K & \cdots &  I_n  & & J \times K
} \]
and 
\[ f_{\xi \times_j K} := \pr_1^* \xi  \in \Hom_{\SSS(A)}( \pr_1^* \gamma_1^*S_1, \dots, \pr_1^*  \gamma_n^*S_n; \pr_1^* \delta^* T). \]
The here defined $\xi \times_j K$ does already lie in the category $\Span^F_\SSS(\dots)$. 
\end{PAR}

\begin{LEMMA}\label{LEMMAPROPDIACORS}
With the notation of \ref{DEF3MORPH}:
\begin{enumerate}
\item The 2-morphisms of \ref{DIACORSADJUNCTION}
\[ \epsilon: \id_{(I, \alpha^*S)} \Rightarrow [\alpha^{(S)}]' \circ  [\alpha^{(S)}] \qquad
 \mu:  [\alpha^{(S)}] \circ  [\alpha^{(S)}]' \Rightarrow \id_{(J, S)} \]
 establish an adjunction between $[\alpha^{(S)}]$ and $[\alpha^{(S)}]'$ in the 2-category $\Dia^{\cor}(\SSS)$. 
 \item The exchange 2-morphisms of (\ref{2commsquare_l}) and and of (\ref{2commsquare_r}) w.r.t.\@ the adjunction of 1.\@, namely
\[        [p_1^{(S_j)}]  \circ [\SSS(\mu_1)(S)] \circ  [\iota_1^{(\alpha^*S)}]' \Rightarrow  [j^{(S)}]' \circ [\alpha^{(S)}]   \]
and
\[  [\iota_2^{(\alpha^*S)}] \circ [\SSS(\mu_2)(S)]   \circ [p_2^{(S_j)}]' \Rightarrow  [\alpha^{(S)}]'  \circ [j^{(S)}]      \]
are 2-isomorphisms. 
\item For any $\alpha: K \rightarrow L$ and $\xi$ as in \ref{TIMESK} there are natural 2-isomorphisms
\begin{equation} \label{commtensorpullback1}
 (\xi \times K) \circ ([I_1 \times \alpha^{(\pr_1^*S_1)}]', \dots, [I_n\times\alpha^{(\pr_1^*S_n)}]')  \Iiso [J\times\alpha^{(\pr_1^*T)}]'  \circ (\xi \times L)   
\end{equation}
and
\begin{equation}  \label{commtensorpullback2}
   (\xi \times_j K) \circ_j [I_j \times \alpha^{(\pr_1^*S_j)}]'  \Iiso   [J \times \alpha^{(\pr_1^*T)}]' \circ (\xi \times_j L). 
\end{equation} 
\item 
The exchange of (\ref{commtensorpullback1}) w.r.t.\@ the adjunction of 1.\@, and w.r.t.\@ the $j$-th slot, namely
\[ [J \times \alpha^{(\pr_1^*T)}] \circ (\xi \times K)   \circ ([I_1 \times \alpha^{(\pr_1^*S_1)}]', \dots, \id, \dots,  [I_n \times \alpha^{(\pr_1^*S_n)}]')  \Rightarrow (\xi \times L) \circ_j [I_j \times \alpha^{(\pr_1^*S_j)}] \]
is an isomorphism if $\alpha$ is an opfibration. The exchange of (\ref{commtensorpullback2}) w.r.t.\@ the adjunction of 1.\@, namely
\[ [J \times \alpha^{(\pr_1^*T)}] \circ (\xi \times_j K)  \Rightarrow (\xi \times_j L) \circ_j [I_j \times \alpha^{(\pr_1^*S_j)}] \]
is an isomorphism for any $\alpha$. 
\item 
For any $f \in \Hom_{\SSS(J)}(S_1, \dots, S_n; T)$ and $\alpha: I \rightarrow J$ there is a natural 2-isomorphism
\begin{equation} \label{commfpullback}
[ \alpha^* f ]  \circ ([\alpha^{(S_1)}]', \dots, [\alpha^{(S_n)}]') \Iiso    [\alpha^{(T)}]' \circ [ f ] .
\end{equation}
\item The exchange of (\ref{commfpullback}) w.r.t.\@ the adjunction of 1.\@, and w.r.t.\@ the $j$-th slot, namely
\[  [\alpha^{(T)}] \circ [ \alpha^* f ] \circ ([\alpha^{(S_1)}]', \dots, \id, \dots,  [\alpha^{(S_n)}]')  \Rightarrow  [ f ] \circ_j  [\alpha^{(S_j)}] \]
is an isomorphism if $\alpha$ is an opfibration.
\end{enumerate}
\end{LEMMA}
\begin{proof}
The proof is a repetition of the proof of \cite[Lemma~5.7]{Hor15b} adapted to take care of the 2-morphisms in $\SSS$. 
\end{proof}

\begin{PAR}\label{PARPREPMORPHDIASCOR2}
Let $\DD \rightarrow \SSS$ be a morphism of  pre-multiderivators satisfying (Der1) and (Der2).  
Consider the strict 2-functor
\[ \Dia^{\cor}(\DD) \rightarrow  \Dia^{\cor}(\SSS)   \]
{\em and assume that it is a 1-opfibration, and 2-fibration with 1-categorical fibers.}
The fiber over a pair $(I, S)$ is just the fiber $\DD(I)_S$ of the usual functor $\DD(I) \rightarrow \SSS(I)$. 
The 1-opfibration and 2-fibration can be seen (via the construction of \cite[Proposition~2.17]{Hor15b}) as a pseudo-functor of 2-multicategories
\[ \Psi: \Dia^{\cor}(\SSS) \rightarrow \mathcal{CAT}. \]
\end{PAR}
\begin{PAR}\label{PARPREPMORPHDIASCOR2R}
If
\[ \Dia^{\cor}(\DD) \rightarrow  \Dia^{\cor}(\SSS)  \]
is a {\em 1-fibration}, and 2-fibration with 1-categorical fibers there is still an associated pseudo-functor of 2-categories (not of 2-multicategories)
\[ \Phi: \Dia^{\cor}(\SSS)^{1-\op, 2-\op} \rightarrow \mathcal{CAT}. \]
\end{PAR}

\begin{PROP}\label{PROPBASICDIACORS}With the notation of \ref{DEF3MORPH}:
\begin{enumerate}
\item 
Assume that
\[ \Dia^{\cor}(\DD) \rightarrow  \Dia^{\cor}(\SSS)   \]
is a {\em 1-opfibration}, and 2-fibration with 1-categorical fibers. 
Then the functor $\Psi$ of \ref{PARPREPMORPHDIASCOR2} maps (up to isomorphism of functors)
\begin{eqnarray*}
{ [\alpha^{(S)}]' } &\mapsto& (\alpha^{(S)})^*   \\
{ [\beta^{(S)}]  } &\mapsto &  \beta_!^{(S)} \\
{ [f] } &  \mapsto  & f_\bullet 
\end{eqnarray*}
where $(\alpha^{(S)})^*$ is $\alpha^*$ restricted to a functor $\DD(J)_S \rightarrow \DD(I)_{\alpha^*S}$ and
where $\beta_!^{(S)}$ is a left adjoint of $(\beta^{(S)})^*$, and where $f_\bullet$ is a functor determined by 
\[\Hom_{\DD(I), f}(\mathcal{E}_1, \dots, \mathcal{E}_n; \mathcal{F}) \cong \Hom_{\DD(I)_T}(f_\bullet(\mathcal{E}_1, \dots, \mathcal{E}_n); \mathcal{F}). \]

\item Assume that
\[ \Dia^{\cor}(\DD) \rightarrow  \Dia^{\cor}(\SSS) \]
is a {\em 1-fibration}, and 2-fibration with 1-categorical fibers. 

Then pullback functors\footnote{In the case of $[\alpha^{(S)}]$ and $[\beta^{(S)}]'$ these are $\Phi([\alpha^{(S)}])$ and $\Phi([\beta^{(S)}]')$.} in the sense of \cite[Definition~2.12 2.]{Hor15b} w.r.t.\@ the following 1-morphisms in $\Dia^{\cor}(\SSS)$ are given by 
\begin{eqnarray*}
{ [\beta^{(S)}]  } &\mapsto &  (\beta^{(S)})^* \\
{ [\alpha^{(S)}]' } &\mapsto&  \alpha_*^{(S)}  \\
{ [f] } &  \mapsto  & f^{\bullet, j}  \qquad \text{pullback w.r.t.\@ the $j$-th slot.}
\end{eqnarray*}
where $\alpha_*^{(S)}$ is a right adjoint of $(\alpha^{(S)})^*$ and $f^{\bullet, j}$ is a functor determined by 
\[ \Hom_{\DD(I), f}(\mathcal{E}_1, \dots, \mathcal{E}_n; \mathcal{F}) \cong \Hom_{\DD(I)_T}(\mathcal{E}_j, f^{\bullet, j}(\mathcal{E}_1, \overset{\widehat{j}}{\dots}, \mathcal{E}_n; \mathcal{F})). \] 
\end{enumerate}
\end{PROP}

\begin{proof}\cite[Proposition~6.10]{Hor15b}
\end{proof}

\begin{KOR}\label{KORFIBDER2_2}
Assuming the conditions of \ref{PARPREPMORPHDIASCOR2},
consider any correspondence 
\[ \xi'   \in \Span_{\SSS}((I_1, S_1), \dots, (I_n, S_n); (J, T)) \] consisting of 
\[ \xymatrix{ &&&A \ar[dlll]_{\alpha_1} \ar[dl]^{\alpha_n}  \ar[dr]^{\beta} \\ 
I_1 & \cdots & I_n &; & J
} \]
and a 1-morphism 
\[ f \in \Hom(\alpha_1^*S_1, \dots, \alpha_n^*S_n; \beta^*T) \]
in $\SSS(A)$.  
\begin{enumerate}
\item Over any 1-morphism  $\xi$ in $\Dia^{\cor}(\SSS)$, which is isomorphic to $\xi'$, a corresponding pushforward functor between
fibers (which is $\Psi(\xi')$ in the discussion \ref{PARPREPMORPHDIASCOR2}) is given (up to natural isomorphism) by the composition: 
\[ \beta_!^{(T)} \circ f_{\bullet} \circ (\alpha_1^*, \dots, \alpha_n^*). \]
\item Over any 1-morphism  $\xi$ in $\Dia^{\cor}(\SSS)$, which is isomorphic to $\xi'$, a pullback functor w.r.t.\@ any slot $j$ between
fibers (which is $\Phi(\xi')$ in the discussion \ref{PARPREPMORPHDIASCOR2R} if $\xi$ is a 1-ary 1-morphism) is given (up to natural isomorphism) by the composition: 
\[ \alpha_{j,*}^{(S_j)} \circ f^{\bullet,j} \circ (\alpha_1^*, \overset{\widehat{j}}{\dots},  \alpha_n^*; \beta^*). \]
\end{enumerate}
\end{KOR}
\begin{proof}\cite[Corollary~6.11]{Hor15b}
\end{proof}

We are now ready to give the
\begin{proof}[Proof of Theorem~\ref{MAINTHEOREMFIBDER2}] This is again a small adaption of the proof of \cite[Main~theorem~6.12]{Hor15b} to include the 2-morphisms in the base $\SSS$.
We concentrate on the lax left case, the other cases are shown completely analogously.

 We first show that $\Dia^{\cor}(\DD) \rightarrow \Dia^{\cor}(\SSS)$ is a 1-opfibration and a 2-fibration, if $\DD \rightarrow \SSS$ satisfies 
 (FDer0 left), (FDer3--4 left'), and (FDer5 left).
By (FDer0 left) $\DD(I) \rightarrow \SSS(I)$ is a 1-opfibration and 2-fibration with 1-categorical fibers and we have already seen that  this implies that 
$\Dia^{\cor}(\DD) \rightarrow \Dia^{\cor}(\SSS)$ is 2-fibered as well (cf.\@ Lemma~\ref{PAREQSPANFIBERED}).

 Let $(x,f)$ be a 1-morphism in $\Dia^{\cor}(\SSS)$ where $x=(A; \alpha_{A,1}, \dots, \alpha_{A,n}; \beta_A)$ is a correspondence in $\Span^F(I_1, \dots, I_n; J)$ and  
 \[ f \in \Hom_{\SSS(A)}(\alpha_{A,1}^*S_1, \dots, \alpha_{A,n}^*S_n; \beta_A^*T).  \]
 In $\Dia^{\cor}(\DD)$ we have the following composition of isomorphisms of sets (identifying small discrete categories with their set of isomorphism classes):
\begin{eqnarray*} 
&& \Hom_{\Dia^{\cor}(\DD),(x,f)}((I_1, \mathcal{E}_1), \dots, (I_n, \mathcal{E}_n); (J, \mathcal{F})) \\
&\cong& \Hom_{\DD(A),f}(\alpha_{A,1}^*\mathcal{E}_1, \dots, \alpha_{A,n}^*\mathcal{E}_n; \beta_A^* \mathcal{F}) \\
&\cong& \Hom_{\DD(A),\id_{\beta_A^*T}}( f_\bullet(\alpha_{A,1}^*\mathcal{E}_1, \dots, \alpha_{A,n}^*\mathcal{E}_n); \beta_A^* \mathcal{F}) \\
&\cong& \Hom_{\DD(A),\id_{T}}( \beta_{A,!}^{(T)} f_\bullet(\alpha_{A,1}^*\mathcal{E}_1, \dots, \alpha_{A,n}^*\mathcal{E}_n);  \mathcal{F}) \\
&\cong& \Hom_{\Dia^{\cor}(\DD),\id_{(J, T)}}( (J, \beta_{A,!}^{(T)} f_\bullet(\alpha_{A,1}^*\mathcal{E}_1, \dots, \alpha_{A,n}^*\mathcal{E}_n));  (J, \mathcal{F}))
\end{eqnarray*}
using (FDer0 left) and (FDer3 left).
One checks that this composition is induced by the composition in $\Dia^{\cor}(\DD)$ with a 1-morphism  in
\[  \Hom_{(x,f)}((I_1, \mathcal{E}_1), \dots, (I_n, \mathcal{E}_n); (J, \beta_{A,!}^{(T)} f_\bullet(\alpha_{A,1}^*\mathcal{E}_1, \dots, \alpha_{A,n}^*\mathcal{E}_n))), \]
which depends only on $(x, f)$ and not on $\mathcal{F}$.
Hence this 1-morphism is weakly coCartesian.

Note that we write $\Hom_{\Dia^{\cor}(\DD),(x,f)}$ for the category of 1-morphisms which map to $(x,f)$ in $\Dia^{\cor}(\SSS)$ and those 2-morphisms that map to $\id_{(x,f)}$ in $\Dia^{\cor}(\SSS)$. 
 
It remains to be shown that the composition of weakly coCartesian 1-morphisms is weakly coCartesian (cf.\@ \cite[Proposition~2.7]{Hor15b}). Let
$(y,g)$ be another 1-morphism, composable with $f$, with $y=(B; \alpha_{B,1}, \dots, \alpha_{B,m}; \beta_B)$ in $\Span^F(J_1, \dots, J_m; K)$ and
 \[ g \in \Hom_{\SSS(B)}(\alpha_{B,1}^*T_1, \dots, \alpha_{B,m}^*T_m; \beta_{B}^*U) . \]
Setting $J_i:=J$ and $T_i:=T$ the composition of $x$ and $y$ w.r.t.\@ the $i$-th slot is the correspondence:
 
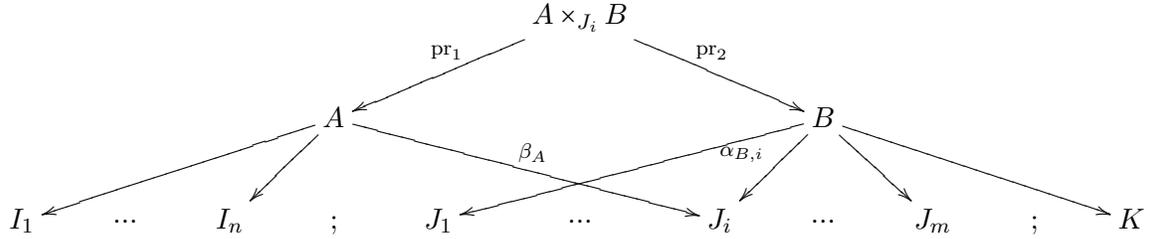
\[ \xymatrix{
&&&&& A \times_{J_i} B  \ar[lld]_{\pr_1} \ar[rrd]^{\pr_2} \\
&&& A \ar[llld] \ar[ld] \ar[rrrd]^{\beta_A} &&&&  B \ar[llld] \ar[ld]^{\alpha_{B,i}}  \ar[rd] \ar[rrrd] \\
I_1 & \cdots & I_n & ; & J_1 & \cdots & J_i & \cdots & J_m & ; & K
} \]
The composition of $g$ and $f$ is the morphism
\[ \pr_2^*g \circ_{i} \pr_1^* f  \]
\[ \in  \Hom(\pr_2^*\alpha_{B,1}^*T_1,\dots,\underbrace{\pr_1^*\alpha_{A,1}^*S_1, \dots, \pr_1^*\alpha_{A,n}^*S_n}_{\text{at }i},\dots,\pr_2^*\alpha_{B,m}^*T_{m}; \pr_2^*\beta_{B}^*U)  \]
We have to show that the natural map
\begin{eqnarray*} 
&&\beta_{B,!}^{(U)} g_\bullet(\alpha_{B,1}^*\mathcal{F}_{1}, \dots,  \underbrace{\alpha_{B,i}^* \beta_{A,!}^{(T_i)} f_\bullet (\alpha_{A,1}^*\mathcal{E}_1, \dots, \alpha_{A,n}^*\mathcal{E}_n)}_{\text{at $i$}}, \dots  \alpha_{B,m}^*\mathcal{F}_m) \\
&\rightarrow &\beta_{B,!}^{(U)} \pr_{2,!} (\pr_2^*g \circ_{i} \pr_1^* f)_\bullet  (\pr_2^* \alpha_{B,1}^*\mathcal{F}_{1}, \dots, \underbrace{\pr_1^* \alpha_{A,1}^* \mathcal{E}_1, \dots, \pr_1^* \alpha_{A,n}^*\mathcal{E}_n}_{\text{at $i$}}, \dots \pr_2^* \alpha_{B,m}^*\mathcal{F}_m)
 \end{eqnarray*}
 is an isomorphism.  It is the composition of the following morphisms which are all isomorphisms respectively by (FDer4 left) using \cite[Proposition~2.3.23.\@ 2.]{Hor15}, by (FDer5 left) observing that $\pr_2$ is a opfibration, by the second part of (FDer0 left) for $\pr_1$, and by the first part of (FDer0 left) in the form that the composition of coCartesian morphisms is coCartesian.
\begin{eqnarray*} 
&&\beta_{B,!} g_\bullet(\alpha_{B,1}^*\mathcal{F}_{1}, \dots,  \underbrace{\alpha_{B,i}^* \beta_{A,!}^{(T_i)} f_\bullet (\alpha_{A,1}^*\mathcal{E}_1, \dots, \alpha_{A,n}^*\mathcal{E}_n)}_{\text{at $i$}}, \dots,  \alpha_{B,m}^*\mathcal{F}_m) \\
&\rightarrow& \beta_{B,!} g_\bullet(\alpha_{B,1}^*\mathcal{F}_{1}, \dots,  \underbrace{\pr_{2,!}^{(\alpha_{B,i}^* T_i)} \pr_1^* f_\bullet (\alpha_{A,1}^*\mathcal{E}_1, \dots, \alpha_{A,n}^*\mathcal{E}_n)}_{\text{at $i$}}, \dots,  \alpha_{B,m}^*\mathcal{F}_m) \\
&\rightarrow& \beta_{B,!} \pr_{2,!}^{(\alpha_{B,i}^* T_i)}  (\pr_{2}^*g)_\bullet(\pr_{2}^* \alpha_{B,1}^*\mathcal{F}_{1}, \dots,  \underbrace{\pr_1^* f_\bullet (\alpha_{A,1}^*\mathcal{E}_1, \dots, \alpha_{A,n}^*\mathcal{E}_n)}_{\text{at $i$}}, \dots,  \pr_{2}^*\alpha_{B,m}^*\mathcal{F}_m) \\
&\rightarrow& \beta_{B,!} \pr_{2,!}^{(\alpha_{B,i}^* T_i)}  (\pr_{2}^*g)_\bullet(\pr_{2}^* \alpha_{B,1}^*\mathcal{F}_{1}, \dots,  \underbrace{ (\pr_1^*f)_\bullet (\pr_1^* \alpha_{A,1}^*\mathcal{E}_1, \dots, \pr_1^*\alpha_{A,n}^*\mathcal{E}_n)}_{\text{at $i$}}, \dots,  \pr_{2}^*\alpha_{B,m}^*\mathcal{F}_m) \\
&\rightarrow& \beta_{B,!} \pr_{2,!}^{(\alpha_{B,i}^* T_i)}  (\pr_2^*g \circ_{i} \pr_1^* f)_\bullet  (\pr_2^* \alpha_{B,1}^*\mathcal{F}_{1}, \dots, \underbrace{ \pr_1^* \alpha_{A,1}^* \mathcal{E}_1, \dots, \pr_1^* \alpha_{A,n}^*\mathcal{E}_n}_{\text{at $i$}}, \dots, \pr_2^* \alpha_{B,m}^*\mathcal{F}_m).
 \end{eqnarray*}
 
 Now we proceed to prove the converse, hence we assume that $\Dia^{\cor}(\DD) \rightarrow \Dia^{\cor}(\SSS)$ is a 1-opfibration and have to show all axioms of a left fibered derivator:

(FDer0 left) 
 First we have an obvious pseudo-functor of 2-multicategories
 \begin{eqnarray*}
  F: \SSS(A) &\hookrightarrow& \Dia^{\cor}(\SSS) \\
   S & \mapsto & (A, S) \\
    f & \mapsto & [f] 
 \end{eqnarray*}
 By \cite[Proposition~2.26]{Hor15b} the pull-back $F^*\Dia^{\cor}(\DD) \rightarrow \SSS(A)$ in the sense of \cite[2.25]{Hor15b} is 1-opfibered and 2-fibered if $\Dia^{\cor}(\DD) \rightarrow \Dia^{\cor}(\SSS)$ is 1-opfibered and 2-fibered. 
 To show that $\DD(I) \rightarrow \SSS(I)$ is a 1-opfibration and 2-fibration of multicategories, it thus suffices to show that the pull-back $F^*\Dia^{\cor}(\DD)$ is equivalent to $\DD(I)$ over $\SSS(I)$. 
 The class of objects of $F^*\Dia^{\cor}(\DD)$ is by definition isomorphic to the class of objects of $\DD(I)$. Therefore we are left to show that there are equivalences of categories (compatible with composition)
 \[ \Hom_{\DD(I),f}(\mathcal{E}_1, \dots, \mathcal{E}_n; \mathcal{F}) \rightarrow \Hom_{F^*\Dia^{\cor}(\DD),f}(\mathcal{E}_1, \dots, \mathcal{E}_n; \mathcal{F})  \]
 for any 1-morphism $f \in \Hom_{\SSS(I)}(S_1, \dots, S_n; T)$, where $\mathcal{E}_i$ is an object of $\DD(I)$ over $S_i$ and $\mathcal{F}$ is an object over $T$. 
Note that the left-hand side is a discrete category. 
We have a 2-Cartesian diagram of categories
\[ \xymatrix{
\Hom_{F^*\Dia^{\cor}(\DD),f}(\mathcal{E}_1, \dots, \mathcal{E}_n; \mathcal{F}) \ar[r] \ar[d] & \Hom_{\Dia^{\cor}(\DD)}((I,\mathcal{E}_1), \dots, (I, \mathcal{E}_n); (I, \mathcal{F})) \ar[d] \\
\{f\} \ar[r]^-F & \Hom_{\Dia^{\cor}(\SSS)}((I,S_1), \dots, (I,S_n); (I,T))
} \]
Since the right vertical morphism is a fibration (cf.\@ Lemma~\ref{PAREQSPANFIBERED}) the diagram is also Cartesian (cf.\@ \cite[Lemma~2.2]{Hor15b}). Futhermore by Lemma~\ref{PAREQSPANFIBERED} the right vertical morphism is equivalent to 
\[ \xymatrix{
  \tau_1(\Span_{\DD}((I, \mathcal{E}_1), \dots, (I, \mathcal{E}_n); (I, \mathcal{F}))) \ar[d] \\
  \tau_1(\Span_{\SSS}((I, S_1), \dots, (I, S_n); (I, T)))
} \]
(Here $\Span_{\DD}^F(\dots)$ was changed to $\Span_{\DD}(\dots)$ and similarly for $\Span_{\SSS}^F(\dots)$.)

In the category $\tau_1(\Span_{\SSS}((I_1, S_1), \dots, (I_n, S_n); (J, T)))$, the object $F(f)$ is isomorphic to $f$ over the trivial correspondence $(\id_I, \dots, \id_I; \id_I)$ whose fiber in $\tau_1(\Span_{\DD}((I, \mathcal{E}_1), \dots, (I, \mathcal{E}_n); (I, \mathcal{F})))$ is precisely the discrete category $\Hom_{\DD(I),f}(\mathcal{E}_1, \dots, \mathcal{E}_n; \mathcal{F})$.
The remaining part of (FDer0 left) will be shown below. 
 
Since we have a 1-opfibration and 2-fibration we can equivalently
see the given datum as a pseudo-functor
\[ \Psi: \Dia^{\cor}(\SSS) \rightarrow \mathcal{CAT} \]
and by Proposition~\ref{PROPBASICDIACORS}, $\Psi$ maps
$[\alpha^{(S)}]$ to a functor natural isomorphic to $\alpha^*: \DD(J)_S \rightarrow \DD(I)_{\alpha^*S}$. We have the freedom to 
choose $\Psi$ in such a way that it maps $[\alpha^{(S)}]$ precisely to $\alpha^*$. 

Axiom (FDer3 left) follows from Lemma~\ref{LEMMAPROPDIACORS}, 1.\@ stating that $[\alpha^{(S)}]$ has a left adjoint $[\alpha^{(S)}]'$ in the category $\Dia^{\cor}(\SSS)$ (cf.\@ also Proposition~\ref{PROPBASICDIACORS}). 

Axiom (FDer4 left) follows by applying $\Psi$ to the (first) 2-isomorphism of Lemma~\ref{LEMMAPROPDIACORS}, 2.

Axiom (FDer5 left) follows by applying $\Psi$ to the 2-isomorphism of Lemma~\ref{LEMMAPROPDIACORS}, 6.

The remaining part of (FDer0 left), i.e.\@ that $\alpha^*$ maps coCartesian arrows to coCartesian arrows follows by applying $\Psi$ to the 2-isomorphism of Lemma~\ref{LEMMAPROPDIACORS}, 5.
\end{proof}

\section{Derivator six-functor-formalisms}\label{DERSSF}

The main purpose for introducing the more general notion of fibered multiderivator over pre-2-multiderivators (as opposed to those over usual pre-multiderivators) is that
it provides the right framework to study any kind of {\em derived} six-functor-formalism.

We recall from \cite[Definition~3.1]{Hor15b}:

\begin{DEF}\label{DEFSCOR}
We define the {\bf 2-multicategory $\mathcal{S}^{\cor}$ of correspondences in $\mathcal{S}$}  to be the following 2-multicategory. 
\begin{enumerate}
\item The objects are the objects of $\mathcal{S}$.
\item The 1-morphisms in $\Hom(S_1, \dots, S_n; T)$ are the (multi)correspondences\footnote{as usual, $n=0$ is allowed.}
\begin{equation}\label{excor2}
 \vcenter{ \xymatrix{ &&&U \ar[dlll]_{\alpha_1} \ar[dl]^{\alpha_n}  \ar[dr]^{\beta} \\ 
S_1 & \cdots & S_n &;& T.
} } \end{equation}

\item The 2-morphisms $(U, \alpha_1, \dots, \alpha_n, \beta) \Rightarrow (U', \alpha_1', \dots, \alpha_n', \beta')$ are the {\em iso}morphisms $\gamma: U \rightarrow U'$ such that in
\begin{equation}\label{scor2morph} \vcenter{ \xymatrix{ 
&&&U \ar[dlll]_{\alpha_1} \ar[dl]^{\alpha_n} \ar[dd]^\gamma \ar[dr]^{\beta} \\ 
S_1 & \cdots & S_n && T \\
&&&U' \ar[ulll]^{\alpha_1'} \ar[ul]_{\alpha_n'}  \ar[ur]_{\beta'}
} } \end{equation}
all triangles are commutative. 
\item The composition is given by the fiber product in the following way: the correspondence 
\[ \xymatrix{
&&&&& U \times_{T_i} V  \ar[lld] \ar[rrd] \\
&&& U \ar[llld] \ar[ld] \ar[rrrd]^{\beta_U} &&&&  V \ar[llld] \ar[ld]^{\alpha_{V,i}}  \ar[rd] \ar[rrrd] \\
S_1 & \cdots & S_n & ; & T_1 & \cdots & T_i & \cdots & T_m & ; & W
} \]
in $\Hom(T_1, \dots, T_{i-1}, S_1, \dots, S_n, T_{i+1},  \dots, T_m; W)$
is the composition w.r.t.\@ the $i$-th slot of the left correspondence in  $\Hom(S_1, \dots, S_n; T_i)$ and
the right correspondence in $\Hom(T_1, \dots, T_m; W)$. 
\end{enumerate}
\end{DEF}

The 2-multicategory $\mathcal{S}^{\cor}$ is symmetric, representable
(i.e.\@ opfibered over $\{\cdot\}$), closed (i.e.\@ fibered over $\{\cdot\}$), every object is self-dual, with tensor product and internal Hom {\em both} given by the product $\times$ in $\mathcal{S}$ and having as unit
the final object of $\mathcal{S}$. 

\begin{DEF}
We define also the larger category $\mathcal{S}^{\cor,G}$ where in addition every morphism $\gamma: U \rightarrow U'$, such that in (\ref{scor2morph}) all triangles commute, is a 2-morphism (i.e.\@ $\gamma$ does not necessarily have to be an isomorphism). 
\end{DEF}

\begin{BEM}\label{BEMAFP}
In reality the above definition gives only bimulticategories because the formation of fiber products is only associative up to isomorphism. One can, however, enlarge the class of objects adjoining strictly associative fiber products.  
We sketch the precise construction of an equivalent 2-multicategory with is strictly symmetric as well. 
Consider the following class of objects, called {\bf abstract fiber products}. An abstract fiber product is a finite unoriented tree (in the sense of graph theory). Each vertex $v$ and each edge $e$ has an associated object $X_v$, resp.\@ $X_e$, in $\mathcal{S}$. 
For each edge $e$ there are morphisms $X_v \rightarrow X_e \leftarrow X_{v'}$ from the objects corresponding to the vertices of the edge. A {\em morphism} from such on object $X$ to an object $S \in \mathcal{S}$ is given 
by the choice of a vertex $v$ and a morphism $X_v \rightarrow S$. From a diagram $X \rightarrow S \leftarrow Y$ where $X$ and $Y$ are abstract fiber products and $S \in \mathcal{S}$, an abstract fiber product (called concatenation) can be formed,
adding a new edge with associated object $S$ to the disjoint union of $X$ and $Y$. Each such abstract fiber product has a {\em reduced form}\footnote{This construction has been stated first in a preliminary version of \cite{Hor17} with a wrong reduction procedure including erroneously also reductions at vertices of higher order. We thank René Recktenwald for pointing out the mistake. } in which repeatedly every subgraph of the form 
\[ X_v \rightarrow X_e = X_{v'} \rightarrow X_{e'}  \leftarrow X_{v''} \]
(where $e$ and $e'$ are the only edges at $v'$) is replaced by
\[ X_v \rightarrow X_{e'} \leftarrow X_{v''}   \]
in which the morphism $X_{v} \rightarrow X_{e'}$ is the obvious composition and every subgraph of the form 
\[ X_v \rightarrow X_e = X_{v'}  \]
(where $e$ is the only edge at $v'$)  is replaced by
\[ X_v  \]
In both cases each morphism of the abstract fiber product to some $S \in \mathcal{S}$ corresponding to a morphism $X_{v'} \rightarrow S$ is translated to the composition with the morphism  $X_v \rightarrow X_e = X_{v'}$ and goes out of $X_v$. 

With each abstract fiber product $X$ one can associate an actual fiber product in $\mathcal{S}$ (the limit over $X$ seen as the obvious diagram of the $X_e$ and $X_v$), and with a morphism $X \rightarrow S$ the obvious projection $\lim X \rightarrow S$. Concatenation corresponds (up to unique isomorphism) to the formation 
of fiber product. 

With this notion of abstract fiber product, in the definition of 1-morphisms of the symmetric 2-multicategory $\mathcal{S}^{\cor}$, we replace the object $U$ by an abstract fiber product and the morphisms to $S_1, \dots, S_n$, and $T$ are understood in the sense above. The composition is given by concatenation and subsequent reduction. This operation is now strictly associative and has strict units. There is still an operation of the symmetric groups strictly compatible with composition.
To define the sets of 2-morphisms we choose an actual fiber product ``$\lim X$''  for each abstract fiber product $X$. This operation maps each new 1-multimorphism to a usual multicorrespondence (\ref{excor2}) and the sets of 2-morphisms are defined to be the isomorphisms between these multicorrespondences. Similarly, the category $\mathcal{S}^{\cor, G}$ is defined allowing all  morphisms between these multicorrespondences. 

Finally, there are now {\em strict} functors $\mathcal{S} \rightarrow \mathcal{S}^{\cor}$, and $\mathcal{S}^{\op} \rightarrow \mathcal{S}^{\cor}$, and for each $S \in \mathcal{S}$ a {\em strict} functor $\{ \cdot \} \rightarrow \mathcal{S}^{\cor}$ from the
final multicategory with value $S$. 
\end{BEM}

\begin{PAR}
Definition~\ref{DEFSCOR} can be generalized to the case of a general opmulticategory (cf.\@ \cite[1.2]{Hor15b}) $\mathcal{S}$ which has multipullbacks in the following sense: Given a multimorphism $T \rightarrow S_1, \dots, S_n$  and a morphism $S_i' \rightarrow S_i$ for some $1 \le i \le n$,  
a {\bf multipullback}  is a universal square of the form
\[ \xymatrix{
T' \ar[r] \ar[d] & S_1, \dots, S_i', \dots, S_n \ar[d] \\
T \ar[r] & S_1, \dots, S_i, \dots, S_n.
} \]

A usual category  $\mathcal{S}$ becomes an opmulticategory setting 
\begin{equation}\label{opmulti} \Hom(T; S_1, \dots, S_n) := \Hom(T, S_1) \times \cdots \times \Hom(T, S_n). \end{equation}
In case that a usual category $\mathcal{S}$ has pullbacks it automatically has multipullbacks w.r.t.\@ opmulticategory structure given by (\ref{opmulti}). Those are given by
 Cartesian squares
\[ \xymatrix{
T' \ar[r] \ar[d] &  S_i'  \ar[d] \\
T \ar[r] & S_i.
} \]
For any opmulticategory $\mathcal{S}$ with multipullbacks we define $\mathcal{S}^{\cor}$ to be the 2-category whose objects are the objects of $\mathcal{S}$, whose 1-morphisms
are the multicorrespondences of the form
\[ \xymatrix{
& U \ar[rd] \ar[ld]  \\
S_1, \dots, S_n &  & T \\
} \]
and whose 2-morphisms are commutative diagrams of multimorphisms
\[ \xymatrix{
& U \ar[rd] \ar[ld] \ar[dd]^\gamma \\
S_1, \dots, S_n &  & T. \\
& U' \ar[ru] \ar[lu]  
} \]
in which $\gamma$ is an isomorphism. 
The composition is given by forming the multipullback. 
The reader may check that if the opmulticategory structure on $\mathcal{S}$ is given by (\ref{opmulti}) we reobtain the 2-multicategory $\mathcal{S}^{\cor}$ defined in \ref{DEFSCOR}.
\end{PAR}

\begin{PAR}
Recall from \cite[]{Hor15b} the following:
\end{PAR}
\begin{DEF}\label{DEF6FU}
Let $\mathcal{S}$ be a opmulticategory with multipullbacks. A {\bf (symmetric) six-functor-formalism} on $\mathcal{S}$ is a 1-bifibered and 2-bifibered (symmetric) 2-multicategory with 1-categorical fibers
\[ p: \mathcal{D} \rightarrow \mathcal{S}^{\cor}. \]
A {\bf (symmetric) Grothendieck context} on $\mathcal{S}$ is a 1-bifibered and 2-opfibered (symmetric) 2-multicategory with 1-categorical fibers
\[ p: \mathcal{D} \rightarrow \mathcal{S}^{\cor,G}. \]
A {\bf (symmetric) Wirthm\"uller context} on $\mathcal{S}$ is a 1-bifibered and 2-fibered (symmetric) 2-multicategory with 1-categorical fibers
\[ p: \mathcal{D} \rightarrow \mathcal{S}^{\cor,G}. \]
\end{DEF}

\begin{PAR}\label{PARS0}
Consider a class of ``proper'' (resp.\@ ``etale'') 1-ary morphisms $\mathcal{S}_0$ in $\mathcal{S}$, which abstractly can be any class satisfying the following properties:
\begin{enumerate}
\item Isomorphisms are in $\mathcal{S}_0$;
\item $\mathcal{S}_0$ is stable under composition;
\item $\mathcal{S}_0$ is stable under multipullback. 
\end{enumerate}
We define $\mathcal{S}^{\cor, 0}$ to be the subcategory of $\mathcal{S}^{\cor, G}$ with the same objects and 1-morphisms, but
where the morphisms $\gamma: U \rightarrow U'$ entering the definition of 2-morphism are the morphisms in $\mathcal{S}_0$. Then consider a 1-bifibration
\[ p: \mathcal{D} \rightarrow \mathcal{S}^{\cor, 0} \]
which is a 2-opfibration in the proper case and a 2-fibration in the etale case. We call this respectively a {\bf (symmetric) proper six-functor-formalism} and a {\bf (symmetric) etale six-functor-formalism}. 
\end{PAR}

We denote by  $\SSS^{\cor}$, $\SSS^{\cor, 0, \lax}$, and \@ $\SSS^{\cor, 0, \oplax}$, respectively  (cf.\@ \ref{PARREPR2PREMULTIDER}) the pre-2-multiderivators represented by 
$\mathcal{S}^{\cor}$ (resp. $\mathcal{S}^{\cor, 0}$ with choice of classes of proper or etale morphisms) and proceed to state the derivator version of the previous definition:

\begin{DEF}\label{DEF6FUDER}
\begin{enumerate}
\item A {\bf (symmetric) derivator six-functor-formalism} is a left and right fibered (symmetric) multiderivator
\[ \DD \rightarrow \SSS^{\cor}. \] 

\item A {\bf (symmetric) proper derivator six-functor-formalism} is as before together with an extension as oplax left fibered (symmetric) multiderivator 
\[ \DD' \rightarrow \SSS^{\cor, 0, \oplax},  \] 
and an extension as lax right fibered (symmetric) multiderivator 
\[ \DD'' \rightarrow \SSS^{\cor, 0, \lax}.  \] 

\item A {\bf (symmetric) etale derivator six-functor-formalism} is as before together with an extension as lax left fibered (symmetric) multiderivator 
\[ \DD' \rightarrow \SSS^{\cor, 0, \lax},  \] 
and an extension as oplax right fibered (symmetric) multiderivator 
\[ \DD'' \rightarrow \SSS^{\cor, 0, \oplax}.  \] 
\end{enumerate}
\end{DEF}
In particular, and in view of \cite[Section~8]{Hor15b}, if $\mathcal{S}^{\cor, 0} = \mathcal{S}^{\cor, G}$ is formed w.r.t.\@ the choice of {\em all morphisms}, we call a proper derivator six-functor-formalism a {\bf derivator Grothendieck context} and
an etale derivator six-functor-formalism a {\bf derivator Wirthm\"uller context}.

\begin{PAR}\label{DUALITY}
As mentioned, if $\SSS$ is really a pre-2-multiderivator, as opposed to a usual pre-multiderivator,
 the functor
\[ \Dia^{\cor}(\SSS) \rightarrow \Dia^{\cor}, \]
has hardly ever any fibration properties, because of the truncation involved in the definition of the categories of 1-morphisms. 
Nevertheless the composition
\[ \Dia^{\cor}(\SSS) \rightarrow \{\cdot\} \]
is often 1-bifibered, i.e.\@ there exists an absolute monoidal product on $\Dia^{\cor}(\SSS)$ extending the one on $\Dia^{\cor}$. 
For example, if $\mathcal{S}$ is a usual 1-category with finite limits equipped with the opmulticategory structure (\ref{opmulti}) then for the pre-2-multiderivator $\SSS^{\cor}$ represented by $\mathcal{S}^{\cor}$, we have on $\Dia^{\cor}(\SSS^{\cor})$ the monoidal product
\[ (I, F) \boxtimes (J, G) = (I \times J, F \times G) \]
where $F \times G$ is the diagram of correspondences in $\mathcal{S}$ formed by applying $\times$ point-wise. Similarly we have
\[ \mathbf{HOM}\left((I, F), (J, G)\right) = (I^{\op} \times J, F^{\op} \times G) \]
 where in $F^{\op}$ all correspondences are flipped. In particular any object $(I, F)$ in $\Dia^{\cor}(\SSS^{\cor})$ is dualizable with duality explicitly given by
 \[ \mathbf{HOM}\left((I, F), (\cdot, \cdot)\right) = (I^{\op}, F^{\op}).  \]
 where $\cdot$ denotes the final object of $\Dia$, and $\mathcal{S}$, respectively. 
 
 Given a derivator six-functor-formalism $\DD \rightarrow \SSS^{\cor}$ we get an external 
 monoidal product even on $\Dia^{\cor}(\DD)$ which prolongs the one on diagrams of correspondences.
\end{PAR}

\section{Construction --- Part I}\label{CONSTRUCTION}

\begin{PAR}\label{PARCONSTRUCTION}
In the remaining part of the article we formally {\em construct} a (symmetric) derivator six-functor-formalism in which $f_!=f_*$, i.e.\@ a derivator Grothendieck context, starting from a (symmetric) fibered multiderivator $\DD \rightarrow \SSS^{\op}$. The construction on the underlying bifibration of multicategories $\DD(\cdot) \rightarrow \mathcal{S}^{\op}$ is very simple: By \cite[Proposition~2.17]{Hor15b} this bifibration may be seen as a pseudo-functor of multicategories
\[ \Phi: \mathcal{S}^{\op} \rightarrow \mathcal{CAT} \]
where the  functors in the image have right adjoints. (For the construction it suffices that the images of 1-ary morphisms have right adjoints.) We may define a pseudo-functor
\[ \mathcal{S}^{\cor, G} \rightarrow \mathcal{CAT} \]
mapping objects in the same way as $\Phi$ and mapping a multicorrespondence 
\[ \xymatrix{
& U \ar[rd]^f \ar[ld]_g  \\
S_1, \dots, S_n &  & T \\
} \]
to $f_* g^*$ where $g^*$ is the image of $g$ under the pseudo-functor $\Phi$ and $f_*$ is the right adjoint of the image of $f$. It has been shown in the discussion after \cite[Definition~3.12]{Hor15b} that this is well-defined if $\DD(\cdot) \rightarrow \mathcal{S}^{\op}$ satisfies multi-base-change in the sense of Definition~\ref{DEFMULTIBASECHANGE}.
\end{PAR}

\begin{HAUPTSATZ}\label{MAINTHEOREM1}
Let $\mathcal{S}$ be a (symmetric) opmulticategory with multipullbacks and let $\SSS^{\op}$ be the (symmetric) pre-multiderivator  represented by $\mathcal{S}^{\op}$ . 
Let $\DD \rightarrow \SSS^{\op}$ be a (symmetric) left and right fibered multiderivator such that the following holds:
\begin{enumerate}
\item 
The pullback along 1-ary morphisms (i.e.\@ pushforward along 1-ary morphisms in $\mathcal{S}$) commutes also with homotopy {\em co}limits.
\item 
In the underlying bifibration $\DD(\cdot) \rightarrow \SSS(\cdot)$ multi-base-change holds in the sense of Definition~\ref{DEFMULTIBASECHANGE}. 
\end{enumerate}
Then there exists a (symmetric) oplax left fibered multiderivator
\[ \EE \rightarrow \SSS^{\cor,G,\oplax}  \]
satisfying the following properties
\begin{itemize}
\item[a)] 
The corresponding (symmetric) 1-opfibration, and 2-opfibration of 2-multicategories with 1-categorical fibers
\[ \EE(\cdot) \rightarrow \SSS^{\cor,G,\oplax}(\cdot) = \mathcal{S}^{\cor,G}  \]
is just (up to equivalence) obtained from $\DD(\cdot) \rightarrow \mathcal{S}^{\op}$ by the procedure described in \ref{PARCONSTRUCTION}.
\item[b)] For every $S \in \mathcal{S}$ there is a canonical equivalence of fibers: 
\[ \EE_S \cong \DD_S  \]
and if those are stable, {\em all} fibers of $\EE \rightarrow \SSS^{\cor,G}$ are stable derivators with domain $\Posf$. 
\end{itemize}
\end{HAUPTSATZ}

Using standard theorems on Brown representability \cite[Section 3.1]{Hor15} we can refine this.
\begin{HAUPTSATZ}\label{MAINTHEOREM2}Let $\Dia$ be an infinite diagram category \cite[Definition 1.1.1]{Hor15} which contains all finite posets. 
Let $\mathcal{S}$ be a (symmetric) opmulticategory with multipullbacks and let $\SSS$ be the corresponding represented (symmetric) pre-multiderivator with domain $\Dia$. 
Let $\DD \rightarrow \SSS^{\op}$ be an {\em infinite} (symmetric) left and right fibered multiderivator with domain $\Dia$ satisfying conditions 1.\@ and 2.\@ of Theorem~\ref{MAINTHEOREM1}, {\em with stable, perfectly generated fibers} (cf.\@ Definition~\ref{DEFSTABLE} and \cite[Section 3.1]{Hor15}).

Then the restriction of the left fibered multiderivator $\EE$ from Theorem~\ref{MAINTHEOREM1} is a (symmetric) left {\em and right} fibered multiderivator with domain $\Dia$
\[ \EE|_{\SSS^{\cor}} \rightarrow \SSS^{\cor} \]
and has an extension as a (symmetric) {\em lax} right fibered multiderivator with domain $\Dia$
\[ \EE' \rightarrow \SSS^{\cor,G,\lax}. \]
In other words, we get a (symmetric) derivator Grothendieck context in the sense of Section~\ref{DERSSF}. 
\end{HAUPTSATZ}

We begin by explaining the construction of $\EE$. In the second part \cite{Hor17} we will give a much more general construction of arbitrary derivator six-functor-formalisms which, however, is more abstract and indirect and obscures the relatively simple construction in this special case (i.e.\@ the case of a derivator Grothendieck context with $f_*=f_!$). 
We need some preparation:

\begin{PAR}\label{PARTW}
Let $I$ be a diagram, $n$ a natural number and $\Xi = (\Xi_1, \dots, \Xi_n) \in \{ \uparrow, \downarrow \}^n$ be a sequence of arrow directions. We define a diagram
\[ {}^\Xi I \]
whose objects are sequences of $n-1$ morphisms in $I$
\[ \xymatrix{
i_1 \ar[r] & i_2 \ar[r] & \cdots \ar[r]  & i_n
} \]
and whose morphisms are commutative diagrams
\[ \xymatrix{
i_1 \ar[r] \ar@{<->}[d] &i_2 \ar[r] \ar@{<->}[d] & \cdots \ar[r]  & i_n \ar@{<->}[d] \\
i_1' \ar[r] &i_2' \ar[r] & \cdots \ar[r]  & i_n' \\
} \]
in which the $j$-th vertical arrow goes in the direction indicated by $\Xi_j$. 
We call such morphisms {\bf of type $j$} if at most the morphism $i_j \rightarrow i_j'$ is {\em not} an identity. 
{\em From now on we assume that $\Dia$ permits this construction for any $I \in \Dia$, i.e.\@ if $I \in \Dia$ then also ${}^\Xi I \in \Dia$ for every finite $\Xi$.}
\end{PAR}

\begin{BEISPIEL}
\begin{eqnarray*}
{}^{\downarrow} I  & = & I \\
{}^{\uparrow} I  & = & I^{\op} \\
{}^{\downarrow\downarrow} I  & = & I \times_{/I} I \\
{}^{\uparrow\downarrow} I  & = & \mathrm{tw}(I) 
\end{eqnarray*}
where $I \times_{/I} I$ is the comma category (the arrow category of $I$) and $\mathrm{tw}(I)$ is called the {\em twisted arrow category}.  
\end{BEISPIEL}

\begin{PAR}\label{PARTW2}
For any ordered subset $\{i_1, \dots, i_m\} \subseteq \{1, \dots, n\}$, denoting $\Xi'$ the restriction of $\Xi$ to the subset, we get an obvious restriction functor
\[ \pi_{i_1, \dots, i_m}:\  {}^\Xi I \rightarrow {}^{\Xi'} I. \]

If $\Xi = \Xi' \circ \Xi'' \circ \Xi'''$, where $\circ$ means concatenation, then the projection
\[ \pi_{1,\dots,n'}: \ {}^{\Xi}  I \rightarrow {}^{\Xi'} I   \]
is a {\em fibration} if the last arrow of $\Xi'$ is $\downarrow$ and an {\em opfibration} if the last arrow of $\Xi'$ is $\uparrow$ while the projection
\[ \pi_{n-n'''+1,\dots,n}: \  {}^{\Xi} I \rightarrow {}^{\Xi'''} I   \]
is an {\em opfibration} if the first arrow of $\Xi'''$ is $\downarrow$ and a {\em fibration} if the first arrow of $\Xi'''$ is $\uparrow$.
\end{PAR}

\begin{PAR}\label{PARTW3}
A functor $\alpha: I \rightarrow J$ induces an obvious functor
\[ {}^\Xi \alpha: {}^\Xi I \rightarrow {}^\Xi J. \]

A natural transfomation $\mu: \alpha \Rightarrow \beta$ induces functors
\[  ({}^\Xi \mu)_0, \dots, ({}^\Xi \mu)_n: {}^{\Xi} I \rightarrow {}^\Xi J   \]
with $({}^\Xi \mu)_0 = {}^\Xi \alpha$, and $({}^\Xi \mu)_n = {}^\Xi  \beta$,
defined by mapping an object $\xymatrix{ i_1 \ar[r]^{\nu_1} & i_2 \ar[r] & \dots \ar[r]^{\nu_{n-1}} & i_n}$ of ${}^\Xi I$ to the sequence:
\[ \xymatrix{
\alpha(i_1) \ar[r]  & \cdots \ar[r]  &  \alpha(i_{n-j}) \ar[rd]^{\ \ \beta(\nu_{n-j}) \circ \mu(i_{n-j})} \\
& & & \beta(i_{n-j+1}) \ar[r] & \cdots \ar[r]  & \beta(i_n) \\
} \]
There is a sequence of natural transformations
\[ {}^\Xi  \alpha=({}^\Xi \mu)_0 \Leftrightarrow \cdots \Leftrightarrow ({}^\Xi \mu)_n= {}^\Xi \beta \]
where the natural transformations at position $i$ (the count starting with 0) goes to the right if $\Xi_{n-i}=\downarrow$ and to the left if $\Xi_{n-i} = \uparrow$. 
Furthermore, the natural transformation at position $i$ consists element-wise of morphisms of type $n-i$.  

\end{PAR}

\begin{PAR}
Let $S: I \rightarrow \mathcal{S}^{\cor}$ be a pseudo-functor. 
We can associate to it a natural functor $S': \tw{I} \rightarrow \mathcal{S}$
such that for each composition of three morphisms $\gamma \beta \alpha$
the commutative diagram
\begin{equation}\label{bla} 
\vcenter{
\xymatrix{
\gamma \beta \alpha \ar[r] \ar[d] &  \beta \alpha  \ar[d] \\
\gamma \beta  \ar[r] & \beta 
}}
\end{equation}
in $\tw{I}$ is mapped to a Cartesian square in $\mathcal{S}$. We call such diagrams {\bf admissible}. 
Note that the horizontal morphisms are of type 2 and the vertical ones of type 1. 
Conversely every square in $\tw{I}$ with these properties has the above form. 

The construction of $S'$ is as follows. $S$ maps a morphism $\nu$ in $I$ to a correspondence in $\mathcal{S}$
\[ X_\nu \longleftarrow A_\nu \longrightarrow Y_\nu, \]
and we define $S'(\nu) := A_\nu$. A morphism $\xi: \nu \rightarrow \mu$ defined by
\[ \xymatrix{
i \ar[r]^\nu \ar[d]_\alpha & j  \\
k \ar[r]_\mu &   \ar[u]_\beta l
} \]
induces, by definition of the composition in $\mathcal{S}^{\cor}$, a commutative diagram in which all squares are Cartesian: 
\[ \xymatrix{ 
& & & A_{\nu} \ar[dl] \ar[rd] & & \\
& & A_{\mu \alpha} \ar[dl] \ar[rd] & & A_{\beta \mu} \ar[dl] \ar[rd] & & \\
& A_\alpha \ar[dl] \ar[rd] & & A_{\mu} \ar[dl] \ar[rd] & & A_{\beta} \ar[dl] \ar[rd] & \\
X  & & Y  && Z && W
 } \]
 We define $S'(\xi)$ to be the induced morphism $A_{\nu} \rightarrow A_{\mu}$. Note that the square of the form (\ref{bla}) is just mapped to the
 upper square in the above diagram, thus to a Cartesian square. Hence the so defined functor $S'$ is admissible. 
\end{PAR}

\begin{PAR}\label{PARALTDIASCOR}\label{ALTSCORPREDER}
A multimorphism 
\[ T \longrightarrow S_1, \dots, S_n \]
of admissible diagrams in $\SSS(\tw{I})$ is called {\bf type $i$ admissible} ($i=1,2$), if for any morphism $\xi: \nu \rightarrow \mu$ in $\tw{I}$ of type $i$
the diagram
\[ \xymatrix{
T(\nu) \ar[r] \ar[d] & (S_1(\nu), \dots, S_n(\nu)) \ar[d] \\
T(\mu) \ar[r] & (S_1(\mu),  \dots, S_n(\mu))
} \]
is a multipullback. 

A multimorphism $(X_1, \dots, X_n) \rightarrow Y$ in $\mathcal{S}^{\cor}(I)$ can be seen equivalently as a multicorrespondence of admissible diagrams in $\mathcal{S}(\tw{I})$
\[ \xymatrix{
& A \ar[rd]^f \ar[ld]_g \\
(X_1, \dots, X_n) & & Y
} \]
where $f$ is type 2 admissible and $g$ is type 1 admissible. 
In this description, the 2-morphisms are the commutative diagrams 
\[ \xymatrix{ 
& A \ar[dd]^h \ar[dl] \ar[rd] \\
(X_1, \dots, X_n)    &&Y   \\
& A' \ar[ur] \ar[lu] 
} \]
where the morphism $h$ is an isomorphism. 

In this way, we see that the 2-multicategory $\mathcal{S}^{\cor}(I)$ is equivalent to the 2-multicategory having as objects admissible diagrams $\tw{I} \rightarrow \mathcal{S}$ with the 1-multimorphisms and 2-morphisms described above. 
\end{PAR}

\begin{LEMMA} \label{LEMMATYPE12}
Type $i$ admissible morphisms $S \rightarrow T$ between admissible diagrams $S, T \in \SSS(\tw{I})$ satisfy the following property:

If $h_3 = h_2 \circ h_1$ and $h_2$ is type $i$-admissible then $h_1$ is type $i$ admissible if and only if $h_3$ is type $i$ admissible. 
\end{LEMMA}
\begin{proof}
This follows immediately from the corresponding property of Cartesian squares. 
\end{proof}

\begin{PAR}
The discussion in \ref{PARALTDIASCOR} has an (op)lax variant. Recall the definition of the category (value of the represented (op)lax pre-2-multiderivator) $\SSS^{\cor, G, \lax}(I)$ (resp.\@ $\SSS^{\cor, G, \oplax}(I)$), of pseudo-functors, {\em (op)lax} natural transformations, and modifications. 
A {\em lax} multimorphism of pseudo-functors 
\[ (X_1, \dots, X_n) \longrightarrow Y \]
 can be equivalently seen as a multicorrespondence  of admissible diagrams in $\SSS(\tw{I})$
\[ \xymatrix{
& A \ar[rd]^f \ar[ld]_g \\
(X_1, \dots, X_n) & & Y
} \]
where $g$ is type 1 admissible and $f$ is {\em arbitrary}.  Similarly an {\em oplax} multimorphism can be seen as such a multicorrespondence in which $g$ is {\em arbitrary} and $f$ is type 2 admissible. 
In the 2-morphisms the morphism $h$ can be an arbitrary morphism, which is automatically type 1 admissible in the lax case and type 2 admissible in the oplax case (cf.\@ Lemma~\ref{LEMMATYPE12}). 
\end{PAR}

\begin{PAR}
A natural transformation $\mu: \alpha \Rightarrow \beta$  gives rise (cf.\@ \ref{PARTW3}) to a sequence of natural transformations
\[ (\tw{\alpha}) \Leftarrow (\tw{\mu})_1 \Rightarrow (\tw{\beta}). \]
For any admissible diagram $S: \tw{I} \rightarrow \mathcal{S}$ this defines a diagram
\[ \xymatrix{
& (\tw{\mu})_1^*S \ar[rd]^{f_S} \ar[ld]_{g_S} \\
(\tw{\alpha})^*S & & (\tw{\beta})^*S
} \]

in which the morphism $f_S$ is type 2 admissible and the morphism $g_S$ is type 1 admissible. This defines a 1-morphism 
\[ (\tw{\alpha})^*S \rightarrow (\tw{\beta})^*S \]
in the alternative description (cf.\@ \ref{PARALTDIASCOR}) of $\SSS^{\cor}(I)$. 
For any admissible diagram $S$ this defines a pseudo-functor $\alpha \mapsto \alpha^* S$ from the category of functors $\Hom(I, J)$ to the 2-category $\SSS^{\cor}(I)$.
\end{PAR}

\begin{PAR}
Let $I$ be a diagram.
Consider the category $\twc{I}$ defined in \ref{PARTW}. Recall that its objects are compositions of two morphisms in $I$ and its morphisms $\nu \rightarrow \mu$ are commutative diagrams
\[ \xymatrix{
i \ar[r]^{\nu_1} \ar[d] & j \ar[r]^{\nu_2}& k  \ar[d]  \\
i' \ar[r]_{\mu_1} & j'  \ar[r]_{\mu_2} \ar[u] & k' 
} \]
\end{PAR}
\begin{PAR}
If $\alpha: I \rightarrow J$ is an opfibration and we form the pull-back
\[ \xymatrix{
\twc{J} \times_J I  \ar[r] \ar[d] & I \ar[d]^{\alpha} \\
\twc{J} \ar[r]_{\pi_1} & J
} \]
and
\[ \xymatrix{
\twc{J} \times_{\tw{J}} \tw{I}  \ar[r] \ar[d] & \tw{I} \ar[d]^{\tw{\alpha}} \\
\twc{J} \ar[r]_{\pi_{12}} & \tw{J}
} \]
then obviously the left vertical functors are opfibrations as well.
\end{PAR}

\begin{LEMMA}\label{LEMMAOPFIB}
Let $\alpha: I \rightarrow J$ be an opfibration, and consider the sequence defined by the universal property of pull-backs
\[ \xymatrix{ \twc{I} \ar[r]^-{q_1} & \twc{J} \times_{(\tw{J})} \tw{I} \ar[r]^-{q_2} & \twc{J} \times_J I. } \]
\begin{enumerate}
\item The functor $q_1$ is an opfibration. The fiber of $q_1$ over a pair $j_1 \rightarrow j_2 \rightarrow j_3$ and $i_1 \rightarrow i_2$ is
\[ i_3 \times_{/I_{j_3}} I_{j_3}\]
where $i_3$ is the target of a coCartesian arrow over $j_2 \rightarrow j_3$ with source $i_2$. 
\item The functor $q_2$ is a fibration. The fiber of $q_2$ over a pair $j_1 \rightarrow j_2 \rightarrow j_3$ and $i_1$ is 
\[ (i_2 \times_{/I_{j_2}} I_{j_2})^{\op} \]
where $i_2$ is the target of a coCartesian arrow over $j_1 \rightarrow j_2$ with source $i_1$. 
\end{enumerate}
\end{LEMMA}

\begin{proof}Straightforward. \end{proof}

Recall the following definition from \cite[Definition 2.4.1]{Hor15}, in which $\SSS$ (generalizing slightly the definition of \cite{Hor15}) can be any pre-2-multiderivator.  

\begin{DEF}\label{DEFCOCART}
Let $\DD \rightarrow \SSS$ be a right (resp.\@ left) fibered (multi-)derivator of domain $\Dia$. 
Let $I, E \in \Dia$ be diagrams and let $\alpha: I \rightarrow E$ be a functor in $\Dia$. 
We say that an object \[ \mathcal{E} \in \DD(I)\] is $E$-{\bf (co-)Cartesian}, if for any morphism
$\mu: i \rightarrow j$ in $I$ mapping to an identity in $E$, the corresponding morphism
$\DD(\mu): i^*\mathcal{E} \rightarrow j^*\mathcal{E}$ is (co-)Cartesian. 

If $E$ is the trivial category, we omit it from the notation, and talk about (co-)Cartesian objects. 
\end{DEF}

These notions define full subcategories $\DD(I)^{E-\mathrm{cart}}$ (resp.\@ $\DD(I)^{E-\mathrm{cocart}}$) of $\DD(I)$, and $\DD(I)_S^{E-\mathrm{cart}}$ (resp.\@  $\DD(I)_S^{E-\mathrm{cocart}}$) of $\DD(I)_S$ for any $S \in \SSS(I)$.
If we want to specify the functor $\alpha$, we speak about $\alpha$-(co)Cartesian objects and denote these e.g.\@ by $\DD(I)_S^{\alpha-\mathrm{cart}}$.

\begin{DEF}\label{DEFE}
Let $\mathcal{S}$ be an opmulticategory with multipullbacks and let $\SSS^{\op}$ be the  pre-multiderivator represented by $\mathcal{S}^{\op}$. 
Let $\DD \rightarrow \SSS^{\op}$ be a (left and right) fibered multiderivator such that 
conditions 1.\@ and 2.\@ of Theorem~\ref{MAINTHEOREM1} hold true.

We define the morphism of pre-2-multiderivators $\EE \rightarrow \SSS^{\cor}$ of Theorem~\ref{MAINTHEOREM1}.  
The pre-2-multiderivator $\EE$ is defined as follows: 
A diagram $I$ is mapped to a 1-opfibered, and 2-opfibered multicategory with 1-categorical fibers  $\EE(I) \rightarrow \SSS^{\cor,G,\oplax}(I)$.
We will specify this by giving the pseudo-functor of 2-multicategories
\[   \SSS^{\cor,G,\oplax}(I)^{2-\op} \rightarrow \mathcal{CAT}  \]
where we understand $\SSS^{\cor,G,\lax}(I)$ in the form described in \ref{ALTSCORPREDER}.  
An oplax pseudo-functor $I \rightarrow \mathcal{S}^{\cor,G}$ represented by an admissible diagram $S: \tw{I} \rightarrow \mathcal{S}$ is mapped to the category

\vspace{0.2cm}
\fbox{\begin{minipage}{\textwidth}
\[  \EE(I)_S := \DD(\twc{I})_{\pi_{23}^* (S^{\op})}^{\pi_{12}-\cocart, \pi_{13}-\cart}  \]
\end{minipage}}\vspace{0.2cm}
(cf.\@ Definition~\ref{DEFCOCART}). Note that $({}^{\downarrow \uparrow } I )^{\op} = {}^{\uparrow \downarrow} I$. 

A multicorrespondence 
\[ \xymatrix{
& A \ar[rd]^f \ar[ld]_g \\
(S_1, \dots, S_n) & & T
} \]
in which $f$ is type 2 admissible and $g$ is type 1 admissible 
is mapped to the functor

\vspace{0.2cm}
\fbox{\begin{minipage}{\textwidth}
\[  (\pi_{23}^*f)^\bullet (\pi_{23}^*g)_\bullet: \EE(I)_{S_1} \times \cdots \times \EE(I)_{S_n} \rightarrow \EE(I)_{T}.   \]
\end{minipage}}\vspace{0.2cm}

Note that, by Lemma~\ref{LEMMACOCARTPRES}, $(\pi_{23}^*g)_\bullet$ preserves the subcategory of $\pi_{12}$-Cartesian objects and, by Lemma~\ref{LEMMACARTPRES}, $(\pi_{23}^*f)^\bullet$  preserves the subcategory of $\pi_{13}$-coCartesian objects. 
In the oplax case, the condition on $f$ is repealed and the multicorrespondence is mapped to 
\[  \Box_* (\pi_{23}^*f)^\bullet (\pi_{23}^*g)_\bullet   \]
where $\Box_*$ is the {\em right coCartesian projection} defined and discussed in Section~\ref{COCARTPROJ}. 

A 2-morphism, given by a morphism of multicorrespondences
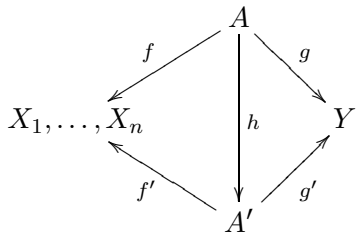
\[ \xymatrix{ 
& A \ar[dd]^h \\
X_1, \dots, X_n  \ar@{<-}[ur]^f \ar@{<-}[rd]_{f'}  &&Y  \ar@{<-}[ul]_{g} \ar@{<-}[ld]^{g'}  \\
& A' 
} \]
where $h$ is an isomorphism, is mapped to the natural transformation given by the unit 

\vspace{0.2cm}
\fbox{\begin{minipage}{\textwidth}
\[  (\pi_{23}^*f)^\bullet (\pi_{23}^*g)_\bullet \cong (\pi_{23}^*f')^\bullet (\pi_{23}^*h)^\bullet  (\pi_{23}^*h)_\bullet (\pi_{23}^*g')_\bullet 
\leftarrow (\pi_{23}^*f')^\bullet (\pi_{23}^*g')_\bullet.   \]
\end{minipage}}\vspace{0.2cm}

In the oplax case, $h$ can be an arbitrary morphism (which will be automatically type 1 admissible). The 2-morphism is then mapped to the natural transformation given by the unit 
\[   \Box_*  (\pi_{23}^*f)^\bullet  (\pi_{23}^*g)_\bullet \cong \Box_* (\pi_{23}^*f')^\bullet \Box_* (\pi_{23}^*h)^\bullet(\pi_{23}^*h)_\bullet (\pi_{23}^*g')_\bullet 
\leftarrow \Box_* (\pi_{23}^*f')^\bullet (\pi_{23}^*g')_\bullet.   \]

A functor $\alpha: I \rightarrow J$ is mapped to the functor 

\vspace{0.2cm}
\fbox{\begin{minipage}{\textwidth}
\[ (\twc{\alpha})^* \] 
\end{minipage}}\vspace{0.2cm}

which obviously preserves the  (co)Cartesianity conditions.
This is strictly compatible with composition of functors between diagrams. 
A natural transformation $\mu: \alpha \rightarrow \beta$ is mapped to the following natural transformation $(\twc{\alpha})^* \rightarrow (\twc{\beta})^*$:
We have the correspondence (cf.\@ \ref{ALTSCORPREDER}) 
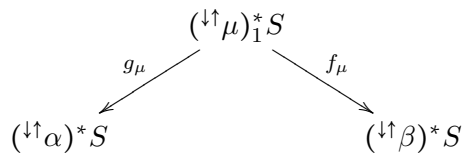
\[ \xymatrix{
& (\tw{\mu})_1^*S \ar[rd]^{f_\mu} \ar[ld]_{g_\mu} \\
(\tw{\alpha})^* S & & (\tw{\beta})^* S
} \]
where $f_\mu$ is type 2 admissible and $g_\mu$ is type 1 admissible by the definition of admissible diagram. 
On the other hand, there are natural transformations (cf.\@ \ref{PARTW3})
\[ \twc{\alpha} \Rightarrow (\twc{\mu})_1 \Leftarrow (\twc{\mu})_2 \Rightarrow \twc{\beta}. \]
Inserting $\pi_{23}^*(S^{\op})$ into this, we get 
\begin{equation}\label{seq1}
 \xymatrix{ \pi_{23}^* (\tw{\alpha})^*(S^{\op}) \ar@{->}[r]^-{\pi_{23}^* g_\mu} & \pi_{23}^*(\tw{\mu})_1^* (S^{\op}) \ar@{<-}[r]^-{\pi_{23}^* f_\mu} & \pi_{23}^* (\tw{\beta})^* (S^{\op}) \ar@{=}[r] &  \pi_{23}^*  (\tw{\beta})^* (S^{\op}).  } 
 \end{equation}

The natural transformation $\mu: \alpha \rightarrow \beta$ may be seen as a functor $\Delta_1 \times I \rightarrow J$ and therefore we get a functor
\[ \twc{\mu}: \twc{\Delta_1} \times \twc{I} \rightarrow \twc{J}. \]
Applying the (pre-)derivator $\DD$ and partially evaluating at the objects and morphisms of $\twc{\Delta_1}$ we get natural transformations
\begin{eqnarray*}
 (\pi_{23}^*g_\mu)_\bullet (\twc{\alpha})^* & \rightarrow & (\twc{\mu})_1^* \\
 (\twc{\mu})_2^* & \rightarrow & (\pi_{23}^*f_\mu)^\bullet (\twc{\mu})_1^*    \\
 (\twc{\mu})_2^* & \rightarrow & (\twc{\beta})^*   
\end{eqnarray*}
where the $(-)^*$-functors are now considered to be functors between the respective fibers over the objects of (\ref{seq1}). 
Clearly the first two morphisms (in particular the second) are isomorphisms when restricted to the respective categories of (co)Cartesian objects. 
Therefore we can form their composition:

\vspace{0.2cm}
\fbox{\begin{minipage}{\textwidth}
\[ (\pi_{23}^*f)^\bullet (\pi_{23}^*g)_\bullet (\twc{\alpha})^*  \rightarrow (\twc{\beta})^*  \]
\end{minipage}}\vspace{0.2cm}

which will be the image of $\mu$ under the pre-2-multiderivator $\EE$. 
One checks that for any admissible diagram $S \in \SSS(\tw{I})$, this defines a pseudo-functor from the category of functors $\Hom(I, J)$ to the 2-category of functors of the 2-category $\EE(I)$ to the 2-category $\EE(J)$, pseudo-natural transformations and modifications.
\end{DEF}

\begin{LEMMA}\label{LEMMACOCARTPRES}
Under the conditions of Theorem~\ref{MAINTHEOREM1},
let $S, T: \tw{I} \rightarrow \mathcal{S}$ be admissible diagrams and let
$f: S \rightarrow T$ be any morphism in $\SSS(\tw{I})$.
Then the functor
\[ (\pi_{23}^*f)^\bullet: \DD(I)_{\pi_{23}^*{T^{\op}}} \rightarrow \DD(I)_{\pi_{23}^*S^{\op}}   \]
maps always $\pi_{13}$-Cartesian objects to $\pi_{13}$-Cartesian objects, and maps
$\pi_{12}$-coCartesian objects to $\pi_{12}$-coCartesian if $f$ is type 2 admissible. 
\end{LEMMA}
\begin{proof}This follows immediately from base-change and from the definition of type 2 admissible. 
\end{proof}

\begin{LEMMA}\label{LEMMACARTPRES}
Under the conditions of Theorem~\ref{MAINTHEOREM1},
let $S_1, \dots, S_n, T: \tw{I} \rightarrow \mathcal{S}$ be admissible diagrams and let
$g: S_1, \dots, S_n \rightarrow T$ be any multimorphism in $\SSS(\tw{I})$.
Then the functor
\[ (\pi_{23}^*g)_\bullet: \DD(I)_{\pi_{23}^*{S_1^{\op}}} \times \cdots \times \DD(I)_{\pi_{23}^*{S_n^{\op}}} \rightarrow \DD(I)_{\pi_{23}^*T^{\op}}   \]
maps always $\pi_{12}$-coCartesian objects to $\pi_{12}$-coCartesian objects, and maps
$\pi_{13}$-Cartesian objects to $\pi_{13}$-Cartesian objects if $g$ is type 1 admissible. 
\end{LEMMA}
\begin{proof}This follows immediately from multi-base-change and from the definition of type 1 admissible. 
\end{proof}

\begin{PAR}
Recall that a diagram $I$ is called contractible, if  
\[ \id \Rightarrow p_{I,*} (p_I)^*, \]
or equivalently 
\[ p_{I,!} (p_I)^* \Rightarrow \id,  \]
is an isomorphism for all derivators. Cisinski showed that this is the case if and only if $N(I)$ is weakly contractible in the sense of simplicial sets. 
For instance, any diagram possessing a final or initial object is contractible. The following lemma was shown in \cite{Hor15} for the case of all contractible diagrams for a restricted class of stable derivators. 
We will only need the mentioned special case which is easy to prove in full generality:  
\end{PAR}

\begin{LEMMA}\label{LEMMAX}
If $\DD$ is a left  derivator and $I$ has a final object, or $\DD$ is a right derivator and $I$ has an initial object, then the functor
\[ p_I^*: \DD(\cdot) \rightarrow \DD(I)^{\cart}=\DD(I)^{\cocart} \]
is an equivalence.
\end{LEMMA}

Note that Cartesian=coCartesian here only means that all morphisms in the underlying diagram in $\Hom(I, \DD(\cdot))$ are isomorphisms. 

\begin{proof}
Assume we have a left derivator and $I$ has a final object (the other statement is dual). It suffices to show that
 the counit
\[   p_{I,!} p_I^*  \Rightarrow \id  \]
is an isomorphism and that the unit
\[ \id \Rightarrow p_I^* p_{I,!}  \]
is an isomorphism when restricted to the subcategory of Cartesian objects. 
Since $I$ has a final object $i$ we have an isomorphism
\[ p_{I,!} \cong i^* \]
and the unit and counit become the morphisms induced by the natural transformations $p_I \circ i = \id$ and $\id \Rightarrow i \circ p_I$.
Hence we have
\[   i^* p_I^*  = \id  \]
and the morphism
\[ \id  \Rightarrow  p_I^* i^*    \]
is an isomorphism on (co)Cartesian objects by definition of (co)Cartesian. 
\end{proof}

\begin{KOR}\label{LEMMAY}
If $\DD$ is a left and right derivator and $I$ has a final {\em or} initial object then
\[ p_I^*: \DD(\cdot) \rightarrow \DD(I)^{\cart}=\DD(I)^{\cocart} \]
is an equivalence. The restrictions of $p_{I,!}$, and $p_{I,*}$, coincide (up to isomorphism) and constitute an inverse equivalence.
\end{KOR}

\begin{LEMMA}\label{LEMMA2}
Under the assumptions of Theorem~\ref{MAINTHEOREM1},
if $\alpha: I \rightarrow J$ is an opfibration then
the functors
\[ \xymatrix{ \DD(\twc{J} \times_J I)_{\pi_{23}^*(S^{\op})}^{\pi_{12}-\cocart, \pi_{13}-\cart} \ar[r]^-{q_2^*} & \DD(\twc{J} \times_{(\tw{J})} \tw{I})^{\pi_{12}-\cocart, \pi_{13}-\cart}_{\pi_{23}^* (S^{\op})} 
\ar[r]^-{q_1^*} &  \DD(\twc{I})^{\pi_{12}-\cocart, \pi_{13}-\cart}_{\pi_{23}^*(S^{\op})}  } \]
are equivalences.
In particular (applying this to $J=\cdot$ and variable $I$) we have an equivalence of fibers:
\[  \EE_S \cong \DD_S. \]
\end{LEMMA}
\begin{proof}
We first treat the case of $q_1^*$.
We know by Lemma~\ref{LEMMAOPFIB} that $q_1$ is an opfibration with fibers of the form $i_3 \times_{/I_{j_3}} I_{j_3}$.
Neglecting the conditions of being (co)Cartesian, we know that $q_1^*$ has a left adjoint:
\[ q_{1,!}: \DD(\twc{I})_{\pi_{23}^*(S^{\op})}
\rightarrow \DD(\twc{J} \times_{(\tw{J})} \tw{I})_{\pi_{23}^* (S^{\op})} \]
We will show that the unit and counit
\[ \id \Rightarrow q_1^* q_{1,!}  \qquad  q_{1,!} q_1^*  \Rightarrow \id  \]
are isomorphisms {\em when restricted to the subcategory of $\pi_{12}$-coCartesian objects}. Since the conditions of being $\pi_{13}$-Cartesian match under $q_1^*$ this shows the first assertion. 
Since $q_1$ is an opfibration this is the same as to show that for any $\gamma \in \twc{J} \times_{(\tw{J})} \tw{I}$ with fiber $F = i_3 \times_{/I_{j_3}} I_{j_3}$ the unit and counit
\begin{equation}\label{eq7} \id \Rightarrow p_F^* p_{F,!}  \qquad  p_{F,!} p_F^*  \Rightarrow \id  \end{equation}
are isomorphisms when restricted to the subcategory of $\pi_{12}$-coCartesian objects. Since $\pi_{12}$ maps all morphisms in the fiber $F$ to an identity,  we have to show that the morphisms in (\ref{eq7}) are isomorphisms when restricted to (absolutely) (co)Cartesian objects. This follows from the fact that  $F$ has an initial object (Lemma~\ref{LEMMAX} and Corollary~\ref{LEMMAY}).

We now treat the case of $q_2^*$.
We know by Lemma~\ref{LEMMAOPFIB} that $q_2$ is a fibration with fibers of the form $(i_2 \times_{/I_{j_2}} I_{j_2})^{\op}$.
Neglecting the conditions of being (co)Cartesian, we know that $q_1^*$ has a right adjoint:
\[ q_{2,*}: \DD(\twc{J} \times_{(\tw{J})} \tw{I})_{\pi_{23}^* (S^{\op})}
\rightarrow \DD(\twc{J} \times_J I)_{\pi_{23}^*(S^{\op})} \]
We will show that the unit and counit
\[ \id \Rightarrow q_{2,*} q_2^*   \qquad  q_2^* q_{2,*}  \Rightarrow \id  \]
are isomorphisms {\em when restricted to the subcategory of $\pi_{13}$-Cartesian objects}. Since the conditions of being $\pi_{12}$-coCartesian match under $q_2^*$ this shows the second assertion. 
Since $q_2$ is a fibration this is the same as to show that for any $\gamma \in \twc{J} \times_J I$ with fiber $F=(i_2 \times_{/I_{j_2}} I_{j_2})^{\op}$ the the unit and counit
\begin{equation}\label{eq8} \id \Rightarrow  p_{F,*} p_F^*  \qquad  p_F^* p_{F,*}  \Rightarrow \id  \end{equation}
are isomorphisms when restricted to the subcategory of $\pi_{13}$-Cartesian objects. Since $\pi_{13}$ maps all morphisms in the fiber $(i_2 \times_{/I_{j_2}} I_{j_2})^{\op}$ to an identity, this means that we have to show that (\ref{eq8}) are isomorphisms when restricted to (absolutely) (co)Cartesian objects. This follows from the fact that  $(i_2 \times_{/I_{j_2}} I_{j_2})^{\op}$ has a final object (Lemma~\ref{LEMMAX} and Corollary~\ref{LEMMAY}). 
\end{proof}

\begin{LEMMA}~\label{LEMMA4}
Let the situation be as in Theorem~\ref{MAINTHEOREM1} and let $p': \EE \rightarrow \SSS^{\cor}$ be the morphism of pre-2-multiderivators defined in \ref{DEFE}. 
Let $\alpha: I \rightarrow J$ be an opfibration. Then $\alpha^*: \EE(J)_{\alpha^*S} \rightarrow \EE(I)_S$ has a left adjoint $\alpha_!^{(S)}$.
\end{LEMMA}
\begin{proof}
We have to show that
\[ (\twc{\alpha})^*: \DD(\twc{J})^{\pi_{12}-\cocart, \pi_{13}-\cart}_{\pi_{23}^* (S^{\op})} \rightarrow \DD(\twc{I})^{\pi_{12}-\cocart, \pi_{13}-\cart}_{\pi_{23}^* (S^{\op})} \]
has a left adjoint. The right hand side category is by Lemma~\ref{LEMMA2} equivalent to  \[ \DD((\twc{J}) \times_J I)^{\pi_{12}-\cocart, \pi_{13}-\cart}_{\pi_{23}^* S}, \]
hence we have to show that
\[ \pr_1^*: \DD(\twc{J})^{\pi_{12}-\cocart, \pi_{13}-\cart}_{\pi_{23}^* (S^{\op})} \rightarrow \DD((\twc{J}) \times_J I)^{\pi_{12}-\cocart, \pi_{13}-\cart}_{\pi_{23}^* (S^{\op})} \]
has a left adjoint. By assumption the functor
\[ \pr_1^*: \DD(\twc{J})_{\pi_{23}^*  (S^{\op})} \rightarrow \DD((\twc{J}) \times_J I)_{\pi_{23}^* (S^{\op})} \]
has a left adjoint $\pr_{2,!}$. We claim that it maps $\pi_{12}$-coCartesian objects to $\pi_{12}$-coCartesian objects and $\pi_{13}$-Cartesian objects to $\pi_{13}$-Cartesian objects. The statement then follows. 

Let $\kappa: \nu \rightarrow \nu'$
\[ \xymatrix{
j_1 \ar[r]^{\nu_1} \ar@{=}[d] & j_2 \ar[r]^{\nu_2} & \ar@{=}[d]  j_3  \\
j_1 \ar[r]_{\nu_1'} & j_2' \ar[u]_{\kappa_2} \ar[r]_{\nu_2'} & j_3
} \]
be a morphism in $\twc{J}$ such that $\pi_{13}$ maps it to an identity. Denote \[ f:=S(\pi_{23}(\kappa)): S(\pi_{23}(\nu)) \rightarrow S(\pi_{23}(\nu')) \] the corresponding morphism in $\SSS(\cdot)^{\op}$.
Denote by $(\nu)$, resp.\@ $(\nu')$ the inclusion of the one element category mapping to $\nu$, resp.\@ $\nu'$ in $\twc{I}$. 
We have to show that the induced map
\[ (\nu)^* \pr_{1,!} \rightarrow f^\bullet (\nu')^* \pr_{1,!} \]
is an isomorphism on $\pi_{13}$-Cartesian objects. Since $\pr_1$ is an opfibration, this is the same as to show that the natural morphism
\[  p_! \iota_{\nu}^* \rightarrow f^\bullet p_! \iota_{\nu'}^*  \]
is an isomorphism on $\pi_{13}$-Cartesian objects where $p: I_{j_1} \rightarrow \cdot$ is the projection. Since $f^\bullet$ commutes with homotopy colimits by assumption 1.\@ of Theorem~\ref{MAINTHEOREM1}, this is to say that
\[ p_! \iota_{\nu}^* \rightarrow p_!  (p^*f)^\bullet \iota_{\nu'}^*  \]
is an isomorphism. However the fibers over $\nu$ and $\nu'$ in $(\twc{J}) \times_J I$ are both equal to $I_{j_1}$ and the natural morphism
\[ \iota_{\nu}^* \rightarrow (p^*f)^\bullet \iota_{\nu'}^*  \]
is already an isomorphism on Cartesian objects by definition. 

Let $\kappa: \nu_1 \rightarrow \nu_2$
\[ \xymatrix{
	j_1 \ar[r]^{\nu_1} \ar@{=}[d] & j_2  \ar[r]^{\nu_2} & j_3 \ar[d] \\
j_1 \ar[r]_{\nu'_1} & j_2' \ar@{=}[u] \ar[r]_{\nu_2'} & j_3'
} \]

be a morphism in $\twc{J}$ such that $\pi_{12}$ maps it to an identity. And denote 
\[ g:=S(\pi_{23}(\kappa)): S(\pi_{23}(\nu)) \rightarrow S(\pi_{23}(\nu'))\]
 the corresponding morphism in $\SSS(\cdot)$.
Denote by $(\nu)$, resp.\@ $(\nu')$ the inclusion of the one element category mapping to $\nu$, resp.\@ $\nu'$. 
We have to show that the induced map
\[ g_\bullet (\nu)^* \pr_{1,!} \rightarrow  (\nu')^* \pr_{1,!} \]
is an isomorphism on $\pi_{12}$-coCartesian objects. This is the same as to show that the natural morphism
\[ g_\bullet p_! \iota_{\nu}^* \rightarrow  p_! \iota_{\nu'}^*  \]
is an isomorphism on $\pi_{12}$-coCartesian objects where $p: I_{j_1} \rightarrow \cdot$ is the projection. Since $g_\bullet$ commutes with homotopy colimits, this is to say that
\[ p_! (p^*g)_\bullet \iota_{\nu}^* \rightarrow p_!   \iota_{\nu'}^*  \]
is an isomorphism. However the fibers over $\nu$ and $\nu'$ in $(\twc{J}) \times_J I$ are both equal to $I_{j}$ and the natural morphism
\[ (p^*g)_\bullet  \iota_{\nu}^* \rightarrow \iota_{\nu'}^*  \]
is already an isomorphism on coCartesian objects by definition of coCartesian. 
\end{proof}

\begin{proof}[Proof of Theorem~\ref{MAINTHEOREM1}.] Using Theorem~\ref{MAINTHEOREMFIBDER2} we have to show the axioms (FDer0 left), (Der1), (Der2), (FDer3 left), (FDer4 left) and (FDer5 left) for the morphism of pre-2-multiderivators 
\[ \EE \rightarrow \SSS^{\cor, G, \oplax} \]
of Definition~\ref{DEFE}.  It is clear that  $\EE$ satisfies axioms (Der1) and (Der2) because $\DD$ satisfies them. Axiom (FDer0 left) holds by construction of $\EE$. 
Instead of Axiom (FDer3 left) it is sufficient to show Axiom (FDer3 left') which follows from Lemma~\ref{LEMMA4}. 
Axiom (FDer4 left') follows from the proof of Lemma~\ref{LEMMA4}. 
(FDer5 left) follows from the corresponding axiom for $\DD$ and the fact that pull-back along 1-ary morphisms in $\SSS^{\op}$ commutes with homotopy colimits as well, by assumption. 

Now assume that $\DD \rightarrow \SSS^{\op}$ has stable fibers. 
Let $I \in \Dia$ and $S \in \SSS^{\cor,G}(I)$. Then we have to show that 
\[ J \mapsto \EE(I \times J)_{\pr_I^* S}  \]
is a right derivator with domain $\Posf$. But for $J \rightarrow K$ we a commutative diagram in which the horizontal functors are equivalences: 
\[ \xymatrix{ \DD( (\twc{I}) \times J )^{\pi_{12}-\cocart, \pi_{13}-\cart}_{\pi_{23}^*(\pr_I^*S^{\op})} \ar[r]^-\sim\ar@{<-}[d]   & \DD( \twc{(I \times J) })^{\pi_{12}-\cocart, \pi_{13}-\cart}_{\pi_{23}^*(\pr_I^*S^{\op})} \ar@{<-}[d]  \\
\DD( (\twc{I}) \times K )^{\pi_{12}-\cocart, \pi_{13}-\cart}_{\pi_{23}^*(\pr_I^*S^{\op})} \ar[r]^-\sim &  \DD( \twc{(I \times K) })^{\pi_{12}-\cocart, \pi_{13}-\cart}_{\pi_{23}^*(\pr_I^*S^{\op})} 
 } \]
The left vertical functor has a right adjoint by assumption ignoring the (co)Cartesianity conditions. However, that functor preserves the conditions of being (co)Cartesian because all pull-back and push-forward functors are exact in this case and thus commute with all left and right homotopy Kan extensions with domain $\Posf$. 
\end{proof}

\begin{PAR}
Let $\alpha: K \rightarrow L$ be a functor in $\Dia$ and
let $\xi: (I_1, S_1), \dots, (I_n, S_n) \rightarrow (J, T)$ be a 1-morphism in $\Dia^{\cor}(\SSS^{\cor})$.
If we have a 1-opfibration and 2-opfibration 
\[ \Dia^{\cor}(\EE) \rightarrow \Dia^{\cor}(\SSS^{\cor}) \]
then the isomorphism of Lemma~\ref{LEMMAPROPDIACORS}, 3.\@ is transformed into an isomorphism
\[ (\alpha \times \id)^* \circ (\xi \times L)_\bullet \rightarrow  (\xi \times K)_\bullet \circ  ((\alpha \times \id)^*, \dots, (\alpha \times \id)^*) \]
which turns $K \mapsto (\xi \times K)_\bullet$ into a morphism of usual derivators  
\begin{equation}\label{eqpffib}
 \xi_\bullet:  \EE_{(I_1, S_1)} \times \cdots \times  \EE_{(I_n, S_n)} \rightarrow \EE_{(J, T)}  .
 \end{equation}
\end{PAR}

\begin{LEMMA} \label{LEMMACONT} The morphism of derivators 
(\ref{eqpffib}) is left exact in each variable, i.e.\@ the exchange
\[   (\xi \times_j L)_\bullet \circ_j (\alpha \times \id)_! \rightarrow (\alpha \times \id)_! \circ (\xi \times_j K)_\bullet   \]
is  an isomorphism for any $\alpha: K \rightarrow L$.
\end{LEMMA}
\begin{proof}
This follows from Lemma~\ref{LEMMAPROPDIACORS}, 4. 
\end{proof}

\begin{proof}[Proof of Theorem~\ref{MAINTHEOREM2}.]
The first assertion is a slight generalization of  \cite[Theorem 3.2.3 (left)]{Hor15}. 
Using Definition~\ref{DEFFIBDER2} of a left, resp.\@ right fibered multiderivator over pre-2-multiderivators we give a different slicker proof.
We have to show that, under the conditions of Theorem~\ref{MAINTHEOREM2}, the constructed 1-opfibration 
\[ \Dia^{\cor}(\EE) \rightarrow \Dia^{\cor}(\SSS^{\cor}) \]
is a 1-fibration as well. The conditions imply: 
\begin{enumerate}
\item $\Dia$ is infinite and $\EE$ and $\SSS^{\cor}$ are infinite,
\item the fibers of $\EE \rightarrow \SSS^{\cor}$  
are stable and perfectly generated infinite left derivators with domain $\Dia$, and also right derivators with domain (at least) $\Posf$.
\end{enumerate}
For 2.\@ note that by \cite[Lemma~4.7]{Hor15} it suffices to see the perfect generation for the categories $\EE(\cdot)_S$ which are the same as $\DD(\cdot)_S$. 
Any multimorphism in $(I_1, S_1), \dots, (I_n, S_n) \rightarrow (J, T)$ in $\Dia^{\cor}(\SSS^{\cor})$ gives a morphism between fibers 
\[ \EE_{I_1, S_1} \times \cdots \times  \EE_{I_n, S_n} \rightarrow \EE_{J, T}.  \]

Lemma~\ref{LEMMACONT} shows that this morphism commutes with homotopy colimits in each variable. Thus by  \cite[Theorem 3.2.1 (left)]{Hor15}
it has a right adjoint in each slot $j$, which, in particular, evaluated at $\cdot$ yields a right adjoint functor in the slot $j$:
\[ \EE(I_1)_{S_1}^{\op} \times \cdots \times \EE(J)_{T} \times \cdots \times \EE(I_n)_{S_n}^{\op} \rightarrow   \EE(I_j)_{S_j} \]
for each $j$. 
This establishes that the morphism
\[ \Dia^{\cor}(\EE) \rightarrow \Dia^{\cor}(\SSS^{\cor}) \]
is 1-fibered as well. 

The lax extension of this 1-fibration is given as follows. 
For each diagram $I$ we again specify a 1-fibered, and 2-opfibered multicategory with 1-categorical fibers  $\EE'(I) \rightarrow \SSS^{\cor, G,\lax}$.

The category 
\[ \EE'(I) \]
has the same objects as $\EE(I)$, i.e.\@ pairs $(S, \mathcal{E})$ consisting of an admissible diagram $S: \tw{I} \rightarrow \mathcal{S}$ and an object 
\[ \mathcal{E} \in \DD(\twc{I})_{\pi_{23}^* (S^{\op})}^{\pi_{12}-\cocart, \pi_{13}-\cart}. \]
The 1-morphisms are the morphisms in $\SSS^{\cor, G,\lax}(I)$, i.e.\@
lax morphisms, which can be given by a multicorrespondence 
\[ \xymatrix{
& A \ar[rd]^f \ar[ld]_g \\
(S_1, \dots, S_n) & & T
} \]
in which $f$ is type 2 admissible, and $g$ is arbitrary,
together with a morphism  
\[ \rho \in \Hom_{\EE(I)}\left((\mathcal{E}_1, S_1), \dots, (\mathcal{E}_n, S_n), (\mathcal{F}, T)\right) = \Hom_{\DD(\twc{I})_{\pi_{23}^*(T^{\op})}}\left((\pi_{23}^*f)^{\bullet }
(\pi_{23}^*g)_{\bullet} (\mathcal{E}_1, \dots, \mathcal{E}_n), \mathcal{F}\right).  \]

Note that the multivalued functor $(\pi_{23}^*g)_{\bullet}$ does not necessarily have values in the subcategory of $\pi_{13}$-Cartesian objects. 

A 2-morphism $(f, g, \rho) \Rightarrow (f', g', \rho')$ is given by a morphism of multicorrespondences
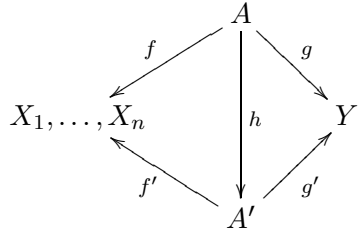
\[ \xymatrix{ 
& A \ar[dd]^h \\
X_1, \dots, X_n  \ar@{<-}[ur]^f \ar@{<-}[rd]_{f'}  &&Y  \ar@{<-}[ul]_{g} \ar@{<-}[ld]^{g'}  \\
& A' 
} \]
where $h$ is an arbitrary  morphism (which is automatically type 2 admissible, cf.\@ Lemma~\ref{LEMMATYPE12}) such that the diagram
\[ \xymatrix{
(\pi_{23}^*f)^{\bullet }
(\pi_{23}^*g)_{\bullet}(\mathcal{E}_1, \dots, \mathcal{E}_n) \ar[rrrd]^{\rho} \ar@{<-}[d]^\sim &\\
(\pi_{23}^*f')^{\bullet } (\pi_{23}^*h)^\bullet (\pi_{23}^*h)_\bullet (\pi_{23}^*g')_{\bullet}(\mathcal{E}_1, \dots, \mathcal{E}_n) \ar@{<-}[d] &&&  \mathcal{F} \\
(\pi_{23}^*f')^{\bullet } (\pi_{23}^*g')_{\bullet}(\mathcal{E}_1, \dots, \mathcal{E}_n) \ar[rrru]_{\rho'} 
} \]
commutes, where the lower left vertical morphism is the unit. 

A functor $\alpha: I \rightarrow J$ is mapped to the functor $(\twc{\alpha})^*$ which obviously preserves the (co)Cartesianity conditions.
Natural morphisms are treated in the same way as in the plain case because no lax morphisms are involved. 

We will now discuss the axioms:

(FDer0 right): It is clear from the definition that 
\[ \EE'(I) \rightarrow  \SSS^{\cor, G,\lax}(I) \]
is 2-opfibered and has 1-categorical fibers. It is also 1-fibered because we have
\begin{eqnarray*} 
& & \Hom_{\EE(I)}((\mathcal{E}_1, S_1), \dots, (\mathcal{E}_n, S_n), (\mathcal{F}, T))\\
& \cong & \Hom_{\DD(\twc{I})_{\pi_{23}^*T^{\op}}}((\pi_{23}^*f)^{\bullet } (\pi_{23}^*g)_{\bullet} (\mathcal{E}_1, \dots, \mathcal{E}_n), \mathcal{F}) \\
& \cong & \Hom_{\DD(\twc{I})_{\pi_{23}^*S_j^{\op}}}(\mathcal{E}_j, \Box_* (\pi_{23}^*g)^{\bullet,j} (\mathcal{E}_1, \overset{\widehat{j}}{\dots}, \mathcal{E}_n; (\pi_{23}^*f)_{?} \mathcal{F})).
\end{eqnarray*}
Here $\Box_*$ is the right coCartesian projection defined and discussed in Section~\ref{COCARTPROJ} and $(\pi_{23}^*f)_{?}$ is a right adjoint of $(\pi_{23}^*f)^\bullet$, which exists by the reasoning in the first part of the proof. (Note that $(\pi_{23}^*f)_{?}$ would be denoted $f^!$, i.e.\@ exceptional pull-back, in the usual language of six-functor-formalisms. Our notation, unfortunately, has reached its limit here.)
Therefore Cartesian morphisms exist w.r.t.\@ to any slot $j$ with pull-back functor explicitly given by 
\[ \Box_* (\pi_{23}^*g)^{\bullet,j} (-, \overset{\widehat{j}}{\dots}, -;  (\pi_{23}^*f)_{?} -). \] 

The second part of (FDer0 right) follows from the corresponding statement for $\DD$ and the fact that $\Box_*$ is ``point-wise the identity'' (cf.\@ Proposition~\ref{PROPCOCARTPROJ}). 
The axioms (FDer3--4 right) do not involve lax morphisms. 
(FDer5 right) follows because the corresponding axiom holds for $\DD$, because $(\pi_{23}^*f)_{?}$, as right adjoint, commutes with homotopy limits, and because $\Box_*$ is ``point-wise the identity'' (cf.\@ Proposition~\ref{PROPCOCARTPROJ}).  
\end{proof}

\section{Cocartesian projectors}\label{COCARTPROJ}

\begin{PAR}
We will show in this section that the fully-faithful inclusion
\[ \DD({}^{\downarrow\uparrow\downarrow} I)^{\pi_{13}-\cart,\pi_{12}-\cocart}_{\pi_{23}^* (S^{\op})}\ \rightarrow \DD({}^{\downarrow\uparrow\downarrow} I)^{\pi_{13}-\cart}_{\pi_{23}^* (S^{\op})}\]
(cf.\@ Definitions~\ref{DEFCOCART},~\ref{DEFE})
has a right adjoint $\Box_*$ which we will call a {\bf right coCartesian projector} (cf.\@ also \cite[Section~2.4]{Hor15}).

A right coCartesian projector (or rather its composition with the fully-faithful inclusion) can be specified by an endofunctor $\Box_*$ of $\DD({}^{\downarrow\uparrow\downarrow} I)^{\pi_{13}-\cart}_{\pi_{23}^*(S^{\op})}$ together with a natural transformation
\[ \nu: \Box_* \Rightarrow \id  \]
such that 
\begin{enumerate}
\item $\Box_* \mathcal{E}$ is $\pi_{12}$-coCartesian for all objects $\mathcal{E}$,
\item  $\nu_{\mathcal{E}}$ is an isomorphism on $\pi_{12}$-coCartesian objects $\mathcal{E}$,
\item $ \nu_{\Box_*\mathcal{E}} = \Box_* \nu_{\mathcal{E}}$ holds true. 
\end{enumerate}

This, in particular, gives a pullback functor
\[  \Box_* f^\bullet: \DD({}^{\downarrow\uparrow\downarrow} I)^{\pi_{13}-\cart,\pi_{12}-\cocart}_{\pi_{23}^* (S^{\op})}\ \rightarrow \DD({}^{\downarrow\uparrow\downarrow} I)^{\pi_{13}-\cart,\pi_{12}-\cocart}_{\pi_{23}^* (T^{\op})}\]
for {\em any} morphism (not necessarily type 2 admissible)
\[ f: S \rightarrow T\]
of admissible diagrams in $\mathcal{S}(\tw{I})$.

Note that, of course, $f^\bullet$ preserves automatically the condition of being $\pi_{13}$-Cartesian. Proposition~\ref{PROPCOCARTPROJ} below shows that this is still computed point-wise, i.e. that we have for any $\alpha: I \rightarrow J$
\[ \alpha^* \Box_* f^\bullet \cong  \Box_* (\alpha^*f)^\bullet.  \]
\end{PAR}

\begin{PAR}\label{COCARTPROJPREP}
We need some technical preparation. Consider the projections:
\[ \pi_{123}, \pi_{125}, \pi_{145} \pi_{345}: {}^{\downarrow\uparrow\downarrow\uparrow\downarrow}I \rightarrow {}^{\downarrow\uparrow\downarrow} I  \]
We have obvious natural transformations
\[ \pi_{123} \Rightarrow \pi_{125} \Leftarrow \pi_{145} \Rightarrow \pi_{345} \]
and therefore
\[ \pi_{123}^* \Rightarrow \pi_{125}^* \Leftarrow \pi_{145}^* \Rightarrow \pi_{345}^* \]
If we plug in $\pi_{23}^*(S^{\op})$ for an admissible diagram $S \in \SSS({}^{\downarrow\uparrow} I)$, we get morphisms of diagrams in $\SSS^{\op}$:
\[ \xymatrix{  \pi_{23}^*(S^{\op}) \ar[r]^-{g} & \pi_{25}^*(S^{\op}) \ar@{<-}[r]^-{f} & \pi_{45}^*(S^{\op}) \ar@{=}[r] &  \pi_{45}^*(S^{\op})  } \]
and therefore natural transformations
\begin{eqnarray*}
 g_\bullet \pi_{123}^* &\Rightarrow&  \pi_{125}^*  \\
 f^\bullet \pi_{125}^* &\Leftarrow&  \pi_{145}^*  
\end{eqnarray*}
of functors between fibers.
\end{PAR}

\begin{LEMMA}
$\pi_{123}$ and $\pi_{345}$ are opfibrations. 
\end{LEMMA}
\begin{proof}This was explained in \ref{PARTW2}.
\end{proof}

\begin{LEMMA}
The natural transformation
\[ \pi_{345,!}^{(\pi_{45}^*S)} \pi_{145}^* \Rightarrow \id \]
induced by the natural transformation
\[ \pi_{145}^* \Rightarrow \pi_{345}^* \]
of functors 
\[  \pi_{145}^* ,  \pi_{345}^*:  \DD({}^{\downarrow\uparrow\downarrow\uparrow\downarrow}I)_{\pi_{45}^*(S^{\op})} \rightarrow \DD(\twc{I})_{\pi_{23}^*(S^{\op})} \]
is an isomorphism.
\end{LEMMA}
\begin{proof}
Since $\pi_{345}$ is an opfibration, we have for any object $\alpha = \{i \rightarrow j \rightarrow k\}$ in $\twc{I}$: 
\[ \alpha^* \pi_{345,!} \pi_{145}^* = p_! \pi_{145}^*  \]
where $p: {}^{\downarrow\uparrow} (I \times_{/I} i) \rightarrow \{\cdot\}$. We can factor $p$ in the following way:
\[ \xymatrix{ {}^{\downarrow\uparrow} (I \times_{/I} i) \ar[r]^-{\pi_1} & I \times_{/I} i \ar[r]^P \ar[r] & \{\cdot\} }\]
The functor $\pi_1$ is an opfibration with fibers of the form $\beta \times_{/(I \times_{/I} i)}(I \times_{/I} i)$. 
Since these fibers have an initial object, and the objects in the image of $\pi_{145}^*$ are constant along it, the homotopy colimit 
over objects in the image of $\pi_{145}^*$ along it are equal to this constant value by Corollary~\ref{LEMMAY}. Furthermore, the homotopy colimit over $I \times_{/I} i$ is
the same as evaluation at $\id_i$ because $\id_i$ is the final object. 
\end{proof}

If $\mathcal{E}$ is an object in $\DD({}^{\downarrow\uparrow\downarrow} I)^{\pi_{13}-\cart}_{\pi_{23}^*(S^{\op})}$ we have that the morphism
\[  f^\bullet \pi_{125}^* \mathcal{E} \leftarrow  \pi_{145}^* \mathcal{E}  \]
is an isomorphism.

\begin{PROP}\label{PROPCOCARTPROJ}Using the notation of \ref{COCARTPROJPREP}, 
denote $\Box_* :=  \pi_{345;!} f^\bullet g_\bullet \pi_{123}^*$. This functor, together
with the composition
\[ \xymatrix{  \mathcal{E} & \ar[l]_-\sim  \pi_{345;!} \pi_{145}^* \mathcal{E} \ar[r]^-\sim &  \pi_{345;!} f^\bullet \pi_{125}^* \mathcal{E}  & \ar[l]  \pi_{345;!} f^\bullet g_\bullet \pi_{123}^* \mathcal{E}  =  \Box_* \mathcal{E} \ar@/^20pt/[lll]^{\nu_{\mathcal{E}}}   },  \]
defines a right coCartesian projector:
\[\Box_* :  \DD({}^{\downarrow\uparrow\downarrow} I)^{\pi_{13}-\cart}_{\pi_{23}^* S} \rightarrow \DD({}^{\downarrow\uparrow\downarrow} I)^{\pi_{13}-\cart,\pi_{12}-\cocart}_{\pi_{23}^* S}. \]

This projector has the following property: 
\begin{itemize}
\item For each $i \in I$ the natural transformation
\[ ({}^{\downarrow \uparrow \downarrow}i)^* \Box_*  \rightarrow  ({}^{\downarrow \uparrow \downarrow}i)^*   \]
is an isomorphism. (Here $i$ denotes, by abuse of notation, the subcategory of $I$ consisting of $i$ and $\id_i$. Hence ${}^{\downarrow \uparrow \downarrow}i$ is the subcategory of
${}^{\downarrow \uparrow \downarrow}I$ consisting of $i=i=i$ and its identity.)
\end{itemize}
\end{PROP}
\begin{proof}
If suffices to show that 
\begin{enumerate}
\item $\Box_* \mathcal{E}$ is $\pi_{12}$-coCartesian for all objects $\mathcal{E}$,
\item  $\nu_{\mathcal{E}}$ is an isomorphism on $\pi_{12}$-coCartesian objects $\mathcal{E}$,
\item the equation $ \Box_* \nu_{\mathcal{E}} = \nu_{\Box_* \mathcal{E}}$ holds true. 
\end{enumerate}
1.\@ Since $\pi_{345}$ is an opfibration, the evaluation of $\Box_* \mathcal{E}$ at an object $i \rightarrow j \rightarrow k$ of $\twc{I}$ is equal to 
\[ p_{{}^{\downarrow\uparrow} (I \times_{/I} i), !} \iota_{i,j,k}^* f^\bullet g_\bullet \pi_{123}^* \mathcal{E}  \]
where $\iota_{i,j,k}: {}^{\downarrow\uparrow} (I \times_{/I} i) \hookrightarrow {}^{\downarrow\uparrow\downarrow\uparrow\downarrow}I$ is the inclusion
\[ (l \rightarrow m \rightarrow i) \mapsto (l \rightarrow m \rightarrow i \rightarrow j \rightarrow k). \]
For any morphism $\mu$ in ${}^{\downarrow\uparrow\downarrow}I$ such that $\pi_{12}$ (resp.\@ $\pi_{13}$) maps it to an an identity we have to see that the morphism
\[ \mu^* (\Box_* \mathcal{E}) \]
in $\DD(\cdot)$ is coCartesian (resp.\@ Cartesian). In the first case, such a morphism $\mu$ is of the form
\[ \xymatrix{
i \ar[r] \ar@{=}[d] & j \ar[r] \ar@{=}[d] & k \ar[d] \\
i \ar[r] &i \ar[r] & k'
} \]
and since homotopy colimits commute with push-forward it suffices to show that all morphisms
\[ (S^{\op}(\pr_{23}(\mu)))_\bullet \iota_{i,j,k'}^* f^\bullet g_\bullet \pi_{123}^* \mathcal{E} \rightarrow \iota_{i,j,k}^* f^\bullet g_\bullet \pi_{123}^* \mathcal{E} \]
are isomorphisms. This follows immediately from the base-change formula \ref{DEFMULTIBASECHANGE}. 

In the second case, such a morphism $\mu$ is of the form
\[ \xymatrix{
i \ar[r] \ar@{=}[d] & j' \ar[r] \ar@{<-}[d] & k \ar@{=}[d] \\
i \ar[r] & j \ar[r] & k
} \]
and since by assumption homotopy colimits commute with pull-backs as well, it suffices to show that all morphisms
\[ (S^{\op}(\pr_{23}(\mu)))^\bullet \iota_{i,j',k}^* f^\bullet g_\bullet \pi_{123}^* \mathcal{E} \rightarrow \iota_{i,j,k}^* f^\bullet g_\bullet \pi_{123}^* \mathcal{E} \]
are isomorphisms which is obvious. 

Assertion 2.\@ follows from the fact that for a $\pi_{13}$-coCartesian and $\pi_{12}$-Cartesian diagram $\mathcal{E}$ the diagram
\[  \iota_{i,j,k}^* f^\bullet g_\bullet \pi_{123}^* \mathcal{E}  \]
is (co)Cartesian over the constant diagram $S(j \rightarrow k)$. Therefore (as in the proof of Lemma~\ref{LEMMA2}) its homotopy colimit is the same as evaluation 
 at $(i = i) \in {}^{\downarrow\uparrow} (I \times_{/I} i)$ which is not affected by $f^\bullet g_\bullet$ and $(i = i)$ is mapped by $\pi_{123} \circ \iota_{i,j,k}$ to $i \rightarrow j \rightarrow k$.

We give a sketch of proof of assertion 3.\@ for which we need a bit of preparation.
Note that the following diagrams are Cartesian:
\[ \xymatrix{
 {}^{\downarrow\uparrow\downarrow\uparrow\downarrow\uparrow\downarrow} I  \ar[rr]^{\pi_{12345}} \ar[d]_{\pi_{34567}} &&  {}^{\downarrow\uparrow\downarrow\uparrow\downarrow} I \ar[d]^{\pi_{345}} \\
 {}^{\downarrow\uparrow\downarrow\uparrow\downarrow} I  \ar[rr]^{\pi_{123}} &&  {}^{\downarrow\uparrow\downarrow} I 
}
\qquad
\xymatrix{
 {}^{\downarrow\uparrow\downarrow\uparrow\downarrow\uparrow\downarrow} I  \ar[rr]^{\pi_{12347}} \ar[d]_{\pi_{34567}} &&  {}^{\downarrow\uparrow\downarrow\uparrow\downarrow} I \ar[d]^{\pi_{345}} \\
 {}^{\downarrow\uparrow\downarrow\uparrow\downarrow} I  \ar[rr]^{\pi_{125}} &&  {}^{\downarrow\uparrow\downarrow} I 
} 
\qquad
\xymatrix{
 {}^{\downarrow\uparrow\downarrow\uparrow\downarrow\uparrow\downarrow} I  \ar[rr]^{\pi_{12367}} \ar[d]_{\pi_{34567}} &&  {}^{\downarrow\uparrow\downarrow\uparrow\downarrow} I \ar[d]^{\pi_{345}} \\
 {}^{\downarrow\uparrow\downarrow\uparrow\downarrow} I  \ar[rr]^{\pi_{145}} &&  {}^{\downarrow\uparrow\downarrow} I 
} \]
Consider the following commutative diagram in $\SSS({}^{\downarrow\uparrow\downarrow\uparrow\downarrow\uparrow\downarrow} I)$ in which the square is Cartesian:
\[ \xymatrix{
 &  & \pi_{67}^*(S^{\op}) \ar[d]^{f_2}  \ar@/^30pt/[dd]^{f_3}\\
 & \pi_{45}^*(S^{\op}) \ar[r]^{g_4} \ar[d]^{f_5} & \pi_{47}^*(S^{\op}) \ar[d]^{f_1} \\
 \pi_{23}^*(S^{\op})  \ar@/_30pt/[rr]_{g_2} \ar[r]_{g_3} & \pi_{25}^*(S^{\op}) \ar[r]_{g_5} &  \pi_{27}^*(S^{\op})
} \]
We  have an isomorphism
\[ \xymatrix{ (\Box_*)^2 = \pi_{345,!} f^\bullet g_\bullet \pi_{123}^* \ \pi_{345,!} f^\bullet g_\bullet \pi_{123}^* & \ar[l]_-{\sim}^-\Theta \pi_{567,!} f_3^\bullet g_{2,\bullet} \pi_{123}^* } \]
given as the following composition
\begin{eqnarray*}
\pi_{345,!} f^\bullet g_\bullet \pi_{123}^*\ \pi_{345,!} f^\bullet g_\bullet \pi_{123}^* & \overset{\sim}{\longleftarrow}  & \pi_{345,!} f^\bullet g_\bullet \pi_{34567,!} \pi_{12345}^* f^\bullet g_\bullet \pi_{123}^* \\
& \overset{\sim}{\longleftarrow} & \pi_{345,!}  \pi_{34567,!} f^{\bullet}_2 g_{4,\bullet}   f_5^\bullet g_{3,\bullet}  \pi_{12345}^* \pi_{123}^* \\
& \overset{\sim}{\longleftarrow}  & \pi_{567,!}  f_3^\bullet g_{2,\bullet}   \pi_{123}^*
\end{eqnarray*}
involving that $()_\bullet$ and $()^\bullet$ commute with both $\alpha^*$ and $\alpha_!$ for any functor $\alpha$ in $\Dia$, and the base-change formula for $f_1, f_5, g_4, g_5$.
For the first isomorphism note that $\pi_{345}$ is an opfibration. 

Similarly, we construct the other horizontal isomorphisms in the following diagram
\begin{equation} \label{hugedia1} \vcenter{ \xymatrix{
\pi_{345,!} f^\bullet g_\bullet \pi_{123}^* \ar@{=}[r] \ar@{<-}[d]^{\sim} & \pi_{345,!} f^\bullet g_\bullet \pi_{123}^* \ar@{<-}[d]^{\numcirc{1}} \\
\pi_{345;!} \pi_{145}^*\   \pi_{345,!} f^\bullet g_\bullet \pi_{123}^* \ar@{<-}[r]^-\sim \ar[d]^{\sim} & \pi_{567,!} f^\bullet_3 g_{2,\bullet} \pi_{123}^* \ar@{=}[d] \\
\pi_{345;!} f^\bullet \pi_{125}^* \  \pi_{345,!} f^\bullet g_\bullet \pi_{123}^* \ar@{<-}[r]^-\sim \ar@{<-}[d]^{\sim} & \pi_{567,!} f_3^\bullet g_{2,\bullet} \pi_{123}^* \ar@{=}[d] \\
\pi_{345,!} f^\bullet g_\bullet \pi_{123}^* \ \pi_{345,!} f^\bullet g_\bullet \pi_{123}^* \ar@{<-}[r]^-\sim_-\Theta \ar@/^100pt/[uuu]^{\nu \ast \Box_!} & \pi_{567,!} f_3^\bullet g_{2,\bullet} \pi_{123}^* 
} }
\end{equation}
The morphism $\numcirc{1}$ is the following composition, in which the first morphism is induced by the counit of the pair of adjoint functors $\pi_{12567,!}$, $\pi_{12567}^*$:
\begin{eqnarray*}
\pi_{345,!} f^\bullet g_\bullet \pi_{123}^* & \longleftarrow  & \pi_{345,!} f^\bullet g_\bullet \pi_{12567,!} \pi_{12567}^*  \pi_{123}^* \\
& \overset{\sim}{\longleftarrow} &  \pi_{567,!} \pi_{12567,!} f_3^{\bullet} g_{5,\bullet} \pi_{125}^* \\
& \longleftarrow  & \pi_{567,!}  f_3^\bullet g_{5,\bullet} g_{3,\bullet}   \pi_{123}^* \\
& \overset{\sim}{\longleftarrow}  & \pi_{567,!}  f_3^\bullet g_{2,\bullet}  \pi_{123}^*.
\end{eqnarray*}
One checks that the diagram (\ref{hugedia1}) commutes. 

There is an analogous commutative diagram
\begin{equation} \label{hugedia2} \vcenter{ \xymatrix{
\pi_{345,!} f^\bullet g_\bullet \pi_{123}^* \ar@{=}[r] \ar@{<-}[d]^{\sim} & \pi_{345,!} f^\bullet g_\bullet \pi_{123}^* \ar@{<-}[d]^{}_{\numcirc{2}} \\
  \pi_{345,!} f^\bullet g_\bullet \pi_{123}^* \ \pi_{345;!} \pi_{145}^* \ar@{<-}[r]^-\sim \ar[d]^{\sim} & \pi_{567,!} f^\bullet_2 g_{4,\bullet} \pi_{145}^* \ar[d]^\sim_{\numcirc{3}} \\
\pi_{345,!} f^\bullet g_\bullet \pi_{123}^* \ \pi_{345;!} f^\bullet \pi_{125}^*  \ar@{<-}[r]^-\sim \ar@{<-}[d]^{\sim} & \pi_{567,!} f_3^\bullet g_{5,\bullet} \pi_{125}^* \ar@{<-}[d]^\sim_{\numcirc{4}} \\
\pi_{345,!} f^\bullet g_\bullet \pi_{123}^* \ \pi_{345,!} f^\bullet g_\bullet \pi_{123}^* \ar@{<-}[r]^-\sim_-\Theta \ar@/^100pt/[uuu]^{\Box_! \ast \nu} & \pi_{567,!} f_3^\bullet g_{2,\bullet} \pi_{123}^* 
} }
\end{equation}
in which the morphism $\numcirc{2}$ is constructed similarly using the counit of the pair of adjoint functors $\pi_{14567,!}$, $\pi_{14567}^*$, and 
$\numcirc{3}$ and $\numcirc{4}$ are constructed as in \ref{COCARTPROJPREP}.

Furthermore, there is a morphism $\numcirc{5}$ constructed similarly using the counit of the pair of adjoint functors $\pi_{12567,!}$, $\pi_{12567}^*$ again, making the diagram 
\begin{equation} \label{hugedia2} \vcenter{ \xymatrix{
& \pi_{567,!}f_3^\bullet g_{5,\bullet} \pi_{125}^* \ar[dd]^{\numcirc{5}}  \\
\pi_{567,!}f_3^\bullet g_{2,\bullet} \pi_{123}^* \ar[ru]^{\numcirc{4}} \ar[rd]_{\numcirc{1}} & &  \pi_{567,!}f_2^\bullet g_{4,\bullet} \pi_{145}^* \ar[lu]_{\numcirc{3}} \ar[ld]^{\numcirc{2}}\\
& \pi_{345,!}f^\bullet g_{\bullet} \pi_{123}^*
} }
\end{equation}
commutative. This proves assertion 3.

For the additional property given in the statement of the Proposition observe that 
$\pi_{345,!}\mathcal{E}$ at an arrow $i \rightarrow i \rightarrow i$ is the homotopy colimit over the diagram $\iota^* \mathcal{E}$ for $\iota: {}^{\downarrow\uparrow} (I \times_{/I} i) \hookrightarrow {}^{\downarrow\uparrow\downarrow\uparrow\downarrow}I$ pulled back to $S(i = i)$. The projection $\pr_1: {}^{\downarrow\uparrow} (I \times_{/I} i) \rightarrow  (I \times_{/I} i) $ is an opfibration with fibers of the form $\beta \times_{/(I \times_{/I} i)} (I \times_{/I} i)$. These categories have an initial object and the restriction of the diagram $\pi_{123}^* \mathcal{E}$  is constant on it, because of the assumption that $\mathcal{E}$ is $\pi_{13}$-Cartesian already. Hence the homotopy colimit over the restriction of $\pi_{123}^* \mathcal{E}$ to these fibers is the corresponding constant value by Lemma~\ref{LEMMAX}.
The colimit over $(I \times_{/I} i)$, furthermore, is evaluation at $\id_i$ because it is a final object. In total, the natural morphism 
\[ ({}^{\downarrow \uparrow \downarrow}i)^* \Box_* \mathcal{E} \rightarrow  ({}^{\downarrow \uparrow \downarrow}i)^* \mathcal{E}  \]
is an isomophism. 
\end{proof}

\section{The (co)localization property and $n$-angels in the fibers of a stable proper or etale derivator six-functor-formalism}\label{SECTIONLOC}

\begin{PAR}
Let $\mathcal{S}$ be a category and $\mathcal{S}_0$ a class of ``proper'' morphisms. 
Let 
\[ \DD' \rightarrow \SSS^{\cor,0,\oplax} \quad \text{resp.}\quad  \DD'' \rightarrow \SSS^{\cor,0,\lax} \]
be a proper derivator six-functor-formalism (cf.\@ Definition~\ref{DEF6FUDER}) with stable fibers (cf.\@ Definition~\ref{DEFSTABLE}). The multi-aspect will not play any role in this section. The reasoning in this section has an ``etale'' analogue that we leave to the reader to state. 
\end{PAR}

\begin{PAR}\label{PAROPENCLOSED}
If $\mathcal{S}$ is a category of some kind of spaces, we are often given a class of elementary squares as follows.
Assume that in $\mathcal{S}$ there are certain distinguished morphisms called ``closed immersions'' or ``open immersions'' respectively, with an
operation of taking complements. 
For a morphism $f: X \rightarrow Y$ in $\mathcal{S}$ we denote by $f$, resp.\@ $f^{\op}$ the correspondences
\[ f: \vcenter{ \xymatrix{
  & S \ar[rd]^{f} \ar@{=}[ld] \\
 S & & T
} } \qquad 
f^{\op}: \vcenter{\xymatrix{
  & S \ar[ld]_{f} \ar@{=}[rd] \\
 T & & S
}}  \]
in $\mathcal{S}^{\cor}$. 
Let 
\[ \xymatrix{ U \ar@{^{(}->}[r]^i & V \ar@{^{(}->}[r]^j & X  } \]
be a sequence of ``open embeddings''. And let 
$\overline{i}: V \setminus U \hookrightarrow V$, resp.\@ $\overline{j \circ  i}: X \setminus U \hookrightarrow X$ be ``closed embeddings of the complements''. 
For now these morphisms can be arbitrary, but to make sense of these definitions in applications they should satisfy the properties of \ref{PARCOMPLEMENT} below. 

We then have the following square in $\Xi_{U,V,X} \in \SSS^{\cor}(\Box)$: 
\[ \xymatrix{
V \ar[r]^{j} \ar[d]^{\overline{i}^{\op}} & X \ar[d]^{\overline{j \circ  i}^{\op}} \\
V \setminus U \ar[r]^j & X \setminus U
} \]
Assume that the ``closed embeddings'' lie in the class $\mathcal{S}_0$ which was fixed to define the notion of proper derivator six-functor-formalism. 
Then the above square comes equipped with a morphism $\xi: \Xi_{U,V,X} \rightarrow p^*X$ in $\SSS^{\cor, 0, \oplax}(\Box)$ represented by the cube (as a morphism
from the front face to the back face):
\[ \xymatrix{
& X \ar@{=}[rr] \ar@{=}[dd] & &X  \ar@{=}[dd] \\
V \ar[ru]^{j} \ar[dd]_-{\overline{i}^{\op}} \ar[rr]^-(.3){j} && X \ar[dd]^-(.3){\overline{j \circ  i}^{\op}} \ar@{=}[ru]   \\
& X \ar@{=}[rr]^-(.3){} & & X \\
V \setminus U \ar[ru]^{j \circ \overline{i}} \ar[rr]^-{j} && X \setminus U \ar[ru]_{\overline{j \circ  i}}
} \]
The top and bottom squares are 2-commutative, whereas the left and right squares are only 
oplax 2-commutative, e.g.\@ there is a 2-morphism making the diagram
\[ \xymatrix{
X \ar[r]^{\overline{j \circ  i}^{\op}} \ar@{=}[d] \ar@{}[rd]|{\Swarrow} & X \setminus U \ar[d]^{\overline{j \circ  i}} \\
X \ar@{=}[r] & X
} \]
commutative, which is given by the morphism of correspondences
\[ \xymatrix{ 
&X \setminus U \ar[dl]_{j \circ  i} \ar[dd]^{j \circ  i} \ar[dr]^{j \circ  i} \\ 
X  && X \\
&X  \ar@{=}[ul]  \ar@{=}[ur] 
} \]
From now on, we forget about the provenance of these squares and just consider a proper derivator six-functor-formalism (more precisely, its oplax left fibered derivator)
\[ \DD \rightarrow \SSS^{\cor,0,\oplax} \]
with a class of distinguised squares  $\Xi \in \SSS^{\cor}(\Box)$ with given morphisms $\xi: \Xi\rightarrow p^*X$ in $\SSS^{\cor, 0, \oplax}(\Box)$.
\end{PAR}

\begin{DEF} \label{PARCOMPLEMENT} 
Let  $\SSS$  be a pre-2-derivator with all 2-morphisms invertible. We call a square $\Xi \in \SSS(\Box)$ {\bf Cartesian}, if the natural functor
\[ \Hom(X, (0,0)^*\Xi) \rightarrow \Hom(p^*X, i_{\righthalfcup}^* \Xi)  \]
is an equivalence of groupoids for all $X \in \SSS(\cdot)$, and {\bf coCartesian} if the natural functor
\[ \Hom((1,1)^*\Xi, X) \rightarrow \Hom(i_{\lefthalfcap}^* \Xi, p^*X)  \]
is an equivalence of groupoids for all $X \in \SSS(\cdot)$.
We call a square $\Xi \in \SSS(\Box)$ {\bf biCartesian} if it is Cartesian and coCartesian. 
\end{DEF}

\begin{BEM}
If $\SSS$ is a usual derivator then this notion coincides with the usual notion \cite{Gro13}.
\end{BEM}

\begin{PAR}
One can show that the squares $\Xi_{U,V,X} \in \SSS^{\cor}(\Box)$ constructed in the last paragraph are actually Cartesian in $\SSS^{\cor}$ 
provided that for all pairs $U, X \setminus U$ of ``open and closed embeddings'' used above we have
 \[ \Hom_{\mathcal{S}}(A, U) = \{ \alpha \in \Hom_{\mathcal{S}}(A, X) \ | \ A \times_{\alpha, X} (X \setminus U) = \emptyset \} \]
and coCartesian provided that we have
 \[ \Hom_{\mathcal{S}}(A, X \setminus U) = \{ \alpha \in \Hom_{\mathcal{S}}(A, X) \ | \ A \times_{\alpha,X} U = \emptyset \} \]
where $\emptyset$ is the initial object.
\end{PAR}

\begin{PAR}
There is a dual variant of the previous construction (not to be confused with the transition to an etale six-functor-formalism).
We consider instead the square $\Xi'_{U,V,X}$ with morphism 
\[ \xymatrix{
& X \ar[dl]_{\overline{j \circ i}^{\op}} \ar@{=}[rr] \ar@{=}[dd] & &X  \ar@{=}[dd] \ar@{=}[dl]  \\
X \setminus U  \ar[dd]_-{i^{\op}} \ar[rr]^-(.3){\overline{j\circ i}} && X \ar[dd]^-(.3){j^{\op}}   \\
& X \ar[dl]^{(  j\circ \overline{i} )^{\op}} \ar@{=}[rr]^-(.3){} & & X \ar[dl]^{j^{\op}} \\
 V \setminus U  \ar[rr]_-{\overline{i}} &&  V 
} \]
In this case the top and bottom squares are only {\em lax} 2-commutative, e.g.\@ there is a 2-morphism making the diagram
\[ \xymatrix{
X \ar@{=}[r] \ar[d]_{\overline{j \circ i}^{\op}} \ar@{}[rd]|{\Nearrow}& X \ar@{=}[d]  \\
X \setminus U \ar[r]^{\overline{j \circ i}}   & X 
} \]
2-commutative. 
This means that for a stable proper derivator six-functor-formalism it is also reasonable to consider a class of distinguished squares with given morphisms $\xi': p^*X \rightarrow \Xi$ in 
$\SSS^{\cor, 0, \lax}(\Box)$.
The morphism $p^*X \rightarrow \Xi'_{U,V,X}$ is just the {\em dual} of the morphism $\Xi_{U,V,X} \rightarrow p^*X$ for the absolute duality on $\Dia^{\cor}(\SSS^{\cor})$ (cf. \ref{DUALITY}).
\end{PAR}

Let $i_\lefthalfcap : \lefthalfcap \hookrightarrow \Box$ and $i_\righthalfcup : \righthalfcup \hookrightarrow \Box$ be the inclusions. 
Analogously to the situation for stable derivators \cite[4.1]{Gro13} we define:

\begin{DEF}
A square $\mathcal{E} \in \DD(\Box)$ over $\Xi \in \SSS(\Box)$ is called  {\bf relatively coCartesian}, if 
for the inclusion  $i_\lefthalfcap: (\lefthalfcap, i_\lefthalfcap^*\Xi) \rightarrow  (\Box, \Xi)$ the unit
$\mathcal{E} \rightarrow i_{\lefthalfcap,*}i_\lefthalfcap^* \mathcal{E}$ is an isomorphism, and it is called  {\bf relatively Cartesian} if
for the inclusion $i_\righthalfcup: (\righthalfcup, i_\righthalfcup^*\Xi) \rightarrow (\Box, \Xi)$ the counit
$i_{\righthalfcup,!}i_\righthalfcup^* \mathcal{E} \rightarrow \mathcal{E}$ is an isomorphism\footnote{The functors $i_{\righthalfcup,!}$ and $i_{\lefthalfcap,*}$  are in both cases considered w.r.t.\@ the base $\Xi$. }. $\mathcal{E}$ is called  {\bf relatively biCartesian} if it is relatively Cartesian and relatively coCartesian. 
\end{DEF}
If $\Xi$ is itself (co)Cartesian in the sense of Definition~\ref{PARCOMPLEMENT} then $\mathcal{E}$ relatively (co)Cartesian implies (co)Cartesian in the sense of Definition~\ref{PARCOMPLEMENT}.

\begin{LEMMA}Assume $V = U$ and
let $\DD(\Box)_{\Xi_{U,U,X}}^{\mathrm{bicart}}$ be the full subcategory of relatively biCartesian squares. Let $(1,0): (\cdot, X) \rightarrow (\Box, \Xi_{U,U,X})$ be the inclusion. 
Then the functor
\[ (1,0)^*: \DD(\Box)_{\Xi_{U,U,X}}^{\mathrm{bicart}} \rightarrow \DD(\cdot)_X \]
and the composition
\[ \xymatrix{  \DD(\cdot)_X \ar[r]^-{1_*} & \DD(\rightarrow)_{U \rightarrow X} \ar[r]^-{0_!} &  \DD(\Box)_{\Xi_{U,U,X}}^{\mathrm{bicart}}  } \]
define an equivalence of categories. 

Also if $V \not= U$ the functor $1_*0_!$ takes values in relatively biCartesian squares.
\end{LEMMA}

Recall that the functor
\[ 0^*: \DD(\Box)_{p^*X}^{\mathrm{bicart},0} \rightarrow \DD(\rightarrow)_{p^*X} \]
is an equivalence, 
where $\DD(\Box)_{p^*X}^{\mathrm{bicart},0}$ is the full subcategory of (relatively) biCartesian objects whose $(1,0)$-entry is zero.
(This is a statement about usual derivators.)

\begin{DEF}
We say that a distinguished square $\Xi$ together with $\xi: \Xi \rightarrow p^*X$ is a {\bf localizing square} if the push-forward $\xi_\bullet$ maps relatively biCartesian squares to relatively biCartesian squares.
We say that a distinguished square $\Xi$ together with $\xi: p^*X \rightarrow \Xi $ is a {\bf colocalizing square} if the pull-back $\xi^\bullet$ maps relatively biCartesian squares to relatively biCartesian squares.
\end{DEF}

If every object in $\DD(\cdot)$ is dualizable w.r.t.\@ the absolute monoidal product in $\Dia^{\cor}(\DD)$ then $\xi: \Xi \rightarrow p^*X$ is localizing if and only if 
 $\xi^\vee:  p^*X \rightarrow \Xi^\vee$ is colocalizing. 

\begin{BEM}
If the proper derivator six-functor-formalism with its oplax extension
\[ \DD \rightarrow \SSS^{\cor, 0, \oplax} \]
has {\em stable} fibers, and the square $\Xi_{U,V,X}$ constructed above is distinguished, then the property of being a localizing square implies that for $\mathcal{E} \in \DD(\cdot)_X$ the triangle
\[ \xymatrix{ j_! j^! \mathcal{E} \ar[r] & (j\circ \overline{i})_! \overline{i}^* j^! \mathcal{E}  \oplus  \mathcal{E} \ar[r] & \overline{j \circ i}_* \overline{j \circ i}^* \mathcal{E} \ar[r]^-{[1]} &   } \]
is distinguished. If $U=V$ this is just the sequence
\[ \xymatrix{ j_! j^! \mathcal{E} \ar[r] &  \mathcal{E} \ar[r] & \overline{j}_* \overline{j}^* \mathcal{E} \ar[r]^-{[1]} &   } \]
\end{BEM}
\begin{BEM}
If the proper derivator six-functor-formalism with its lax extension
\[ \DD \rightarrow \SSS^{\cor, 0, \lax} \]
has {\em stable} fibers, and the square $\Xi_{U,V,X}$ constructed above is distinguished, then the property of being a colocalizing square implies that for $\mathcal{E} \in \DD(\cdot)_X$ the triangle
\[ \xymatrix{ \overline{j \circ i}_! \overline{j \circ i}^! \mathcal{E} \ar[r] & (j \circ \overline{i})_* i^* (\overline{j \circ i})^! \mathcal{E}  \oplus  \mathcal{E} \ar[r] & j_* j^* \mathcal{E} \ar[r]^-{[1]} &   } \]
is distinguished. If $U=V$ this is just the sequence:
\[ \xymatrix{ \overline{j}_! \overline{j}^! \mathcal{E} \ar[r] &  \mathcal{E} \ar[r] & j_* j^* \mathcal{E} \ar[r]^-{[1]} &   } \]
\end{BEM}

\begin{DEF}
We say that the proper derivator six-functor-formalism with its extension as oplax left fibered derivator 
\[ \DD \rightarrow \SSS^{\cor, 0, \oplax} \]
satisfies {\bf the localization property} w.r.t.\@ a class of
distinguished squares 
$\xi: \Xi \rightarrow p^*X$
if these are localizing squares. 

We say that the proper derivator six-functor-formalism with its extension as lax right fibered derivator 
\[ \DD \rightarrow \SSS^{\cor, 0, \lax} \]
satisfies {\bf the colocalization property} w.r.t.\@ a class of
distinguished squares 
$\xi: p^*X \rightarrow \Xi$
if these are colocalizing squares. 
\end{DEF}

There is an analogous notion in which an {\em etale} derivator-six-functor-formalism w.r.t.\@ a class of ``etale morphisms'' $\mathcal{S}_0$ in $\mathcal{S}$ satisfies the (co)localization property

\begin{PAR}
Consider again the situation in \ref{PAROPENCLOSED}. 
More generally we may consider a sequence
\[ \xymatrix{ X_1 \ar@{^{(}->}[r] & X_2 \ar@{^{(}->}[r] & \cdots \ar@{^{(}->}[r] & X_n  } \]
of open embeddings. They lead to a diagram $\Xi$
\[ \xymatrix{ 
X_1 \ar[r] \ar[d] & X_2 \ar[r] \ar[d]  & \cdots \ar[r]   & X_n   \ar[d] \\
\emptyset \ar[r]  & X_2 \setminus X_1 \ar[r] \ar[d]  & \cdots \ar[r]   & X_n \setminus X_1  \ar[d]   \\
& \emptyset  \ar[r] & \ddots  \ar[d] & \vdots   \ar[d]   \\
&& \emptyset  \ar[r] & X_n \setminus X_{n-1} \\
} \]
in which all squares are biCartesian in $\SSS^{\cor}$. 
Starting from an object $\mathcal{E} \in \DD(\cdot)_{X_n}$ we may form again
\[ 0_* (n)_! \mathcal{E} \]
 where $(n): \cdot \rightarrow [n]$ is the inclusion of the last object and 
$0:   [n] \rightarrow \Xi$ is the inclusion of the first line. It is easy to see that in the object 
$0_* (n)_! \mathcal{E}$ all squares are biCartesian. There is furthermore again a morphism $\xi: \Xi \rightarrow p^*{X_n}$ in $\SSS^{\cor,0,\oplax}$ 
such that all squares in $\xi_\bullet 0_* (n)_! \mathcal{E}$ are biCartesian with zero's along the diagonal. This category is
equivalent to $\DD([n])_{p^*X_n}$ by the embedding of the first line. It can be seen as a category of $n$-angels in the stable derivator
$\DD_{X_n}$ (the fiber of $\DD$ over $X$).

Hence for an oplax derivator six-functor-formalism with localization property, and 
for any filtration of a space $X$ by $n$ open subspaces, and for any object $\mathcal{E} \in \DD(\cdot)_X$ we get a corresponding $(n+1)$-angle in the derivator $\DD_X$ in the sense of \cite[\S 13]{GS14b}.
\end{PAR}

\appendix

\section{Representable pre-2-multiderivators}\label{APPENDIX}

Let $p: \mathcal{D} \rightarrow \mathcal{S}$ be a strict functor of 2-(multi)categories. 
Recall the definition of their associated represented pre-2-(multi)derivators $\DD, \DD^{\lax}, \DD^{\oplax}$ etc.\@ from \ref{PARREPR2PREMULTIDER}.
In this appendix we investigate to which extend 2-categorical fibration properties for $p$ are inherited by the morphism between their associated pre-2-(multi)derivators.

\begin{PROP}\label{PROPREPR}
\begin{enumerate}
\item If $\mathcal{D} \rightarrow \mathcal{S}$ is a 1-fibration (resp.\@ 1-opfibration, resp.\@ 2-fibration, resp.\@ 2-opfibration) of
2-categories then 
$\DD(I) \rightarrow \SSS(I)$ is a 1-fibration (resp.\@ 1-opfibration, resp.\@ 2-fibration, resp.\@ 2-opfibration) of 2-categories.

\item If $\mathcal{D} \rightarrow \mathcal{S}$ is a 1-fibration and 2-opfibration of 2-categories then
$\DD^{\mathrm{lax}}(I) \rightarrow \SSS^{\mathrm{lax}}(I)$ is a 1-fibration and 2-opfibration of 2-categories. 
 If $\mathcal{D} \rightarrow \mathcal{S}$ is a 1-opfibration and 2-fibration of 2-multicategories then
$\DD^{\mathrm{lax}}(I) \rightarrow \SSS^{\mathrm{lax}}(I)$ is a 1-opfibration and 2-fibration of 2-multicategories.  

If $\mathcal{D} \rightarrow \mathcal{S}$ is a 1-fibration and 2-fibration of 2-categories then
$\DD^{\mathrm{oplax}}(I) \rightarrow \SSS^{\mathrm{oplax}}(I)$ is a 1-fibration and 2-fibration of 2-categories.  
If $\mathcal{D} \rightarrow \mathcal{S}$ is a 1-opfibration and 2-opfibration of 2-multicategories then
$\DD^{\mathrm{oplax}}(I) \rightarrow \SSS^{\mathrm{oplax}}(I)$ is a 1-opfibration and 2-opfibration of 2-multicategories.

\item If $\mathcal{D} \rightarrow \mathcal{S}$ is a 1-bifibration and 2-isofibration of
2-multicategories with {\em complete} 1-categorical fibers then 
$\DD(I) \rightarrow \SSS(I)$ is a 1-bifibration and 2-isofibration of
2-multicategories.
\end{enumerate}
\end{PROP}

\begin{proof}[Proof (sketch).]
We show exemplarily 1.\@ for 1-fibrations of 2-categories and 2.\@ for 1-fibrations of 2-categories. The same proof works for 1-opfibrations (even of 2-multicategories). 
If we have a 1-bifibration of 2-multicategories {\em with 1-categorical fibers} then a slight extension of the proof of \cite[Proposition 4.1.6]{Hor15} shows 3.
For 1-opfibrations and 2-fibrations of 2-multicategories with 1-categorical fibers this is actually easier to prove using the encoding by a pseudo-functor as follows:
The 1-opfibration $\mathcal{D} \rightarrow \mathcal{S}$ with 1-categorical fibers is encoded in a pseudo-functor 
\[ \mathcal{S} \rightarrow \mathcal{MCAT}\]
The category $\DD(I) \rightarrow \SSS(I)$
is encoded in the pseudo-functor
\[ \SSS(I) \rightarrow \mathcal{MCAT}\]
which maps a pseudo-functor $F: I \rightarrow \mathcal{S}$ to the multicategory of natural transformations and modifications
\[ \Hom_{\Fun(I, \mathcal{MCAT})}( \cdot, F),  \]
where $\cdot$ is the constant functor with value the 1-point category. 

This construction may be adapted to 2-categorical fibers by using ``pseudo-functors'' of 3-categories. 

The problem with {\em 1-fibrations} of (1- or 2-){\em multi}\,categories comes from the fact that the internal Hom cannot be computed point-wise but involves a
limit construction (cf.\@ \cite[Proposition 4.1.6]{Hor15}). The difference between external and internal monoidal product in $\Dia^{\cor}(\DD)$ gives a
theoretical explanation of this phenomenon (cf.\@ \cite[Example~7.5]{Hor15b}). 

{\em 1-fibration of 2-categories $\mathcal{D} \rightarrow \mathcal{S}$ implies 1-fibration of 2-categories $\DD(I) \rightarrow \SSS(I)$: }

Let $I$ be a diagram in $\Dia$.
Let $Y, Z: I \rightarrow \mathcal{S}$ be pseudo-functors, $f: Y \Rightarrow Z$ be a pseudo-natural transformation and 
$\mathcal{E}: I \rightarrow \mathcal{D}$ be a pseudo-functor over $Z$. 
For each morphism $\alpha: i \rightarrow i'$ we are given a 2-commutative diagram
\[ \xymatrix{
Y(i) \ar[r]^{f(i)} \ar[d]_{Y(\alpha)} \ar@{}[rd]|{\Nearrow^{f_\alpha}} & Z(i) \ar[d]^{Z(\alpha)} \\
Y(i') \ar[r]^{f(i')} & Z(i')
} \]
Since $f$ is assumed to be a pseudo-functor, the morphism $f_\alpha$ is invertible. We will construct a pseudo-functor $\mathcal{G}: I \rightarrow \mathcal{D}$ over $Y$ and a 1-coCartesian morphism 
$\xi: \mathcal{G} \rightarrow \mathcal{E}$ over $f$. 
For each $i$, we choose a 1-coCartesian morphism
\[ \xi(i): \mathcal{G}(i) \rightarrow \mathcal{E}(i) \]
over $f(i): Y(i) \rightarrow Z(i)$. For each $\alpha$, we look at the 2-Cartesian diagram
\[ \xymatrix{
\Hom_{\mathcal{D}}(\mathcal{G}(i), \mathcal{G}(i')) \ar[r]^{\xi(i') \circ} \ar[d] & \Hom_{\mathcal{D}}(\mathcal{G}(i), \mathcal{E}(i')) \ar[d] \\
\Hom_{\mathcal{S}}(Y(i), Y(i')) \ar[r]^{f(i') \circ} & \Hom_{\mathcal{S}}(Y(i), Z(i'))
} \]
The triple $(\mathcal{E}_\alpha \circ \xi(i), f_\alpha, Y_\alpha)$ is an object in the category 

\[ \Hom_{\mathcal{D}}(\mathcal{G}(i), \mathcal{E}(i')) \times_{/ \Hom_{\mathcal{S}}(Y(i), Z(i'))}^\sim \Hom_{\mathcal{S}}(Y(i), Y(i')) \]
Define $\mathcal{G}(\alpha)$ to be an object in $\Hom_{\mathcal{D}}(\mathcal{G}(i), \mathcal{G}(i'))$ such that there exists a 2-isomorphism
\[ \Xi_\alpha: (\xi(i') \circ \mathcal{G}_\alpha, \id, p(\mathcal{G}_\alpha)) \Rightarrow (\mathcal{E}_\alpha \circ \xi(i), f_\alpha^{-1}, Y_\alpha).  \]
Such an object exists because the above square is 2-Cartesian.

We get a 2-commutative square
\[ \xymatrix{
\mathcal{G}(i) \ar[r]^{\xi(i)} \ar[d]_{\mathcal{G}_\alpha} \ar@{}[rd]|{\Nearrow^{\xi_\alpha}} & \mathcal{E}(i) \ar[d]^{\mathcal{E}_\alpha} \\
\mathcal{G}(i') \ar[r]^{\xi(i')} & \mathcal{E}(i') 
} \]
Here $\xi_\alpha$ is the first component of  $\Xi_\alpha$. 

This defines a pseudo-functor $\mathcal{G}: I \rightarrow \mathcal{D}$ as follows.
Let $\alpha: i \rightarrow i'$ and $\beta: i' \rightarrow i''$ be two morphisms in $I$. 
We need to define a 2-isomorphism
$G_{\beta \alpha} \Rightarrow G_{\beta} \circ G_{\alpha}$. 
It suffices to define the 2-isomorphism after applying the embedding
\[  \Hom_{\mathcal{D}}(\mathcal{G}(i), \mathcal{G}(i'')) \hookrightarrow 
\Hom_{\mathcal{D}}(\mathcal{G}(i), \mathcal{E}(i'')) \times_{/ \Hom_{\mathcal{S}}(Y(i), Z(i''))}^\sim \Hom_{\mathcal{S}}(Y(i), Y(i''))  \]
which maps $G_{\beta} \circ G_{\alpha}$ to

\[ (\xi(i'') \circ \mathcal{G}_\beta \circ \mathcal{G}_\alpha,  \id, p(\mathcal{G}_\beta) \circ p(\mathcal{G}_\alpha)) \]
and  $G_{\beta \alpha}$ to
\[ (\xi(i'') \circ \mathcal{G}_{\beta \alpha},  \id, p(\mathcal{G}_{\beta \alpha})). \] 

We have the chains of 2-isomorphisms
\[ \xymatrix{ 
\xi(i'') \circ \mathcal{G}_\beta \circ \mathcal{G}_\alpha \ar[d]^{\xi_\beta \ast \mathcal{G}_\alpha} \\
 \mathcal{E}_\beta \circ \xi(i') \circ \mathcal{G}_\alpha \ar[d]^{\mathcal{E}_\beta \ast \xi_\alpha} \\
 \mathcal{E}_\beta \circ  \mathcal{E}_\alpha \circ \xi(i) \ar[d] \\
 \mathcal{E}_{\beta \alpha} \circ \xi(i) \\
\xi(i'') \circ \mathcal{G}_{\beta \alpha} \ar[u]
} \qquad  \xymatrix{ 
p(\mathcal{G}_\beta) \circ p(\mathcal{G}_\alpha) \ar[d]^{\Xi_{\beta,2} \ast p(\mathcal{G}_\alpha)} \\
Y_\beta \circ p(\mathcal{G}_\alpha) \ar[d]^{Y_\beta \ast \Xi_{\alpha,2}} \\
Y_\beta \circ Y_\alpha \ar[d] \\
Y_{\beta \alpha} \\
p(\mathcal{G}_{\beta \alpha})  \ar[u] 
} \]

Applying $p$ to the first chain and $f(i'')\circ$ to the second chain, we get the commutative diagram
\[ \xymatrix{
& f(i'') \circ p(\mathcal{G}_\beta) \circ p(\mathcal{G}_\alpha) \ar@{=}[r] \ar[d] & f(i'') \circ p(\mathcal{G}_\beta) \circ p(\mathcal{G}_\alpha) \ar[d] \\
 & \ar[ld] \ar[d] Z_\beta \circ  f(i') \circ p(\mathcal{G}_\alpha) \ar[r] & f(i'') \circ Y_\beta \circ p(\mathcal{G}_\alpha) \ar[d] \\
 Z_\beta \circ  Z_\alpha \circ f(i) \ar[r] \ar[d] &  Z_\beta \circ  f(i') \circ Y_\alpha \ar[r]  & f(i'') \circ Y_\beta \circ Y_\alpha  \ar[d] \\
Z_{\beta \alpha} \circ f(i)  \ar[rr]^{f_{\beta\alpha}} && f(i'') \circ Y_{\beta \alpha} \\
f(i'') \circ p(\mathcal{G}_{\beta \alpha}) \ar@{=}[rr] \ar[u] &&  f(i'') \circ p(\mathcal{G}_{\beta \alpha}) \ar[u] \\
} \]
hence a valid isomorphism in $\Hom_{\mathcal{D}}(\mathcal{G}(i), \mathcal{E}(i'')) \times_{/ \Hom_{\mathcal{S}}(Y(i), Z(i''))}^\sim \Hom_{\mathcal{S}}(Y(i), Y(i''))$. One checks that this satisfies the axioms of a pseudo-functor \cite[Definition~1.3]{Hor15b} and that $\xi$ is indeed a pseudo-natural transformation (\cite[Definition~1.4]{Hor15b}). 
 
Now assume that we have a lax natural transfomation, i.e.\@ the $f_\alpha$ go into the opposite direction and are no longer invertible. 
We assume that we have a 2-fibration as well. Then the diagram
\[ \xymatrix{
\Hom_{\mathcal{D}}(\mathcal{G}(i), \mathcal{G}(i')) \ar[r]^{\xi(i') \circ} \ar[d] & \Hom_{\mathcal{D}}(\mathcal{G}(i), \mathcal{E}(i')) \ar[d] \\
\Hom_{\mathcal{S}}(Y(i), Y(i')) \ar[r]^{f(i') \circ} & \Hom_{\mathcal{S}}(Y(i), Z(i'))
} \]
is Cartesian as well. Moreover we have an adjunction with the full comma category
\[ \xymatrix{ \Hom_{\mathcal{D}}(\mathcal{G}(i), \mathcal{G}(i')) \ar@/_3pt/[rrr]_-{\mathrm{can}}  &&&  \ar@/_3pt/[lll]_-{\rho_{\xi(i'), \mathcal{G}(i)}} \Hom_{\mathcal{D}}(\mathcal{G}(i), \mathcal{E}(i')) \times_{/ \Hom_{\mathcal{S}}(Y(i), Z(i'))} \Hom_{\mathcal{S}}(Y(i), Y(i')) } \]
with $\rho \circ \mathrm{can} = \id$, in particular with the morphism `$\mathrm{can}$' fully faithful. See below for the precise definition of $\rho$. 
Hence we define
\[ \mathcal{G}(\alpha) := \rho_{\xi(i'), \mathcal{G}(i)} (\mathcal{E}_\alpha \circ \xi(i), f_\alpha, Y_\alpha) \]
and get at least a morphism (coming from the unit of the adjunction): 
\[ \Xi_\alpha: (\xi(i') \circ \mathcal{G}(\alpha), \id, p(\mathcal{G}(\alpha))) \Rightarrow (\mathcal{E}_\alpha \circ \xi(i), f_\alpha, Y_\alpha).  \]
The first component of $\Xi_\alpha$ this time (potentially) define a lax-natural transformation $\xi: \mathcal{G} \rightarrow \mathcal{E}$ only. 
To turn $\mathcal{G}$ into a pseudo-functor, we have to see that $\rho$ is functorial.

For a Cartesian arrow $\xi: \mathcal{E} \rightarrow \mathcal{F}$, we define 
\[ \rho_{\xi, \mathcal{G}}:  \Hom_{\mathcal{D}}(\mathcal{G}, \mathcal{F}) \times_{/\Hom_{\mathcal{S}}(U, T)} \Hom_{\mathcal{S}}(U,S) \rightarrow \Hom_{\mathcal{D}}(\mathcal{G}, \mathcal{E}) \] 
 as follows:
Let $(\tau, \mu, g)$ be a tuple with $g \in \Hom_{\mathcal{S}}(U,S)$, $\tau \in  \Hom_{\mathcal{D}}(\mathcal{G}, \mathcal{F})$ and
\[ \mu: f \circ g \Rightarrow p(\tau) \]
a 2-morphism. We may choose a coCartesian 2-morphism
\[ \widetilde{\mu}: X \Rightarrow \tau \]
above $\mu$. We set $\rho_\xi(\tau, \mu, g)$ equal to an object with an isomorphism
\[  (\xi \circ \rho_\xi(\tau, \mu, g), \id, p(\rho_\xi(\tau, \mu, g))) \iso (X, \id, g).  \]
Together with the morphism
\[ \widetilde{\mu}:  (X, \id, g) \longrightarrow (\tau, \mu, g)  \]
we get the counit
\[ \mathrm{can} \circ \rho_\xi \Rightarrow \id.  \]

We need to define a 2-isomorphism
$G_{\beta \alpha} \Rightarrow G_{\beta} \circ G_\alpha$. 
i.e.
\[ \rho_{\xi(i''), \mathcal{G}(i)} (\mathcal{E}_{\beta \alpha} \circ \xi(i), f_{\beta \alpha}, Y_{\beta \alpha}) \rightarrow  \]
\[ \rho_{\xi(i''), \mathcal{G}(i')} (\mathcal{E}_{\beta} \circ \xi(i'), f_\beta, Y_\beta) \circ \rho_{\xi(i'), \mathcal{G}(i)} (\mathcal{E}_\alpha \circ \xi(i), f_\alpha, Y_\alpha)\]

First of all, we get three Cartesian 2-morphisms 
\[ \begin{array}{rllll}
 \widetilde{f_{\beta\alpha}}:&  X_{\beta \alpha} &\Rightarrow& \mathcal{E}_{\beta \alpha} \circ \xi(i) & \text{ over } f_{\beta\alpha} \\
 \widetilde{f_{\alpha}}: & X_{\alpha} &\Rightarrow& \mathcal{E}_{\alpha} \circ \xi(i) & \text{ over } f_\alpha \\
 \widetilde{f_{\beta}}: & X_{\beta} &\Rightarrow& \mathcal{E}_{\beta} \circ \xi(i') &  \text{ over } f_\beta 
\end{array} \]
and have to define an isomorphism (after applying can) 
\[  (X_{\beta \alpha}, \id, Y(\beta \alpha)) \iso   (X_\beta, \id, Y(\beta)) \circ (X_\alpha, \id, Y(\alpha)).\]

We have the diagram
\begin{equation}\label{eqpast1} \vcenter{ \xymatrix{ 
\mathcal{G}(i) \ar[rr]^{f(i)} \ar[rdrd]^{X_\alpha} \ar[dd]_{\mathcal{G}(\alpha)} && \mathcal{E}(i) \ar[dd]^{\mathcal{E}(\alpha)} \ar@/^35pt/[dddd]^{\mathcal{E}(\beta\alpha)}_{\Nearrow^\sim} \\
&\ar@{}[ru]|{\Nearrow^{\mu_\alpha}} & \\
\mathcal{G}(i') \ar[rr]^{f(i')} \ar[rdrd]^{X_\beta} \ar[dd]_{\mathcal{G}(\beta)} \ar@{}[ur]|{\Nearrow^{\sim}} && \mathcal{E}(i') \ar[dd]^{\mathcal{E}(\beta)} \\
&\ar@{}[ru]|{\Nearrow^{\mu_\beta}} & \\
\mathcal{G}(i'') \ar[rr]^{f(i'')} \ar@{}[ur]|{\Nearrow^{\sim}} && \mathcal{E}(i'')
}} \end{equation}
and the diagram
\begin{equation}\label{eqpast2} \vcenter{  \xymatrix{
\mathcal{G}(i) \ar[rr]^{f(i)} \ar[rdrd]^{X_{\beta\alpha}}  \ar[dd]_{\mathcal{G}(\beta \alpha)} && \mathcal{E}(i) \ar[dd]^{\mathcal{E}(\beta \alpha)} \\
&\ar@{}[ru]|{\Nearrow^{\mu_{\beta\alpha}}} & \\
\mathcal{G}(i'') \ar[rr]^{f(i'')} \ar@{}[ur]|{\Nearrow^{\sim}} && \mathcal{E}(i'')
}} \end{equation}
The two pastings are both Cartesian (using Lemma~\ref{LEMMAWHISKERINGCART} below) over the pastings in the diagram
\[
\vcenter{ \xymatrix{ 
Y(i) \ar[r]^{f(i)}  \ar[d]_{Y(\alpha)} & Z(i) \ar[d]^{Z(\alpha)} \ar@/^45pt/[dd]^{Z(\beta\alpha)}_{\Nearrow^{Z_{\beta,\alpha}}} \\
Y(i') \ar[r]^{f(i')} \ar[d]_{Y(\beta)} \ar@{}[ur]|{\Nearrow^{f_\alpha}} & Z(i') \ar[d]^{Z(\beta)} \\
Y(i'') \ar[r]^{f(i'')} \ar@{}[ur]|{\Nearrow^{f_\beta}} & Z(i'')
} } \quad \text{resp.} \quad  \vcenter{ \xymatrix{
Y(i) \ar[r]^{f(i)} \ar@/_45pt/[dd]_{Y(\beta\alpha)}^{\Nearrow^{Y_{\beta, \alpha}}} \ar[d]_{Y(\alpha)} & Z(i) \ar[d]^{Z(\alpha)} \ar@/^45pt/[dd]^{Z(\beta\alpha)}_{\Nearrow^{Z_{\beta,\alpha}}} \\
Y(i') \ar[r]^{f(i')} \ar[d]_{Y(\beta)} \ar@{}[ur]|{\Nearrow^{f_\alpha}} & Z(i') \ar[d]^{Z(\beta)} \\
Y(i'') \ar[r]^{f(i'')} \ar@{}[ur]|{\Nearrow^{f_\beta}} & Z(i'')
} } \]
The pasting in the second diagram is just $f_{\beta,\alpha}$ by definition of lax natural transformation for $f$. 
This yields an isomorphism between the pastings in diagram (\ref{eqpast1}) and (\ref{eqpast2})  {\em over $Y_{\beta, \alpha}$} which we define to be $\mathcal{G}_{\beta, \alpha}$.
One checks that this defines indeed a pseudo-functor $\mathcal{G}$ such that $\xi: \mathcal{G} \rightarrow \mathcal{E}$ is a lax natural transformation which is 1-Cartesian.
\end{proof}

\begin{LEMMA}\label{LEMMAWHISKERINGCART}
Let $\mathcal{D} \rightarrow \mathcal{S}$ be a 2-(op)fibration of 2-categories. 
Let $\mu: \alpha \Rightarrow \beta$ be a 2-(co)Cartesian morphism, where $\alpha, \beta: \mathcal{E} \rightarrow \mathcal{F}$ are 1-morphisms.
If $\gamma: \mathcal{F} \rightarrow \mathcal{G}$ is a 1-morphism then $\gamma \ast \mu$ is 2-(co)Cartesian. 
Similarly, if $\gamma': \mathcal{G} \rightarrow \mathcal{E}$ is a 1-morphism then $\mu \ast \gamma'$ is 2-(co)Cartesian. 
\end{LEMMA}
\begin{proof}
This follows immediately from the axiom that composition is a morphism of (op-)fibrations. 
\end{proof}

\newpage
\bibliographystyle{abbrvnat}
\bibliography{6fu}

\end{document}